\documentclass[12pt]{article}
\RequirePackage{amsthm, amsmath, amsfonts, amssymb, graphicx}
\usepackage{dsfont}

\usepackage{float}
\usepackage[shortlabels]{enumitem}
\usepackage{booktabs}
\newcommand{\bneqn}{\vspace{-0.25cm}\begin{eqnarray}}
\newcommand{\eneqn}{\end{eqnarray}}
\allowdisplaybreaks

\usepackage[utf8]{inputenc}
\usepackage[toc,page,title,titletoc]{appendix}
\usepackage{hyperref,bookmark}
\hypersetup{
  colorlinks=true,
  linkcolor=red,
  citecolor=blue,
  filecolor=blue,
  urlcolor=blue,
}

\newtheorem{theorem}{Theorem}
\newtheorem{lemma}[theorem]{Lemma}
\newtheorem{remark}[theorem]{Remark}
\newtheorem{proposition}[theorem]{Proposition}
\newtheorem{corollary}[theorem]{Corollary}

\advance \topmargin by -\headheight
\advance \topmargin by -\headsep
\advance \topmargin .2in
\title{Exact and asymptotic distribution theory for the empirical correlation of two AR(1) processes with Gaussian increments}
\author{Philip A. Ernst\footnote{Department of Mathematics, Imperial College London, London SW7 2AZ, UK}\,
 and Dongzhou Huang\footnote{Department of Statistics, Colorado State University, Fort Collins, CO 80523, USA}}
\begin{document}
\maketitle

\begin{abstract}
This paper begins with a study of the exact distribution of the empirical correlation of two independent AR(1) processes with Gaussian increments. We proceed to develop rates of convergence for the distribution of the scaled empirical correlation to the standard Gaussian distribution in both Wasserstein distance and Kolmogorov distance. Given $n$ data points, we prove the convergence rate in Wasserstein distance is $n^{-1/2}$ and the convergence rate in Kolmogorov distance is $n^{-1/2} \sqrt{\ln n}$. We conclude by extending these results to two AR(1) processes with correlated Gaussian increments.
\end{abstract}

\indent \indent \textbf{MSC 2020 Codes} Primary: 60G15, 60F05. Secondary: 62M10.\\

\indent \indent \textbf{Keywords}: Autoregressive processes; Gaussian processes; Kolmogorov distance; \indent \indent Wasserstein distance; Wiener chaos; Yule's ``nonsense correlation.''

\section{Introduction.}

It is well known that the distribution of the empirical correlation (defined in \eqref{eq:defforthetan} below) of two independent and identically distributed (i.i.d.) simple random walks is widely dispersed and frequently large in absolute value (see histograms in \cite[p.1791]{ernst2017}) and \cite[p.35]{yule1926we}). In addition, the observed correlation has a very different distribution than that of the nominal $t$-distribution.  The phenomenon was first (empirically) observed in 1926 by the famed British statistician G. Udny Yule, who called it ``nonsense correlation'' (\cite{yule1926we}). More than ninety years later, Ernst, Shepp, and Wyner (\cite{ernst2017}) succeeded in analytically calculating the variance of this distribution to be .240522.

Yule's ``nonsense correlation'' provides a stark warning that, in the case of two sequences of i.i.d. random walks (or, alternatively, two independent standard Wiener processes), empirical correlation cannot be used to test the independence of the two processes. However, this may not necessarily hold for other classes of stochastic processes, especially those admitting stationary structures. In particular, for discrete-time processes, it is a well-known result that the autocorrelation and cross-correlation (which corresponds to empirical correlation in our framework) of two independent AR(1) processes follow asymptotic normal distributions with zero mean. A formal proof of this result is contained in \cite[Theorem 11.2.2]{brockwell2002introduction}. For further details, we also refer readers to the monographs \cite{shumway2000time,hannan2009multiple} for an overview. Furthermore, for continuous-time processes, the authors of \cite{ernst2019distribution} proved that, for two independent Ornstein-Uhlenbeck processes, the scaled empirical correlation asymptotically converges to a Gaussian random variable with mean zero and variance $\frac{1}{2r}$, where $r$ is the mean reversion parameter of the Ornstein-Uhlenbeck process (see \cite[Theorem 4]{ernst2019distribution}). Accordingly, and desirably, the variance tends to zero as the mean reversion parameter tends towards $\infty$. Later, the rate of this convergence was studied in \cite{douissi2021asymptotics}.

However, the aforementioned results are insufficient for practical purposes. For instance, \cite{ernst2019distribution} and \cite{douissi2021asymptotics} focus mainly on continuous-time processes, whereas real-world observations are discrete-time. Although \cite{douissi2021asymptotics} also examined the empirical correlation between discrete-time processes sampled from independent Ornstein-Uhlenbeck processes, their results depend on the observation mesh approaching $0$ -- a requirement infeasible for given discrete-time datasets. This leads us to consider the discrete-time version of Ornstein-Uhlenbeck process, which arises from equidistant sampling of its continuous counterpart and is provably an AR(1) process. Additionally, existing literature on empirical correlation of independent AR(1) processes has focused primarily on asymptotic results, with less comprehensive analysis of exact distributions or convergence rates. These gaps motivate our study, and we thus consider the following two AR(1) processes: they are derived from equidistant sampling of two independent Ornstein-Uhlenbeck processes, followed by appropriate scaling.

Consider
\begin{equation}
\begin{cases}
X_{n} = \alpha X_{n-1} + \xi_{n}, \quad n \in \mathds{N}_{+} \\
X_0 =0
\end{cases}
\text{  and  }
\begin{cases}
Y_{n} = \alpha Y_{n-1} + \eta_{n}, \quad n \in \mathds{N}_{+} \\
Y_0 =0
\end{cases},  \label{eq:defofAR(1)processes}
\end{equation}
where $|\alpha|<1$ and $\xi_1, \xi_2, \xi_3, \dots$, $\eta_1, \eta_2, \eta_3, \dots$ are independent standard normal random variables. In fact, $\{X_n\}_{n=0}^{\infty}$ and $\{Y_n\}_{n}^{\infty}$ are two independent AR(1) processes with Gaussian increments. Notably, in our model, the initial conditions are set as $X_0 = 0$ and $Y_0 = 0$, chosen to align with the discrete-time analog of the Ornstein-Uhlenbeck processes studied in \cite{douissi2021asymptotics}. We believe that our analytical framework remains applicable if $X_0$ and $Y_0$ follow i.i.d. zero-mean Gaussian distributions, independent of $\xi$ and $\eta$, though this extension would introduce additional computational complexity. Furthermore, the empirical correlation for $\{X_n\}_{n=0}^{\infty}$ and $\{Y_n\}_{n=0}^{\infty}$ is defined in the standard way
\begin{equation}
\theta_{n} := \frac{ \frac{1}{n} \sum_{i=1}^{n} X_i Y_i - \frac{1}{n^2} (\sum_{i=1}^{n} X_i) (\sum_{i=1}^{n} Y_i) }{  \sqrt{ \frac{1}{n} \sum_{i=1}^{n} X_i^2 - \frac{1}{n^2} (\sum_{i=1}^{n} X_i)^2 } \sqrt{ \frac{1}{n} \sum_{i=1}^{n} Y_i^2 - \frac{1}{n^2} (\sum_{i=1}^{n} Y_i)^2 }}. \label{eq:defforthetan}
\end{equation}

Before presenting our results on $\theta_n$, we mention relevant literature. Beyond the aforementioned works, residual cross-correlations are also used to examine the independence of time series, with related work including \cite{hong1996testing,haugh1976checking,himdi1997tests,robbins2015cross}. In machine learning, testing the independence of observations often relies on the Hilbert–Schmidt independence criterion (HSIC). See, for example, \cite{gretton2005kernel,gretton2007kernel,smola2007hilbert,zhang2008kernel}. Furthermore, studies on the convergence rates for estimation include \cite{douissi2020ar,douissi2022berry,es2019optimal}.

In recent literature, distance covariance has emerged as an additional methodology for testing the independence of stochastic processes. The concept of distance correlation is first introduced by Sz{\'e}kely et al. \cite{szekely2007measuring}, and further explored by Sz{\'e}kely and Rizzo in a series of papers \cite{szekely2009brownian,szekely2012uniqueness,szekely2013distance,szekely2014partial}. Subsequently, distance correlation is also applied to measuring the dependence between two stochastic processes. Relevant work in this area includes \cite{Matsui2017distance,Davis2018application,10.3150/20-BEJ1206,betken2021bootstrap,betken2021longrange}. Specifically, Matsui et al. \cite{Matsui2017distance} introduce a formulation of distance covariance for two stochastic processes on some interval. Later, Davis et al. \cite{Davis2018application} establish the asymptotic theory for the sample auto- and cross-distance correlation functions of two stationary multivariate time series under a strong mixing condition. In parallel, Betken et al. \cite{betken2021bootstrap} present a test for dependence between two strictly stationary time series, based on a bootstrap procedure for distance covariance. For stationary, long-range dependent time series, Betken and Dehling \cite{betken2021longrange} develop a test statistic for independence that relies on a linear combination of empirical distance cross-covariances, defined for two processes observed at fixed time lags. Further, Dehling et al. \cite{10.3150/20-BEJ1206} utilize distance covariance to construct a measure of independence for a pair of stochastically continuous and bounded stochastic processes on the unit interval, given sample data consisting of discretizations from an i.i.d. sequence on the same partition. 

However, our study differs from existing literature in two aspects: (i) We focus primarily on testing the independence of two AR(1) processes via the empirical correlation; (ii) In addition to examining the asymptotic normality of the empirical correlation, we also provide both the exact distribution for any given $n$ and the rate of convergence to the limiting distribution. While the rate of convergence helps determine when the number of data points $n$ is sufficient for using asymptotic normality in independence testing, the results on the exact distribution provide a remedy to the accuracy problems caused by having insufficient data points.

We proceed to introduce our results. Section \ref{sec:exactdistribution} studies the exact distribution of the empirical correlation $\theta_n$ (defined in \eqref{eq:defforthetan}). To this end, we derive formulas for all the moments of $\theta_n$ for any given $n$ and use the first $10$ moments to approximate the density of $\sqrt{n} \theta_n$. The analysis is based on a symbolically tractable integro-differential representation expression for moments (see, for example, \cite{ernst2021yule,ernst2019distribution,Sawa}), which requires explicitly computing the tri-variate moment generating function (MGF) of the three sums comprising $\theta_n$. To achieve this, we represent these sums as quadratic forms in the increments $\xi$ and $\eta$ to obtain matrix $K_n$ (defined in \eqref{eq:defKn}), then explicitly calculate its ``alternative characteristic polynomial'' $d_n(\lambda)$ (defined in \eqref{polynow}), yielding a closed-form for the tri-variate MGF $\phi_n$. In contrast to the matrix $K_n$ in \cite{ernst2021yule}, the matrix $K_n$ here is considerably more complex, and indeed none of the techniques developed in \cite{ernst2021yule} can be applied. In computing $d_n(\lambda)$, the key observation is that the matrix $I_n - \lambda K_n$ can be decomposed as the sum of two matrices. The first matrix (after elementary matrix operations) is an invertible tri-diagonal matrix which exhibits self-similarity (with the exception of one cell) and the second matrix has rank $1$. Using this, the determinant of $I_n - \lambda K_n$ can be expressed in terms of the determinant and cofactors of the aforementioned tri-diagonal matrix, which can be explicitly derived by solving a second-order recursion formula. The explicit calculation of $d_{n}(\lambda)$ constitutes a fundamental mathematical tool in this paper. This is because (i) it serves as a key ingredient in computing the tri-variate MGF, which underpins the formulas for the moments of $\theta_n$; and (ii) it plays an important role in estimating the negative moments and the moment generating functions of elements forming the denominator of $\theta_n$, which are crucial for deriving the rates of convergence in Section \ref{sec:mainconvergence}. The full details of these calculations are provided in Appendix A.

Section \ref{sec:mainconvergence} establishes convergence rates for the distribution of the scaled empirical correlation (i.e. $\sqrt{n} \theta_n$ with a normalization) of two independent AR(1) processes with Gaussian increments to the standard normal distribution. Specifically, we find that the convergence rate in Wasserstein distance is $n^{-1/2}$; in Kolmogorov distance, it is $n^{-1/2}  \sqrt{\ln n }$. The derivation of the convergence rate relies upon the estimates of the eigenvalues of $K_n$ in Section \ref{subsec:estimateeigenvalues}.
In Section \ref{subsec:nominatorconvergence}, we analyze the convergence of the scaled numerator of $\theta_n$ to the standard normal distribution and establish its convergence rate as $n^{-1/2}$. This result draws heavily on the work of Nourdin and Peccati \cite{nourdin2009stein}, who showed that the Kolmogorov distance between the distribution of a $q$-th Wiener chaos random variable $X$ and the standard normal distribution can be bounded using the Malliavin derivative $DX$. The Nourdin-Peccati result extends straightforwardly to Wasserstein distance and total variation; further details are available in \cite{nourdin2012normal,nualart2018introduction}. Section \ref{subsec:convergencewasserstein} and Section \ref{subsec:convergenceKolmogorov} are devoted to utilizing the convergence rate of the scaled numerator of $\theta_n$ to determine the convergence rate of the entire fraction (in both Wasserstein distance and Kolmogorov distance). The work relies extensively on the explicit expression of the aforementioned alternative characteristic polynomial $d_{n}(\lambda)$. For Wasserstein distance, the key ingredient is developing a uniform upper bound for the negative second moments of the elements of the denominator of $\theta_n$. After extensive calculation, this reduces to a uniform lower bound of $(n-1)$ times the product of the positive eigenvalues of $K_n$. Using the explicit expression for the alternative characteristic function $d_{n}(\lambda)$, the product of these positive eigenvalues can be calculated explicitly, yielding the required uniform lower bound. For Kolmogorov distance, the analysis requires estimating the tails of the elements in the denominator of $\theta_n$. These estimates are derived from the explicit expressions of their moment generating functions, which can be represented in terms of $d_n(\lambda)$. These two parts culminate in convergence rates of $n^{-1/2}$ for Wasserstein distance and $n^{-1/2} \sqrt{\ln n}$ for Kolmogorov distance. Notably, our result slightly improves upon that in \cite{douissi2021asymptotics}, where the rate of convergence in Kolmogorov distance is $T^{-1/2} \ln T$; here, $T$ represents the length of the Ornstein–Uhlenbeck process and therefore is analogous to the sample size $n$ for a discrete-time process. The improvement arises from our use of the explicit forms of the moment generating functions of two elements comprising the denominator of $\theta_n$, whereas the authors of \cite{douissi2021asymptotics} relied on the upper bounds for these moment generating functions.

In addition to the independent case, we also investigate the distribution and asymptotic behavior of the empirical correlation between two AR(1) processes with correlated Gaussian increments. Formally, two AR(1) processes are correlated with coefficient $r$ if the pairs of increments $(\xi_n, \eta_{n})$ are i.i.d. Gaussian random vectors with mean zero and covariance matrix $$\begin{pmatrix}
1 & r \\
r & 1
\end{pmatrix},$$ where $|r| \leq 1$. We begin by examining the distribution of the empirical correlation through the calculation of its moments at all orders. The details are provided in Section \ref{sbsec:numericorrelatedincrements}. Furthermore, in Section \ref{sec:Power}, we find the convergence rate of the scaled empirical correlation to the standard normal to be $n^{-1/2} \sqrt{\ln n}$ in Kolmogorov distance. This proves that the scaled empirical correlation is asymptotically normal with mean $r$. Unlike the independent case, deriving the convergence rate for the correlated case requires much more effort. This is because, when we re-express $\theta_n - r$ as a fraction (see \eqref{eqc:reformationoftheta-r}), the numerator is no longer a second Wiener-chaos variable. To tackle this issue, we multiply a factor in both numerator and denominator, archiving a fourth Wiener-chaos variable in the numerator. The full details are contained in Appendix C. With the results on convergence rate for both independent and correlated cases, we are capable of analyzing the asymptotic statistical power of the test for the region \begin{equation*}
|\sqrt{n}\, \theta_n| > c,
\end{equation*}
for any constant $c$.

Another process of interest shall be the nearly nonstationary AR(1) process, which follows the same recursion formula as in \eqref{eq:defofAR(1)processes} but allows the coefficient $\alpha$ to vary with $n$ and tend to 1 as $n$ approaches $\infty$. For more details on the nearly nonstationary AR(1) process and two different methods for parameterizing $\alpha_n$, refer to \cite{phillips1987towards} and \cite{cox1991maximum}. When $\alpha_n \equiv 1$, this process corresponds to the random walk. Consequently, the empirical correlation between two such independent processes is Yule's ``nonsense correlation'', with a widely dispersed limiting distribution (see \cite{phillips1986understanding,ernst2017}). When $\alpha_n \equiv \alpha$ which $|\alpha| <1$, it reduces to the cases studied in the present paper. It would be interesting to explore how the distribution and asymptotic behavior of the empirical correlation change as $\alpha_n$ converges to $1$ at different rates. While we believe our approaches could address these questions, they fall beyond the scope of this paper and are left for future research. Additionally, we are looking to address the limitations of the Gaussian distribution assumption and extend our results to more general AR(1) processes. However, without the Gaussian condition, the tri-variate moment generating function becomes extremely complicated and seems nearly impossible to have an explicit expression. This partly explains the limited results regarding exact distributions in this area.

The remainder of the paper is organized as follows. In Section \ref{sec:preliminaries}, we introduce some key elements of analysis on Wiener space, which enable us to work with convergence rates for the scaled numerator of the empirical correlation $\theta_n$. We also introduce some necessary notation, including that of the ``kernel'' matrix $K_n$, its ``alternative characteristic polynomial'' $d_{n}(\lambda)$ and the tri-variate moment generating function $\phi_{n}$. Section \ref{sec:exactdistribution} is devoted to the exact distribution theory of $\theta_n$ for the model of two independent AR(1) processes with Gaussian increments. We start by explicitly calculating $d_n$ and $\phi_n$, and continue with an explicit formula for the second moment of $\sqrt{n}\theta_n$ for any $n$. Additionally, we provide numerical results for moments of $\sqrt{n} \theta_n$, facilitating the approximation of the density of $\sqrt{n} \theta_n$. In Section \ref{sec:mainconvergence}, we study the asymptotic behavior of the empirical correlation for two independent AR(1) processes with Gaussian increments. We also establish the rates of convergence to the standard normal distribution in both Wasserstein distance and Kolmogorov distance. In Section \ref{sec:Power}, we extend these results to the model of AR(1) processes with correlated Gaussian increments, allowing us to study the asymptotic statistical power of the test.

\section{Preliminaries.} \label{sec:preliminaries}
This first subsection introduces fundamental tools for analysis on Wiener space. The exposition below very closely follows \cite[Section 2]{douissi2021asymptotics} and \cite[Section 2]{douissi2020ar}. For further details about the analysis on Wiener space, we refer the reader to \cite{nourdin2012normal} and \cite{nualart2018introduction}. 

\subsection{Fundamental Tools for Analysis on Wiener space}
Let $(\Omega, \mathcal{F}, P)$ denote the Wiener space of a standard Wiener process $W$ and let the space $\mathcal{H}$ be denoted as  $$\mathcal{H}=L^2( \mathbb{R}_{+} ).$$ For two deterministic functions $f$ and $g$, the inner product is defined as $$\langle f, g \rangle_{\mathcal{H}}:=\int_{\mathbb{R}_{+}} f(s) g(s) \,ds.$$ The fundamental elements of analysis on Wiener space employed in the present paper are detailed in the bullet points below.
\begin{itemize}
\item \textbf{Wiener chaos expansion}. For all $q\geq 1$, let $\mathcal{H}_{q}$ denote the $q$-th Wiener chaos of $W$. 
Note that Wiener chaoses of different orders are orthogonal in $L^2(\Omega)$. The fact that any $X \in L^2(\Omega)$ can be orthogonally decomposed as
\begin{equation}
X = E[X] + \sum_{q=1}^{\infty} X_q   \label{eq:chaosexpansion}
\end{equation}
for some $X_q \in \mathcal{H}_{q}$ and for every $q \geq 1$ is known as the ``Wiener chaos expansion.''
\item \textbf{Symmetrization and symmetric tensor space}. For a mapping $f : \mathbb{R}_{+}^{n} \rightarrow \mathbb{R}$, its symmetrization is given by
\begin{equation*}
\tilde{f}(t_1, \dots, t_n) = \frac{1}{n!} \sum_{\sigma} f\left( t_{\sigma(1)}, \dots, t_{\sigma(n)} \right),
\end{equation*}
where the sum is taken over all permutations $\sigma$ of $\{1,2, \dots, n\}$. Let $f \, \widetilde{\otimes} \, g$ denote the symmetrization of the tensor $f\otimes g$. The symmetric tensor space $\mathcal{H}^{\odot q}$ is the linear space of the symmetrization of functions from the tensor space $\mathcal{H}^{\otimes q}$.
\item \textbf{Multiple Wiener integrals}. 
The mapping $\mathcal{I}_{q}$ is the so-called ``multiple Wiener integral.'' The mapping $\mathcal{I}_{q}$ is a contraction; that is, for any $f \in \mathcal{H}^{\otimes q}$,
\begin{equation*}
\left\lVert \mathcal{I}_{q}(f) \right\rVert_{\mathcal{H}_{q}} \leq \sqrt{q!} \left\lVert f \right\rVert_{\mathcal{H}^{\otimes q}}.
\end{equation*}
It is useful to note that, for any $f \in \mathcal{H}^{\otimes q}$, $\mathcal{I}_{q}( \tilde{f} ) = \mathcal{I}_{q}(f).$ 

If we consider the symmetric tensor space $\mathcal{H}^{\odot q}$, then $\mathcal{I}_{q}$ is a linear isometry between the symmetric tensor product $\mathcal{H}^{\odot q}$ (equipped with the modified norm $ \lVert \cdot \rVert_{\mathcal{H}^{\odot q}} $) and $\mathcal{H}_{q}$. For any nonnegative integers $p$ and $q$, and functions $f \in \mathcal{H}^{\odot p}$ and $g \in \mathcal{H}^{\odot q}$, the formula for the inner product is given by
\begin{equation} \label{eq:innerproductWeinerIntegral}
E\left[ \mathcal{I}_{p}(f) \mathcal{I}_{q}(g) \right]
= \begin{cases}
p! \left \langle f ,g \right \rangle_{\mathcal{H}^{\otimes p}} & \quad \text{if $p=q$} \\
0 & \quad \text{otherwise}
\end{cases},
\end{equation}

Hence, for both $X$ and its Wiener chaos expansion in \eqref{eq:chaosexpansion} above, each term $X_q$ is a multiple Wiener integral $\mathcal{I}_{q}(f_q)$ for some $f_q \in \mathcal{H}^{\odot q}$.
\item \textbf{Product formula and contractions}. Let $f \in \mathcal{H}^{\otimes p}$ and $g \in \mathcal{H}^{\otimes q}$. For any $r = 0, \dots, p \wedge q$, the contraction of $f$ and $g$ of order $r$ is defined as an element in $\mathcal{H}^{\otimes (p+q -2r)}$ and is given by
\begin{eqnarray*}
&&( f \otimes_r g )(s_1, \dots, s_{p-r}, t_1, \dots, t_{q-r}) \\
&:=& \int_{ \mathbb{R}_{+}^{r} } f( s_1, \dots, s_{p-r}, u_1, \dots, u_r  )\, g( t_1, \dots, t_{q-r}, u_1, \dots, u_r ) \, du_1 \cdots du_r.
\end{eqnarray*}
Furthermore, the product of two multiple Wiener integrals satisfies 
\begin{equation}
\mathcal{I}_{p}(f) \mathcal{I}_{q}(g) = \sum_{r=0}^{ p \wedge q } r! \binom{p}{r} \binom{q}{r} \mathcal{I}_{p+q -2r}( f \widetilde{\otimes}_{r} g ), \label{eq:productformulaforWienerintegral}
\end{equation}
where $f \widetilde{\otimes}_r g$ denotes the symmetrization of $f \otimes_r g$.

\item \textbf{Hypercontractivity in Wiener chaos}. For $q \geq 2$, every random variable in $q$-th Wiener chaos admits the ``hypercontractivity property.'' The hypercontractivity property implies the equivalence in $\mathcal{H}_q$ of all $L^p$ norms. In particular, for any $F \in \mathcal{H}_{q}$, and for $1<p<r<\infty$, 
\begin{equation}
\left( E\left[ |F|^{r} \right] \right)^{1/r} \leq \left( \frac{r-1}{p-1} \right)^{q/2} \left( E\left[ |F|^{p} \right] \right)^{1/p}.  \label{eq:hypercontractivitypropertyforchaos}
\end{equation}
\item \textbf{The Malliavin derivative}. Consider a function $\Phi \in C^{1}( \mathbb{R}^{n} )$ with bounded derivative. Let $h_1, \dots, h_n \in \mathcal{H}$. Let the Wiener integral $\int_{\mathbb{R}_{+}} h(s) \,dW(s)$ be denoted by $W(h)$. The ``Malliavin derivative'' $D$ of the random variable $X:= \Phi(W(h_1), \dots, W(h_n) )$ obeys the following chain rule:
\begin{equation*}
DX : X \mapsto D_{s} X := \sum_{k=1}^{n} \frac{\partial \Phi}{\partial x_{k}}( W(h_1), \dots, W(h_n) ) \, h_{k}(s) \in L^{2}(\Omega \times \mathbb{R}_{+}).
\end{equation*}
One may proceed to extend the Malliavin derivative $D$ to the Gross-Sobolev subset $\mathbb{D}^{1,2} \subsetneqq L^2(\Omega)$. This is done by closing $D$ inside $L^2(\Omega)$ under the norm
\begin{equation*}
\lVert X \rVert_{1,2}^{2} = E\left[ X^2 \right] + E\left[ \int_{\mathbb{R}_{+}} | D_s X|^2 \, ds \right].
\end{equation*}
A useful fact is that all Wiener chaos random variables are in the domain $\mathbb{D}^{1,2}$ of $D$. For any $X \in L^2(\Omega)$ with Wiener chaos expansion $E[X]+ \sum_{q} \mathcal{I}_{q}(f_q),$ $X \in \mathbb{D}^{1,2}$ if and only if $\sum_{q} q q! \lVert f_q \rVert_{\mathcal{H}^{\otimes q}}^2  < \infty$.
\item \textbf{Generator} $L$ \textbf{of the Ornstein-Uhlenbeck semigroup}. The linear operator $L$ is diagonal under the Wiener chaos expansion of $L^2(\Omega)$. Further, $\mathcal{H}_q$ is the eigenspace of $L$ with eigenvalue $-q$ (or, equivalently, for $X\in \mathcal{H}_q$, $LX = -q X$). Finally, the operator $- L^{-1}$ is the negative pseudo-inverse of $L$. This means that for any $X\in \mathcal{H}_q$, $- L^{-1} X = q^{-1} X$.
\item \textbf{Kolmogorov distance, Wasserstein distance, and total variation}. Let $X$ and $Y$ be two real-valued random variables. We define the Kolmogorov distance between the law of $X$ and the law of $Y$ as
\begin{equation*}
d_{Kol}(X,Y) = \sup_{z \in \mathbb{R}} \left| P(X \leq z) - P(Y \leq z)\right|.
\end{equation*}
We define the total variation between the law of $X$ and the law of $Y$ by
\begin{equation}
d_{TV}(X,Y) = \sup_{ A \in \mathcal{B}(\mathbb{R}) } \left| P(X  \in A ) - P(Y  \in A )\right|,  \label{eqchaos:defofTVdistance}
\end{equation}
where the supremum is taken over all Borel sets on $\mathbb{R}$.
If $X$ and $Y$ are integrable, we may define the Wasserstein distance between the law of $X$ and the law of $Y$ by
\begin{equation*}
d_{W}(X,Y) = \sup_{f \in \mathrm{Lip}(1)} \left| Ef(X) - Ef(Y) \right|,
\end{equation*}
where $\mathrm{Lip}(1)$ is the set of all Lipschitz functions with Lipschitz constant $\leq 1$. Moreover, let $X$ be a mean zero random variable with $ X \in \mathbb{D}^{1,2}$ and let the random variable $Y$ be standard normal. The following upper bounds for Kolmogorov distance, Wasserstein distance, and total variation are given below (see Chapter 8.2, 8.3 in \cite{nualart2018introduction} or Theorem 2.4 in \cite{nourdin2009stein}):
\begin{eqnarray}
&& d_{Kol}(X,Y) \leq \sqrt{E\left[ \left( 1- \langle DX, - DL^{-1}X \rangle_{\mathcal{H}} \right)^2 \right]}, \label{eq:firstmalliavinboundforKolmogorov} \\ [1mm]
&& d_{W}(X,Y) \leq \sqrt{\frac{2}{\pi}} \sqrt{E\left[ \left( 1- \langle DX, - DL^{-1}X \rangle_{\mathcal{H}} \right)^2 \right]},  \label{eq:firstmalliavinboundforWasserstein} \\ [1mm]
&& d_{TV}(X,Y) \leq 2 \sqrt{E\left[ \left( 1- \langle DX, - DL^{-1}X \rangle_{\mathcal{H}} \right)^2 \right]} . \label{eq:firstmalliavinboundfortotalvariation}
\end{eqnarray}
Note that if, for all $q \geq 2$, $X \in \mathcal{H}_{q}$ , then $ \langle DX, - DL^{-1}X \rangle_{\mathcal{H}} = q^{-1} \lVert DX \rVert_{\mathcal{H}}^{2} $. The following three inequalities hold:
\begin{eqnarray}
&&  d_{Kol} (X,Y) \leq \sqrt{ E\left[ \left(1- q^{-1} \lVert DX \rVert_{\mathcal{H}}^{2}\right)^2 \right] },   \label{eq:malliavinboundforKolmogorov}  \\ [1mm]
&&  d_{W} (X,Y) \leq \sqrt{\frac{2}{\pi}} \sqrt{ E\left[ \left(1- q^{-1} \lVert DX \rVert_{\mathcal{H}}^{2}\right)^2 \right] },   \label{eq:malliavinboundforWasserstein} \\ [1mm]
&& d_{TV} (X,Y) \leq 2 \sqrt{ E\left[ \left(1- q^{-1} \lVert DX \rVert_{\mathcal{H}}^{2}\right)^2 \right] }.   \label{eq:malliavinboundfortotalvariation}
\end{eqnarray}
\end{itemize}

\subsection{Notation.}
This subsection introduces some necessary notation to be utilized in the sequel. We use $I_{n}$ to denote the $n\times n$ identity matrix. For $n \in \mathds{N}_{+}$, we define the $n\times n$ symmetric matrix $K_{n}$ by 
\begin{equation}
K_{n} := \left\{ \frac{1}{n} \, \frac{\alpha^{|k-j|} - \alpha^{k+j}}{1- \alpha^2} - \frac{1}{n^2} \, \frac{(1-\alpha^k)(1- \alpha^j)}{(1-\alpha)^2} \right\}_{j,k=1}^{n}. \label{eq:defKn}
\end{equation}
The ``alternative characteristic polynomial" $d_{n}(\lambda)$ for the matrix $K_{n}$ is defined by 
\begin{equation}\label{polynow}
d_{n}(\lambda) = \det(I_{n} -  \lambda K_n). 
\end{equation} 
In the case that the eigenvalues of $K_{n}$ are known, and are denoted by $\lambda_1, \lambda_2 \cdots, \lambda_n$, the alternative characteristic polynomial can also be expressed as
\begin{equation}
 d_{n}(\lambda) =\prod_{j=1}^{n} (1- \lambda_j \lambda). \label{eq:anotherformofacp}  
\end{equation}
We also define two $ n \times 1$ column random vectors $\mathbf{\Xi}_{n}$ and $\mathbf{H}_{n}$ by
\begin{equation*}
\mathbf{\Xi}_{n} := \left( \xi_{n}, \xi_{n-1} , \dots, \xi_1 \right)^{\intercal}  \quad \text{and} \quad
\mathbf{H}_{n} := \left( \eta_n, \eta_{n-1} , \dots, \eta_1 \right)^{\intercal}, 
\end{equation*}
where $\xi_1, \xi_2, \xi_3, \dots$, $\eta_1, \eta_2, \eta_3, \dots$ are independent standard normal random variables. Let
\begin{eqnarray}
&& Z_{11}^n := \frac{1}{n} \sum_{i=1}^{n} X_i^2 - \frac{1}{n^2} \left(\sum_{i=1}^{n} X_i\right)^2, \label{eq:defz11} \\
&& Z_{22}^n := \frac{1}{n} \sum_{i=1}^{n} Y_i^2 - \frac{1}{n^2} \left(\sum_{i=1}^{n} Y_i\right)^2,  \label{eq:defz22} \\
&& Z_{12}^n := \frac{1}{n} \sum_{i=1}^{n} X_i Y_i - \frac{1}{n^2} \left(\sum_{i=1}^{n} X_i\right) \left(\sum_{i=1}^{n} Y_i\right). \label{eq:defz12}
\end{eqnarray}
Combining the above with \eqref{eq:defforthetan}, one can easily check that $$\theta_n = Z_{12}^n/\sqrt{Z_{11}^n Z_{22}^n}.$$ Finally, let us define the joint moment generating function (mgf) of the random vector $\left( Z^n_{11}, Z^n_{12}, Z^n_{22} \right)$ by
\begin{equation*}
\phi_{n}(s_{11}, s_{12}, s_{22}) := E \left[ \exp \left\{ -\frac{1}{2} \left( s_{11} Z_{11}^{n} + 2 s_{12} Z_{12}^{n} + s_{22} Z_{22}^{n} \right) \right\} \right], 
\end{equation*}
where $s_{11}, s_{12} $ and $s_{22}$ are such that $s_{11}, s_{22} \geq 0$ and $s_{12}^2 \leq s_{11} s_{22}$. These two inequalities ensure that $\phi_{n}(s_{11}, s_{12}, s_{22})$ is well-defined, as we shall see in Section \ref{sec:jointmgf}.

\section{The distribution of $ \sqrt{n}\theta_n$.}  \label{sec:exactdistribution}
This section is devoted to the exact distribution theory of the scaled empirical correlation of two independent AR(1) processes with Gaussian noise. Let two independent AR(1) processes be defined by
\begin{equation*}
\begin{cases}
X_{n} = \alpha X_{n-1} + \xi_{n}, \quad n \in \mathds{N}_{+} \\
X_0 =0
\end{cases}
\text{  and  }
\begin{cases}
Y_{n} = \alpha Y_{n-1} + \eta_{n}, \quad n \in \mathds{N}_{+} \\
Y_0 =0
\end{cases},  
\end{equation*}
where $|\alpha|<1$ and $\xi_1, \xi_2, \xi_3, \dots$, $\eta_1, \eta_2, \eta_3, \dots$ are independent standard normal random variables. We study the empirical correlation (as originally defined in \eqref{eq:defforthetan} above) for the AR(1) processes $\{X_n\}_{n=0}^{\infty}$ and $\{Y_n\}_{n=0}^{\infty}$.
Our interest is to study the exact distribution of $\sqrt{n} \theta_n$ for any fixed $n$ (note that we do not study the distribution of $\theta_n$ since $\theta_n$ converges to $0$). We begin by calculating the joint moment generating function $\phi_{n}(s_{11}, s_{12}, s_{22})$, followed by the formulas for moments of $\sqrt{n} \theta_n$. We then proceed to approximate the distribution of $\sqrt{n} \theta_n$ using its higher-order moments.

\subsection{Calculating the joint moment generating function.} \label{sec:jointmgf}
In this subsection, we derive an expression for the joint moment generating function $\phi_{n}(s_{11}, s_{12}, s_{22})$, which shall serve as the key ingredient for computing the moments of $\theta_{n}$ for all $n$.

By the definition of an AR(1) process, we may easily obtain that, for all $n \in \mathds{N}_{+}$,
\begin{equation*}
X_n = \sum_{k=1}^{n} \alpha^{n-k} \xi_{k} \quad \text{and} \quad
Y_{n} = \sum_{k=1}^{n} \alpha^{n-k} \eta_{k}.
\end{equation*}
Then
\begin{eqnarray}
&& \sum_{i=1}^{n} X_i Y_i = \sum_{i=1}^{n} \left( \sum_{k=1}^{i} \alpha^{i-k} \xi_k \right) \left( \sum_{j=1}^{i} \alpha^{i-j} \eta_j \right) 
= \sum_{j,k=1}^{n} \, \sum_{i= \max(k,j)}^{n} \alpha^{i-k} \alpha^{i-j} \,  \xi_{k} \eta_{j}  \notag \\
&=& \sum_{j,k=1}^{n} \frac{\alpha^{2 \times \max(k,j)} - \alpha^{2n+2}}{1- \alpha^2}\, \alpha^{-k} \alpha^{-j} \,  \xi_{k} \eta_{j} 
= \sum_{j,k=1}^{n} \frac{\alpha^{ |k-j|} - \alpha^{2n+2 -k -j}}{1- \alpha^2}  \,  \xi_{k} \eta_{j}.  \label{eq:expressionforsumXY}
\end{eqnarray}
Further,
\begin{equation*}
\sum_{i=1}^{n} X_i = \sum_{i=1}^{n} \sum_{k=1}^{i} \alpha^{i-k} \,\xi_{k}
=\sum_{k=1}^{n} \sum_{i=k}^{n} \alpha^{i-k}\, \xi_{k}
= \sum_{k=1}^{n} \frac{1- \alpha^{n-k+1}}{1- \alpha}\, \xi_{k}.
\end{equation*}
Similarly,
\begin{equation*}
\sum_{i=1}^{n} Y_i = \sum_{k=1}^{n} \frac{1- \alpha^{n-k+1}}{1- \alpha}\, \eta_{k}.
\end{equation*}
Hence,
\begin{equation}
\left(\sum_{i=1}^{n} X_i\right) \left(\sum_{i=1}^{n} Y_i\right)
= \sum_{j,k=1}^{n} \frac{\left(1- \alpha^{n-k+1}\right) \left(1- \alpha^{n-j+1}\right) }{(1- \alpha)^2}\, \xi_{k} \eta_{j}.  \label{eq:expressionforsumXsumY}
\end{equation}
Combining \eqref{eq:defz12}, \eqref{eq:expressionforsumXY} and \eqref{eq:expressionforsumXsumY} yields
\begin{eqnarray*}
Z_{12}^{n} &=& \sum_{j,k=1}^{n} \left[ \frac{1}{n}\, \frac{\alpha^{ |k-j|} - \alpha^{2n+2 -k -j}}{1- \alpha^2} - \frac{1}{n^2} \frac{\left(1- \alpha^{n-k+1}\right) \left(1- \alpha^{n-j+1}\right) }{(1- \alpha)^2}  \right] \, \xi_{k} \eta_{j}  \\
&=& \sum_{j,k=1}^{n} \left[ \frac{1}{n}\, \frac{\alpha^{ |k-j|} - \alpha^{k + j}}{1- \alpha^2} - \frac{1}{n^2} \frac{\left(1- \alpha^{k}\right) \left(1- \alpha^{j}\right) }{(1- \alpha)^2}  \right] \, \xi_{n +1 - k} \, \eta_{n +1 -j} \\ [1mm]
&=& \mathbf{\Xi}_n^{\intercal} K_n \mathbf{H}_{n},
\end{eqnarray*}
where the second equality holds by making the change of variables: $k = n+1 -k$ and $j= n+1 - j$. Similarly, we have
\begin{equation*}
Z_{11}^{n} = \mathbf{\Xi}_n^{\intercal} K_n \mathbf{\Xi}_{n} \quad  \text{and}  \quad Z_{22}^{n} = \mathbf{H}_n^{\intercal} K_n \mathbf{H}_{n}.
\end{equation*}
\begin{remark} \label{remark:semidefinite}
The above display reveals the matrix $K_n$ is positive semi-definite, since $Z_{11}^{n} \geq 0$. Noting that $Z_{11}^{n} =0$ if and only if $X_1 = X_2 = \dots = X_n$, one can easily see that the rank of $K_n$ is $n-1$.
\end{remark}

Since $K_n$ is a symmetric matrix, it admits the following orthogonal decomposition 
\begin{equation*}
K_{n} = P_{n}^{\intercal} \, \mathrm{diag} (\lambda_1, \lambda_2, \dots, \lambda_{n}) \, P_{n},
\end{equation*}
where $P_n$ is a $n \times n $ orthogonal matrix, $\lambda_1, \lambda_2, \dots, \lambda_{n}$ are eigenvalues of $K_n$ and $\mathrm{diag} (\lambda_1, \lambda_2, \dots, \lambda_{n})$ is a diagonal matrix whose entry in the $j$-th row and the $j$-th column is $\lambda_{j}$. Further, let
\begin{eqnarray*}
&& \mathbf{W}_{n} = \left( W_1, W_2, \dots, W_n \right)^{\intercal} := P_{n}\, \mathbf{\Xi}_n  \\
&& \mathbf{V}_{n} = \left( V_1, V_2, \dots, V_n \right)^{\intercal} := P_{n}\, \mathbf{H}_n.
\end{eqnarray*}
Then
\begin{eqnarray}
Z_{12}^{n} &=& \mathbf{\Xi}_n^{\intercal} K_n \mathbf{H}_{n}
= \mathbf{\Xi}_n^{\intercal} \, P_{n}^{\intercal} \, \mathrm{diag} (\lambda_1, \lambda_2, \dots, \lambda_{n}) \, P_{n} \, \mathbf{H}_{n} \notag \\
&=& \mathbf{W}_{n}^{\intercal} \, \mathrm{diag} (\lambda_1, \lambda_2, \dots, \lambda_{n}) \, \mathbf{V}_n
= \sum_{k=1}^{n} \lambda_k\, W_k V_k.  \label{eq:decompositionforZ12}
\end{eqnarray}
Similarly,
\begin{eqnarray}
&& Z_{11}^{n} = \mathbf{W}_{n}^{\intercal} \, \mathrm{diag} (\lambda_1, \lambda_2, \dots, \lambda_{n}) \, \mathbf{W}_n =  \sum_{k=1}^{n} \lambda_k\, W_k^2, 
\label{eq:decompositionforZ11}  \\
&& Z_{22}^{n} = \mathbf{V}_{n}^{\intercal} \, \mathrm{diag} (\lambda_1, \lambda_2, \dots, \lambda_{n}) \, \mathbf{V}_n = \sum_{k=1}^{n} \lambda_k\, V_k^2 . 
\label{eq:decompositionforZ22}
\end{eqnarray}
Noting that the multivariate standard normal distribution is invariant under orthogonal transformations, combined with the fact that random vectors $\mathbf{\Xi}_n$ and $\mathbf{H}_n$ are independent and have multivariate standard normal distributions, it follows immediately that $W_1, W_2, \dots, W_n$, $V_1, V_2, \dots, V_n$ are independent standard normal random variables.

Before presenting Theorem \ref{thm1} below, we pause to reveal an explicit calculation of the alternative characteristic polynomial $d_{n}(\lambda)$. The proof is relegated to Appendix A.

\begin{lemma}\label{lem1}
The alternative characteristic polynomial $d_{n}(\lambda)$ may be written as
\begin{eqnarray*}
d_{n}(\lambda) &=& \frac{\left(\gamma_{1}(\lambda/n)\right)^{n+1}  - \left(\gamma_{2}(\lambda/n)\right)^{n+1}}{\sqrt{\Delta(\lambda/n)}}
-  \frac{\alpha^2\left[\left(\gamma_{1}(\lambda/n)\right)^{n}  - \left(\gamma_{2}(\lambda/n)\right)^{n}\right]}{\sqrt{\Delta(\lambda/n)}} \\ [2mm]
&& + \frac{\lambda\left[\left(\gamma_{1}(\lambda/n)\right)^{n+1}  + \left(\gamma_{2}(\lambda/n)\right)^{n+1}\right]}{n \,\Delta(\lambda/n)}  
  - \frac{(n-1) \alpha^2 \, \lambda\left[\left(\gamma_{1}(\lambda/n)\right)^{n}  + \left(\gamma_{2}(\lambda/n)\right)^{n}\right]}{n^2 \,\Delta(\lambda/n)}  \\ [2mm]
&& + \frac{2 (n-1) \alpha \, \lambda\left[\left(\gamma_{1}(\lambda/n)\right)^{n+1}  + \left(\gamma_{2}(\lambda/n)\right)^{n+1}\right]}{n \,\Delta(\lambda/n) \left( n (1-\alpha)^2 - \lambda \right)}
- \frac{2 (n-2) \alpha^3 \, \lambda\left[\left(\gamma_{1}(\lambda/n)\right)^{n}  + \left(\gamma_{2}(\lambda/n)\right)^{n}\right]}{n \,\Delta(\lambda/n) \left( n (1-\alpha)^2 - \lambda \right)}  \\  [2mm]
&& - \frac{2 (n+1) \alpha^2 \, \lambda\left[\left(\gamma_{1}(\lambda/n)\right)^{n}  + \left(\gamma_{2}(\lambda/n)\right)^{n}\right]}{n \,\Delta(\lambda/n) \left( n (1-\alpha)^2 - \lambda \right)}
+ \frac{2  \alpha^4 \, \lambda\left[\left(\gamma_{1}(\lambda/n)\right)^{n-1}  + \left(\gamma_{2}(\lambda/n)\right)^{n-1}\right]}{\Delta(\lambda/n) \left( n (1-\alpha)^2 - \lambda \right)} \\[2mm]
&& + \frac{2 \, \alpha^{n+1} (1- \alpha) \, \lambda\left[\gamma_{1}(\lambda/n)  + \gamma_{2}(\lambda/n)\right]}{n \,\Delta(\lambda/n) \left( n (1-\alpha)^2 - \lambda \right)}  
+ \frac{4 \, \alpha^{n+2} (1- \alpha) \, \lambda}{n \,\Delta(\lambda/n) \left( n (1-\alpha)^2 - \lambda \right)},
\end{eqnarray*}
where
\begin{eqnarray}
&& \gamma_1 := \gamma_{1} (\lambda) = \frac{(1-\lambda + \alpha^2) + \sqrt{(1-\lambda+\alpha^2)^2 - 4 \alpha^2}}{2},  \label{eq:defgamma1} \\ [1mm]
&& \gamma_2 := \gamma_{2} (\lambda) = \frac{(1-\lambda + \alpha^2) - \sqrt{(1-\lambda+\alpha^2)^2 - 4 \alpha^2}}{2},  \label{eq:defgamma2}\\ [1mm]
&& \Delta := \Delta(\lambda) = (1-\lambda+\alpha^2)^2 - 4 \alpha^2.  \label{eq:defDelta}
\end{eqnarray}
\end{lemma}
\begin{proof}
See Appendix A.
\end{proof}

With the above preparation in hand, we may now calculate the joint mgf $\phi_{n}(s_{11}, s_{12}, s_{22})$.

\begin{theorem}\label{thm1}
The joint moment generating function may be calculated as
\begin{equation*}
\phi_{n}(s_{11}, s_{12}, s_{22})= \left( d_{n}(\rho) \, d_{n}(\upsilon)  \right)^{-1/2} 
\end{equation*}
where $\rho$ and $\upsilon$ are defined as follows:
\begin{eqnarray}
&&\rho := \rho\left(s_{11}, s_{12}, s_{22}\right) = -\frac{s_{11}+s_{22} + \sqrt{(s_{11}-s_{22})^2 + 4 s_{12}^{2}}}{2} ,  \label{eq:rho} \\ [1mm]
&&\upsilon :=  \upsilon\left(s_{11}, s_{12}, s_{22}\right) = -\frac{s_{11}+s_{22} - \sqrt{(s_{11}-s_{22})^2 + 4 s_{12}^{2}}}{2}.  \label{eq:upsilon} 
\end{eqnarray}
\end{theorem}
\begin{proof}
Since the proof of Theorem \ref{thm1} is similar to that of Theorem 2 in \cite{ernst2021yule}, we omit the details.
\end{proof}

\subsection{Moments of $ \sqrt{n}\theta_{n}$.}  \label{sec:numerics}
In the previous subsection, we give an exact representation for the joint mgf $\phi_{n}(s_{11}, s_{12}, s_{22})$. In this subsection, we use it to calculate the moments of $\theta_{n}$ by a method provided by Ernst, Rogers, and Zhou (see Proposition 1 in \cite{ernst2019distribution}). The proposition is as follows:
\begin{proposition}[Ernst et al. (2022)]  \label{propErnst2019}
For $m = 0,1,2,\cdots$, we have
\begin{equation}
E \left( \theta_{n}^{m} \right) = \frac{(-1)^{m}}{2^{m} \Gamma(m/2)^{2} }
\int_{0}^{\infty} \int_{0}^{\infty} s_{11}^{m/2-1} s_{22}^{m/2-1}\, \frac{\partial^{m}\phi_{n}}{\partial s_{12}^{m}}(s_{11}, 0 , s_{22})\, ds_{11} ds_{22}. 
\label{eq:allmoments} 
\end{equation}
\end{proposition}

With Proposition \ref{propErnst2019} in hand, we may now derive an explicit formula for the second moment. Before doing so, we pause to calculate the first derivative of $d_{n}(\lambda)$, which plays an important role in the formula.
A direct calculation yields
\begin{eqnarray*}
&& d_{n}'(\lambda)  \\
&=&  \left(d_{n}(\lambda) - p_n(\lambda/n) + \alpha^2\, p_{n-1}(\lambda/n)\right)/\lambda - \frac{n+1}{n} \,r_{n}(\lambda/n) + \frac{1}{n}\, r_{0}(\lambda/n) \, p_{n}(\lambda/n)  \\
&& + \alpha^2 \, r_{n-1}(\lambda/n) - \frac{\alpha^2}{n} \, r_{0}(\lambda/n) \, p_{n-1}(\lambda/n) - \frac{n+1}{n^2}\, \frac{\lambda}{\Delta(\lambda/n)} \, p_{n}(\lambda/n) \\
&& + \frac{2}{n^2} \, \lambda \,  r_{0}(\lambda/n) \, r_{n}(\lambda/n) + \frac{(n-1) \alpha^2}{n^2} \, \frac{\lambda}{\Delta(\lambda/n)}\, p_{n-1}(\lambda/n) \\
&& -\frac{2(n-1) \alpha^2}{n^3} \, \lambda \, r_{0}(\lambda/n) \, r_{n-1}(\lambda/n)   \\
&& - \frac{2(n^2 -1) \alpha}{n^3} \, \lambda \, p_{n}(\lambda/n) \, l_{1}(\lambda/n) + \frac{2(n-1) \alpha}{n^3} \, \lambda \, r_{n}(\lambda/n) l_{2}(\lambda/n)  \\
&& + \frac{2(n -2) \alpha^3}{n^2} \, \lambda \, p_{n-1}(\lambda/n) \, l_{1}(\lambda/n) - \frac{2(n-2) \alpha^3}{n^3} \, \lambda \, r_{n-1}(\lambda/n) l_{2}(\lambda/n)  \\
&& + \frac{2(n +1) \alpha^2}{n^2} \, \lambda \, p_{n-1}(\lambda/n) \, l_{1}(\lambda/n) - \frac{2(n+1) \alpha^2}{n^3} \, \lambda \, r_{n-1}(\lambda/n) l_{2}(\lambda/n)  \\
&& - \frac{2(n -1) \alpha^4}{n^2} \, \lambda \, p_{n-2}(\lambda/n) \, l_{1}(\lambda/n) + \frac{2 \alpha^4}{n^2} \, \lambda \, r_{n-2}(\lambda/n) l_{2}(\lambda/n)  \\
&& - \frac{2\alpha^{n+1} (1- \alpha)}{n^3} \, \lambda \, l_{1}(\lambda/n) + \frac{2 \alpha^{n+1} (1- \alpha)}{n^3} \, \lambda\, r_{0}(\lambda/n)\, l_2(\lambda/n) \\
&& + \frac{4 \alpha^{n+2} (1-\alpha)}{n^3} \, \frac{\lambda}{\Delta(\lambda/n)} \, l_2(\lambda/n).
\end{eqnarray*}
For $n \in \mathds{N}$, we have that
\begin{eqnarray*}
&& p_{n}(\lambda) := \frac{ \gamma_{1}^{n+1}(\lambda) - \gamma_{2}^{n+1}(\lambda) }{ \sqrt{ \Delta(\lambda) } }, \\
&& r_{n}(\lambda) := \frac{\gamma_{1}^{n+1}(\lambda)  +  \gamma_{2}^{n+1}(\lambda)}{\Delta(\lambda)},  \\ 
&& l_{1}(\lambda) := \left[\Delta(\lambda) \left( (1-\alpha)^2  - \lambda\right)\right]^{-1} ,  \\ [1mm]
&& l_{2}(\lambda) := \frac{3 \gamma_{1}(\lambda) + 3 \gamma_{2}(\lambda) + 2 \alpha}{\Delta(\lambda) \left( (1-\alpha)^2  - \lambda\right) }.
\end{eqnarray*}

With the above preparation in hand, we can now provide an explicit integral representation for the second moment of $\sqrt{n} \, \theta_n$.

\begin{theorem}  \label{thm:explicitexpressionforsecondmoment}
The second moment of $\sqrt{n}\, \theta_n$ is
\begin{equation*}
E\left[ \left(\sqrt{n}\, \theta_n \right)^2 \right] 
= \frac{n}{4} \, \int_{0}^{\infty} \int_{0}^{\infty} \frac{ \frac{d_{n}'(-\max(s_{11}, s_{22}))}{ d_{n} (-\max(s_{11}, s_{22}))} -   \frac{d_{n}'(-\min(s_{11}, s_{22}))}{ d_{n} (-\min(s_{11}, s_{22}))}}{\left[ d_{n}(-\max(s_{11}, s_{22})) d_{n}(-\min(s_{11}, s_{22})) \right]^{1/2} |s_{11}-s_{22}|} \, ds_{11} ds_{22}.
\end{equation*}
\end{theorem}
\begin{proof}
Since the proof is nearly identical to that of Theorem 3 in \cite{ernst2021yule}, we omit the details.
\end{proof}

\subsection{Numerics.}
We now turn to numerics. \textsf{Mathematica} allows us to calculate the second moment of the scaled empirical correlation $\sqrt{n} \theta_{n}$ for any given $n$ for any $ \alpha $ satisfying $|\alpha|<1$. The numerical results are summarized in Table \ref{tab:Yulenumerical}.
\begin{table}[H]
\centering
\begin{tabular}{ccccccc}
\toprule[1.5pt]
\specialrule{0em}{2pt}{2pt}
 $n$  & 10  &   20  & 30  &  40  &  50  &  60  \\
 $E \left(\sqrt{n}\theta_{n} \right)^2$  & 1.122613  & 1.068110 & 1.051453 & 1.043226 & 1.038489 & 1.035362 \\
 \hline
 \specialrule{0em}{2pt}{2pt}
 $n$  & 70  &   80  & 90  &  100  &  200  &  300  \\
 $E \left(\sqrt{n}\theta_{n} \right)^2$  & 1.033146 & 1.031493 & 1.030211 & 1.029190 & 1.024627 & 1.023118 \\
 \hline
 \specialrule{0em}{2pt}{2pt}
 $n$  & 400  &   500  & 600  &  700  &  800  &  $\infty$  \\
 $E \left(\sqrt{n}\theta_{n} \right)^2$  & 1.022367 & 1.021917 & 1.021616 & 1.021402 & 1.021242 & 1.020202 \\
\bottomrule[1.5pt]
\end{tabular}
\caption{Numerical Results of the second moment of $\sqrt{n}\theta_n$ for various values of $n$ when $\alpha=0.1$ and the AR(1) processes are independent.}
\label{tab:Yulenumerical}
\end{table}
For higher-order moments as represented in \eqref{eq:allmoments}, we can use \textsf{Mathematica} to perform symbolic high-order differentiation and two-dimensional integration. This then gives the higher moments of $\sqrt{n}\theta_n$ for all $n$. As all odd moments vanish, we need only focus on the even moments. The numerical results for higher-order moments of $\sqrt{30} \, \theta_{30}$ with $\alpha=0.05$ are summarized in Table \ref{tab:Yulenumericalhighmomentfor30}.
\begin{table}[H]
\centering
\begin{tabular}{cccccc}
\toprule[1.5pt]
 $k$  & 2  &   4  & 6  &  8   & 10 \\
 $E \left( \sqrt{30} \, \theta_{30} \right)^{k}$  & 1.038702  & 3.026394 & 11.938520 & 73.447734  & 545.793589 \\
\bottomrule[1.5pt]
\end{tabular}{}
\caption{Numerical Results of higher-order moments of $\sqrt{30}\,\theta_{30}$ when $\alpha=0.05$ and the AR(1) processes are assumed independent.}
\label{tab:Yulenumericalhighmomentfor30}
\end{table}

The numerical results of higher-order moments of $\sqrt{30}\,\theta_{30}$ enable us to approximate its density using the Legendre polynomial approximation method (see, for example, \cite{provost2005moment}). In Figure \ref{Fig.1}, we present an approximation to the density of $\sqrt{30}\, \theta_{30}$ based on its first ten moments. This approximation shows that the distribution of $\sqrt{30}\, \theta_{30}$  very closely resembles the Gaussian.

\begin{figure}[htbp] 
\centering 
\includegraphics[width=0.65\textwidth]{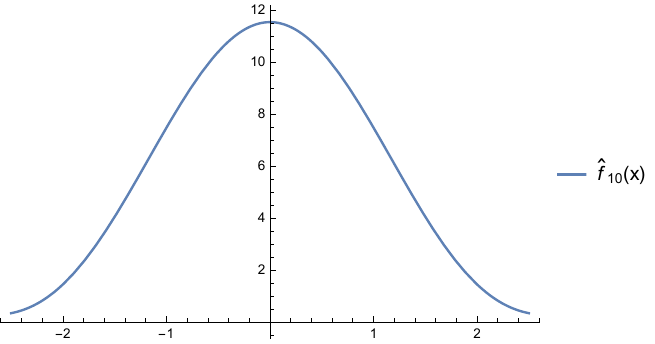} 
\caption{The 10th-order polynomial approximation to the probability density function of $\sqrt{30} \,\theta_{30} $ for $\alpha =0.5$. The AR(1) processes are assumed to be independent.} 
\label{Fig.1} 
\end{figure}

\subsection{Correlated increments.}  \label{sbsec:numericorrelatedincrements}
In the previous subsections, we have studied the distribution of the scaled empirical correlation of two independent AR(1) processes by deriving formulas for the higher-order moments. The same techniques may also be employed if the AR(1) processes are now assumed to have correlated increments. Specifically, let us assume that the pairs of the increments $(\xi_n, \eta_n)$ are i.i.d. Gaussian random vectors with mean zero and covariance matrix $$\begin{pmatrix}
1 & r \\
r & 1
\end{pmatrix},$$ where $|r| \leq 1$. We then say that these two AR(1) processes are correlated with coefficient $r$. With slight abuse of notation, we also denote by $\theta_n$ the empirical correlation of these two correlated AR(1) processes and denote by $\phi_{n}(s_{11}, s_{12}, s_{22})$ the joint moment generating function. After a similar calculation to that above, we obtain 
\begin{equation*}
\phi_{n}(s_{11}, s_{12}, s_{22})= \left( d_{n}(\rho_r) \, d_{n}(\upsilon_r)  \right)^{-1/2},
\end{equation*}
where $\rho_r$ and $\upsilon_r$ are defined as 
\begin{eqnarray*}
&&\rho_r := \rho_{r}\left(s_{11}, s_{12}, s_{22}\right) = -\frac{s_{11}+s_{22} + 2r s_{12} + \sqrt{(s_{11}-s_{22})^2 + 4 (rs_{11} + s_{12})(rs_{22} + s_{12})  }}{2} ,   \\ [1mm]
&&\upsilon_r :=  \upsilon_{r}\left(s_{11}, s_{12}, s_{22}\right) = -\frac{s_{11}+s_{22} + 2r s_{12} - \sqrt{(s_{11}-s_{22})^2 + 4 (rs_{11} + s_{12})(rs_{22} + s_{12})  }}{2}.   
\end{eqnarray*}
Together with Proposition \ref{propErnst2019}, we can derive expressions for the higher-order moments of $\sqrt{n} \,\theta_n$. Furthermore, we can similarly obtain numerical results for the higher-order moments of $\sqrt{n} \, \theta_n$ and subsequently approximate its density using \textsf{Mathematica}. The results are summarized in Table \ref{tab:Yulenumericalhighmomentfor30correlated} and Figure \ref{Fig.2}. The computation of moments beyond the $9$th order requires substantial memory capacity. Therefore, we only provide numerical results for the moments up to order $9$. Additionally, as demonstrated in Figure \ref{Fig.2}, using the first $9$ moments is effective in generating a ``good'' approximation for the density.

\begin{table}[H]
\centering
\begin{tabular}{cccccc}
\toprule[1.5pt]
\specialrule{0em}{2pt}{2pt}
 $k$  & 1  &   2  & 3 &  4 & 5    \\
 $E \left( \sqrt{30} \, \theta_{30} \right)^{k}$  & 0.538403  & 1.309724 & 1.697504  & 4.567613  & 8.285348 \\
 \hline
 \specialrule{0em}{2pt}{2pt}
 $k$  & 6  &   7  & 8  &  9 &  \\
 $E \left( \sqrt{30} \, \theta_{30} \right)^{k}$  & 24.081011 & 52.901232  & 165.222506  & 525.234538 & \\
\bottomrule[1.5pt]
\end{tabular}
\caption{Numerical results of higher-order moments of $\sqrt{30}\,\theta_{30}$ for $\alpha=0.05$. The AR(1) processes are assumed to be correlated with $r =0.1$.}
\label{tab:Yulenumericalhighmomentfor30correlated}
\end{table}

\begin{figure}[htbp] 
\centering 
\includegraphics[width=0.65\textwidth]{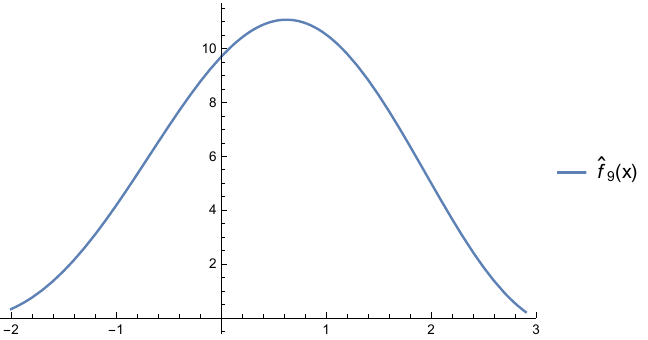} 
\caption{The 9th-order polynomial approximation to the probability density function of $\sqrt{30} \,\theta_{30}$ and $\alpha =0.5$. The AR(1) processes are assumed to be correlated with $r =0.1$.} 
\label{Fig.2} 
\end{figure}

\section{Convergence rate of $ \sqrt{n} \theta_n$.} \label{sec:mainconvergence}
In this section, we shall investigate the convergence rates of $\sqrt{n} \, \theta_n$ to the standard normal distribution. We shall develop upper bounds for the rate of convergence in both Wasserstein distance and Kolmogorov distance.

We begin by writing $\sqrt{n}\, \theta_n$ as 
\begin{equation*}
\sqrt{n}\, \theta_n = \frac{\sqrt{n} \, Z_{12}^{n}}{ \sqrt{ Z_{11}^{n}  Z_{22}^{n} }}
\end{equation*}
We first investigate the convergence rates of the numerator (after scaling) to the standard normal distribution utilizing results of Nourdin and Peccati in  \eqref{eq:malliavinboundforKolmogorov} and \eqref{eq:malliavinboundforWasserstein}. We then study convergence rates for the entire fraction. However, before studying these convergence rates, we first need to prepare several necessary estimates for the eigenvalues of the matrix $K_n$ as defined in \eqref{eq:defKn}.

\subsection{Estimates for eigenvalues of $K_n$.} \label{subsec:estimateeigenvalues}
As before, we use $\lambda_1 \geq \lambda_2 \geq \dots \geq \lambda_n \geq 0$ to denote the eigenvalues of $K_n$. In what follows, we will give four lemmas regarding estimates of power sum symmetric polynomials in $\lambda_1, \lambda_2, \dots, \lambda_n$ and of the product of these positive eigenvalues. The first two lemmas, Lemma \ref{lm:lambda1} and \ref{lm:lambda2}, involve long but direct calculations. In Lemma \ref{lm:lambda4}, we apply Young’s convolution inequality to simplify the calculation. The explicit computation of the product of the positive eigenvalues of $K_n$ in Lemma \ref{lm:productlambda} is due to the explicit expression of $d_{n}$. The proofs of the last three lemmas are much more involved and are deferred to Appendix B. We now continue with Lemma \ref{lm:lambda1} below.

\begin{lemma} \label{lm:lambda1}
Let
\begin{equation*}
\kappa_{1}(n) = - \frac{\alpha^{2}(1- \alpha^{2n}) + (1+ \alpha)^2}{ (1- \alpha^2)^2} + \frac{2\alpha(1+\alpha)(1- \alpha^{n}) - \alpha^2 (1- \alpha^{2n})}{n \, (1-\alpha)^2 (1- \alpha^2)}.
\end{equation*}
Then
\begin{equation*}
\sum_{k=1}^{n} \lambda_k
= \frac{1}{1- \alpha^2} +  \frac{\kappa_{1}(n)}{n}.
\end{equation*}
In particular, for every $n \in \mathds{N}_{+}$,
\begin{equation*}
|\kappa_{1}(n)| \leq \frac{\alpha^2 + (1 + \alpha)^2}{(1- \alpha^2)^2} + \frac{4 |\alpha| + 5 \alpha^2}{(1- \alpha)^2 (1- \alpha^2)} =: C_{1}(\alpha).
\end{equation*}
\end{lemma}
\begin{proof}
We begin by proving the first assertion.
\begin{eqnarray*}
&& \sum_{k=1}^{n} \lambda_k = \mathrm{tr}(K_n)
= \sum_{k=1}^{n}\left[ \frac{1}{n} \, \frac{1- \alpha^{2k}}{1- \alpha^2}  - \frac{1}{n^2} \, \frac{(1-\alpha^k)^2}{(1-\alpha)^2} \right] \\
&=& \frac{1}{n} \, \sum_{k=1}^{n} \frac{1- \alpha^{2k}}{1- \alpha^2} - \frac{1}{n^2} \, \sum_{k=1}^{n} \frac{1- 2 \alpha^k + \alpha^{2k}}{(1-\alpha)^2} \\
&=& \frac{1}{1- \alpha^2} - \frac{\alpha^2(1- \alpha^{2n})}{n\, (1-\alpha^2)^{2}} - \frac{1}{n \, (1-\alpha)^2} + \frac{2 \alpha(1-\alpha^n)}{n^2 \, (1-\alpha)^3} - \frac{\alpha^2 (1- \alpha^{2n})}{n^2 \,(1-\alpha)^2  (1- \alpha^2)} \\ [1mm]
&=& \frac{1}{1- \alpha^2} - \frac{\alpha^{2}(1- \alpha^{2n}) + (1+ \alpha)^2}{n \, (1- \alpha^2)^2} + \frac{2\alpha(1+\alpha)(1- \alpha^{n}) - \alpha^2 (1- \alpha^{2n})}{n^2 \, (1-\alpha)^2 (1- \alpha^2)}.
\end{eqnarray*}
The second assertion follows immediately by noting that $ |\alpha^2 (1- \alpha^{2n}) | \leq \alpha^2 $ and $| 2 \alpha ( 1+ \alpha)(1- \alpha^n) | \leq 4 |\alpha| (1+ \alpha) $.
\end{proof}

We now continue with Lemma \ref{lm:lambda2} below.

\begin{lemma} \label{lm:lambda2}
There exists a bounded function $\kappa_{2}(n)$ on $\mathds{N}_+$ such that
\begin{equation*}
\sum_{k=1}^{n} \lambda_{k}^{2} = \frac{1+ \alpha^2}{n \, (1- \alpha^2)^3} + \frac{\kappa_{2}(n)}{n^2}.
\end{equation*}
In fact,
\begin{eqnarray*}
|\kappa_{2}(n)| &\leq&  \frac{4}{(1- \alpha^2)^3} \, \sup_{n \in \mathds{N}_{+}} n \alpha^{2n+2} + \frac{4 \alpha^2 + \alpha^4}{(1- \alpha^2)^4}  + \frac{8}{(1- \alpha)^2 ( 1- \alpha^2)}  \\
&& + \frac{16|\alpha|}{(1- |\alpha|)(1- \alpha)^2 ( 1- \alpha^2)} + \frac{16}{(1-\alpha)^4} = : C_2(\alpha).
\end{eqnarray*}
\end{lemma}
\begin{proof}
See Appendix B.
\end{proof}

\begin{lemma}  \label{lm:lambda4}
Let
\begin{equation*}
C_3(\alpha) := \left( \frac{2^{4/3}}{ (1- \alpha^2)^{4/3} ( 1- |\alpha|^{4/3}) }  + \frac{16}{(1-\alpha)^{8/3}}\right)^3.
\end{equation*}
Then
\begin{equation*}
\sum_{k=1}^{n} \lambda_{k}^{4} \leq C_3 (\alpha) \, n^{-3}.
\end{equation*}
\end{lemma}
\begin{proof}
See Appendix B.
\end{proof}

\begin{remark} \label{remark:boundlambda}
It follows directly from Lemma \ref{lm:lambda4} that $\lambda_k \leq C_{3}(\alpha)^{1/4} \, n^{-3/4}$ for $k=1,2 \dots, n$.
\end{remark}

It follows easily from Remark \ref{remark:semidefinite} that $\lambda_1 \geq \lambda_2 \geq \dots \geq \lambda_{n-1} >0$ but $\lambda_n =0$. The following Lemma gives an estimate for the product of $\lambda_1$, $\lambda_2$, $\dots$, $\lambda_{n-1}$.

\begin{lemma} \label{lm:productlambda}
For $n \geq 2$,
\begin{equation*}
\prod_{k=1}^{n-1} \lambda_k = n^{-(n-1)} \left( 1+ \alpha^2 \frac{n-1}{n} - 2\alpha \frac{n-1}{n} \right).
\end{equation*}
Further,
\begin{equation*}
(n-1)  \sqrt[n-1]{ \lambda_1 \lambda_2 \dots \lambda_{n-1}}    \geq \frac{1}{4}.
\end{equation*}
\end{lemma}
\begin{proof}
See Appendix B.
\end{proof}

\subsection{Convergence of $\sqrt{n} Z_{12}^{n}$.} \label{subsec:nominatorconvergence}
With the above estimates in hand, we are ready to investigate the convergence rate of $\sqrt{n} \theta_n$ (after scaling) to the normal distribution. The following theorem reveals the convergence rate of the numerator $\sqrt{n} Z_{12}^{n}$ to the standard normal distribution.

\begin{theorem} \label{thm:nominator}
Let
\begin{equation*}
C_{4}(\alpha) := \frac{(1- \alpha^2)^3}{1 + \alpha^2} \, \sqrt{\left(C_{2}(\alpha)\right)^2 + C_{3}(\alpha)}.
\end{equation*}
We have
\begin{eqnarray*}
&& d_{Kol}\left( \sqrt{\frac{(1- \alpha^2)^3}{1+\alpha^2}} \sqrt{n} \, Z_{12}^{n}, \,   \mathcal{N}(0,1) \right) \leq  \frac{C_{4}(\alpha)}{\sqrt{n}},  \\ [1mm]
&& d_{W}\left( \sqrt{\frac{(1- \alpha^2)^3}{1+\alpha^2}} \sqrt{n} \, Z_{12}^{n}, \,   \mathcal{N}(0,1) \right) \leq   \frac{ \sqrt{2/\pi} \,C_{4}(\alpha)}{\sqrt{n}}.
\end{eqnarray*}
In particular, $\sqrt{n} Z_{12}^{n}$ converges in distribution to $\mathcal{N}\left( 0, \frac{1+\alpha^2}{(1- \alpha^2)^3}\right)$ as $n$ tends to $\infty$.
\end{theorem}
\begin{proof}
Recalling the representation \eqref{eq:decompositionforZ12} in Section \ref{sec:jointmgf}, we have
\begin{equation*}
Z_{12}^{n} = \sum_{k=1}^{n} \lambda_k W_k V_k.
\end{equation*}
There exist orthonormal functions $e_1, e_2, \dots, e_n$, $f_1, f_2, \dots, f_n$ on $\mathbb{R}_{+}$ such that
\begin{equation*}
Z_{12}^{n} \overset{d}{=} \sum_{k=1}^{n} \lambda_k \, \mathcal{I}_{1}(e_k) \, \mathcal{I}_{1}(f_k),
\end{equation*}
where $\mathcal{I}_{1}(h)$ as defined in Section \ref{sec:preliminaries} is the first-order Wiener integral of $h$. For simplicity, we use $F_n$ to denote $\sum_{k=1}^{n} \lambda_k \, \mathcal{I}_{1}(e_k) \, \mathcal{I}_{1}(f_k)$. Then, it follows by applying the product formula in \eqref{eq:productformulaforWienerintegral} that $F_n \in \mathcal{H}_{2}$. Note that
\begin{eqnarray*}
d_{Kol}\left( \sqrt{\frac{(1- \alpha^2)^3}{1+\alpha^2}} \sqrt{n} \, Z_{12}^{n}, \,   \mathcal{N}(0,1) \right)
= d_{Kol}\left( \sqrt{\frac{(1- \alpha^2)^3}{1+\alpha^2}} \sqrt{n} \, F_{n}, \,   \mathcal{N}(0,1) \right), \\ [1mm]
d_{W}\left( \sqrt{\frac{(1- \alpha^2)^3}{1+\alpha^2}} \sqrt{n} \, Z_{12}^{n}, \,   \mathcal{N}(0,1) \right)
= d_{W}\left( \sqrt{\frac{(1- \alpha^2)^3}{1+\alpha^2}} \sqrt{n} \, F_{n}, \,   \mathcal{N}(0,1) \right). 
\end{eqnarray*}
If we can now prove that
\begin{eqnarray}
&&\sqrt{E\left[ \left(1- \frac{1}{2} \left\lVert D \left(\sqrt{\frac{(1- \alpha^2)^3}{1+\alpha^2}} \sqrt{n} F_n\right) \right\rVert_{\mathcal{H}}^{2} \right)^2 \right]}   \notag \\  [2mm]
&=& \sqrt{E\left[ \left(1- \frac{n ( 1- \alpha^2)^3}{2(1+ \alpha^2)} \lVert D  F_n\rVert_{\mathcal{H}}^{2} \right)^2 \right]}
 \leq \frac{C_{4}(\alpha)}{\sqrt{n}}, \label{eq:MSestimate}
\end{eqnarray}
then together with the estimates \eqref{eq:malliavinboundforKolmogorov} and \eqref{eq:malliavinboundforWasserstein} in Section \ref{sec:preliminaries}, the desired result will follow. Thus, we proceed to prove \eqref{eq:MSestimate} below.

By the definition of Malliavin derivative, we have
\begin{equation*}
D F_n = \sum_{k=1}^{n} \left( \lambda_k \, \mathcal{I}_{1}(e_k) f_k  +  \lambda_k \, \mathcal{I}_{1}(f_k) e_k \right).
\end{equation*}
Note that $\langle e_k, f_j \rangle_{\mathcal{H}} =0$ and $\langle e_k, e_j \rangle_{\mathcal{H}} = \langle f_k, f_j \rangle_{\mathcal{H}} = \sigma_{kj}$, where $\sigma_{kj}$ equals $ 1$ if $k=j$ and $0$ otherwise. It directly follows that
\begin{eqnarray*}
\lVert D F_n \rVert_{\mathcal{H}}^{2}
&=& \sum_{j,k=1}^{n} \langle \lambda_k \, \mathcal{I}_{1}(e_k) f_k  +  \lambda_k \, \mathcal{I}_{1}(f_k) e_k, \lambda_j \, \mathcal{I}_{1}(e_j) f_j  +  \lambda_j \, \mathcal{I}_{1}(f_j) e_j \rangle_{\mathcal{H}} \\
&=& \sum_{k=1}^{n} \left( \lambda_k^2 \ \mathcal{I}_{1}^{2}(e_k) +  \lambda_k^2 \ \mathcal{I}_{1}^{2}(f_k) \right).
\end{eqnarray*}
Thus,
\begin{eqnarray}
&& 1- \frac{n ( 1- \alpha^2)^3}{2(1+ \alpha^2)} \lVert D  F_n\rVert_{\mathcal{H}}^{2}
= 1- \frac{n ( 1- \alpha^2)^3}{2(1+ \alpha^2)} \sum_{k=1}^{n} \left( \lambda_k^2 \ \mathcal{I}_{1}^{2}(e_k) +  \lambda_k^2 \ \mathcal{I}_{1}^{2}(f_k) \right) \notag\\
&=& 1- \frac{n ( 1- \alpha^2)^3}{(1+ \alpha^2)} \sum_{k=1}^{n} \lambda_{k}^{2} -  \frac{n ( 1- \alpha^2)^3}{2(1+ \alpha^2)} \sum_{k=1}^{n}  \lambda_k^2 \ \left(\mathcal{I}_{1}^{2}(e_k) -1\right) \notag \\
&& - \frac{n ( 1- \alpha^2)^3}{2(1+ \alpha^2)} \sum_{k=1}^{n}  \lambda_k^2 \ \left(\mathcal{I}_{1}^{2}(f_k) -1\right). \label{eq:DFvariance}
\end{eqnarray}
Noting that $\mathcal{I}(e_1), \mathcal{I}(e_2), \dots, \mathcal{I}(e_n)$, $\mathcal{I}(f_1), \mathcal{I}(f_2), \dots, \mathcal{I}(f_n)$ are independent standard normal random variables, we have
\begin{eqnarray*}
&& E\left[ \mathcal{I}_{1}^{2}(e_k) -1 \right] = E\left[ \mathcal{I}_{1}^{2}(f_k) -1 \right]=0,  \\
&& E\left[ \left( \mathcal{I}_{1}^{2}(e_k) -1 \right) \left( \mathcal{I}_{1}^{2}(f_k) -1 \right) \right] =0, \\
&& E\left[ \left( \mathcal{I}_{1}^{2}(e_k) -1 \right) \left( \mathcal{I}_{1}^{2}(e_j) -1 \right)\right] = 2 \sigma_{kj}, \\
&& E\left[ \left( \mathcal{I}_{1}^{2}(f_k) -1 \right) \left( \mathcal{I}_{1}^{2}(f_j) -1 \right)\right] = 2 \sigma_{kj}.
\end{eqnarray*}
Together with \eqref{eq:DFvariance}, routine but lengthy calculation gives that
\begin{eqnarray*}
&& E\left[ \left(1- \frac{n ( 1- \alpha^2)^3}{2(1+ \alpha^2)} \lVert D  F_n\rVert_{\mathcal{H}}^{2} \right)^2 \right]  \\
&=& \left( 1- \frac{n (1- \alpha^2)^3}{(1+ \alpha^2)} \, \sum_{k=1}^{n} \lambda_{k}^{2} \right)^2
+ \frac{n^2 (1- \alpha^2)^{6}}{(1+ \alpha^2)^2} \, \sum_{k=1}^{n} \lambda_{k}^{4} \\ [1mm]
&=& \left(\frac{\kappa_{2}(n) (1-\alpha^2)^3 }{n (1+ \alpha^2)}\right)^{2}
+ \frac{n^2 (1- \alpha^2)^{6}}{(1+ \alpha^2)^2} \, \sum_{k=1}^{n} \lambda_{k}^{4}  \\ [1mm]
&\leq& \frac{(1- \alpha^2)^{6}}{(1+ \alpha^2)^2} \, \frac{\left(C_{2}(\alpha)\right)^2}{n^2}
 + \frac{(1- \alpha^2)^{6}}{(1+ \alpha^2)^2} \, \frac{C_{3}(\alpha)}{n}  \\ [1mm]
& \leq& \frac{(1- \alpha^2)^{6}}{(1+ \alpha^2)^2}\, \left( \left( C_{2}(\alpha)\right)^2 + C_{3}(\alpha) \right) \times n^{-1},
\end{eqnarray*}
where the second equality follows by Lemma \ref{lm:lambda2} and the first inequality follows by Lemma \ref{lm:lambda4} combined with the boundedness of $\kappa_{2}(n)$. Then \eqref{eq:MSestimate} follows immediately from the last display.
\end{proof}

\subsection{Convergence in Wasserstein distance.}  \label{subsec:convergencewasserstein}
In this subsection, we will derive an upper bound for the Wasserstein distance between the distribution of $\sqrt{n} \theta_n$ (after scaling) and the standard normal distribution. This result relies on two preparatory lemmas. The first, Lemma \ref{lm:boundedness4moment}, proves that the forth moment of $\sqrt{n} Z_{12}^{n}$ is bounded uniformly in $n$. The second, Lemma \ref{lm:inversesecondmoment}, shows that $Z_{11}^{n}$ and $Z_{22}^{n}$ have inverse second moments uniformly bounded in $n$.

\begin{lemma} \label{lm:boundedness4moment}
Let
\begin{equation*}
C_{5}( \alpha) := 3^4 \left( \frac{1 + \alpha^2}{(1- \alpha^2)^3} + \frac{C_{2}( \alpha )}{10} \right)^2.
\end{equation*}
For $n \geq 10$, we have
\begin{equation*}
E\left[ \left( \sqrt{n} \, Z_{12}^{n} \right)^{4} \right] \leq C_{5}(\alpha).
\end{equation*}
\end{lemma}
\begin{proof}
Recall that $F_n$, as defined in the proof of Theorem \ref{thm:nominator}, has the same distribution as $Z_{12}^{n}$ and belongs to the second Wiener chaos $\mathcal{H}_{2}$. Applying the hypercontractivity property from \eqref{eq:hypercontractivitypropertyforchaos} with $q=2$, $p=2$ and $r=4$, we have
\begin{equation*}
\left\{ E\left[ \left( \sqrt{n} F_n \right)^4 \right] \right\}^{1/4} \leq 3 \left\{ E\left[ \left( \sqrt{n} F_n \right)^2 \right] \right\}^{1/2}.
\end{equation*}
Since $Z_{12}^{n}$ and $F_n$ are identically distributed,
\begin{equation}
\left\{ E\left[ \left( \sqrt{n} Z_{12}^{n} \right)^4 \right] \right\}^{1/4} \leq 3 \left\{ E\left[ \left( \sqrt{n} Z_{12}^{n} \right)^2 \right] \right\}^{1/2}. \label{eq:hypercontractivityinequality}
\end{equation}
Note that
\begin{eqnarray*}
&& E\left[ \left( \sqrt{n} Z_{12}^{n} \right)^2 \right]
= E \left[ n \left( \sum_{k=1}^{n} \lambda_{k} W_{k} V_{k} \right)^2  \right]
= E \left[ n \sum_{j,k=1}^{n} \lambda_k \lambda_j W_{k} V_{k} W_{j} V_{j} \right]  \\
&=& n \sum_{k=1}^{n} \lambda_{k}^{2}
 = \frac{1+ \alpha^2}{(1- \alpha^2)^3} + \frac{\kappa_{2}(n)}{n}
\leq \frac{1+ \alpha^2}{(1- \alpha^2)^3} + \frac{C_{2}(\alpha)}{10}.
\end{eqnarray*}
Together with \eqref{eq:hypercontractivityinequality}, the desired result follows.
\end{proof}

\begin{lemma}  \label{lm:inversesecondmoment}
For $ n \geq 10$, we have
\begin{equation*}
E\left[ \left( Z_{11}^{n} \right)^{-2} \right] = E\left[ \left( Z_{22}^{n} \right)^{-2} \right]
\leq 16 \left( 2 \sqrt{2/ \pi } +1\right)^{8}.
\end{equation*}
\end{lemma}
\begin{proof}
Since $Z_{11}^{n}$ and $Z_{22}^{n}$ are identically distributed, it suffices to prove that $E\left[ \left( Z_{11}^{n} \right)^{-2} \right] \leq 16 \left( 2 \sqrt{2/ \pi } +1\right)^{8} $. Noting that $\lambda_n =0$, we have that
\begin{eqnarray*}
&& Z_{11}^{n} = \sum_{k=1}^{n-1} \lambda_{k} W_{k}^{2} \geq (n-1) \sqrt[n-1]{(\lambda_1 W_1^2) (\lambda_2 W_2^2) \dots ( \lambda_{n-1} W_{n-1}^{2} )  } \\
&=& (n-1) \sqrt[n-1]{\lambda_1 \lambda_2 \dots \lambda_{n-1}} \, \prod_{k=1}^{n-1} W_{k}^{\frac{2}{n-1}} \geq \frac{1}{4} \, \prod_{k=1}^{n-1} W_{k}^{\frac{2}{n-1}},
\end{eqnarray*}
where the first inequality follows by inequality of arithmetic and geometric means and the second inequality follows by Lemma \ref{lm:productlambda}. Invoking the fact that $W_1, W_2, \dots, W_{n-1}$ are independent and identically distributed, we have that
\begin{eqnarray*}
E\left[ \left( Z_{11}^{n} \right)^{-2} \right] \leq 16 \, E\left[ \prod_{k=1}^{n-1} W_{k}^{- \frac{4}{n-1}} \right]
= 16 \prod_{k=1}^{n-1} E\left[ W_{k}^{- \frac{4}{n-1}} \right] = 16 \, \left\{ E\left[ W_{1}^{- \frac{4}{n-1}} \right] \right\}^{n-1}.
\end{eqnarray*}
Applying H\"{o}lder's inequality yields
\begin{equation*}
E\left[ W_{1}^{- \frac{4}{n-1}} \right] \leq \left\{ E\left[ |W_1|^{-\frac{1}{2}} \right] \right\}^{\frac{8}{n-1}}.
\end{equation*}
Thus,
\begin{equation*}
E\left[ \left( Z_{11}^{n} \right)^{-2} \right] \leq 16 \left\{ E\left[ |W_1|^{-\frac{1}{2}} \right] \right\}^{8}.
\end{equation*}
Noting that
\begin{eqnarray*}
&& E\left[ |W_1|^{-\frac{1}{2}} \right]
= 2 \int_{0}^{\infty} x^{-\frac{1}{2}} \frac{1}{\sqrt{2\pi}} \, e^{-\frac{x^2}{2}} \, dx  \\
&=& 2 \int_{0}^{1} x^{-\frac{1}{2}} \frac{1}{\sqrt{2\pi}} \, e^{-\frac{x^2}{2}} \, dx
+ 2 \int_{1}^{\infty} x^{-\frac{1}{2}} \frac{1}{\sqrt{2\pi}} \, e^{-\frac{x^2}{2}} \, dx  \\
&\leq& \sqrt{\frac{2}{\pi}} \int_{0}^{1} x^{-\frac{1}{2}} \, dx + 2 \int_{1}^{\infty} \frac{1}{\sqrt{2 \pi}} \, e^{-\frac{x^2}{2}} \, dx \\
&\leq& 2 \sqrt{\frac{2}{\pi}} + 1,
\end{eqnarray*}
the desired result follows.
\end{proof}

With the above preparation in hand, we are ready to derive the upper bound for the Wasserstein distance between the distribution of $\sqrt{n} \theta_n$ (after scaling) and the standard normal distribution.

\begin{theorem}
For $n \geq 10$, we have
\begin{equation*}
d_{W}\left( \sqrt{\frac{1-\alpha^2}{1+ \alpha^2}} \sqrt{n} \theta_n , \, \mathcal{N}(0,1) \right) \leq \frac{\sqrt{2/\pi} \, C_{4}(\alpha) + C_{6}(\alpha)}{\sqrt{n}},
\end{equation*}
where $C_{4}(\alpha)$ is defined in Theorem \ref{thm:nominator} and
\begin{eqnarray*}
C_{6}(\alpha) &:=& \left(4 \left(2 \sqrt{2/\pi} + 1 \right)^{4} \, \sqrt{ \frac{1- \alpha^2}{1+ \alpha^2} }  +  2 \left(2 \sqrt{2/\pi} + 1 \right)^{2} \, \frac{1- \alpha^2}{\sqrt{1+ \alpha^2}} \right)\, \left( C_{5}(\alpha) \right)^{\frac{1}{4}} \\
&& \times \left[ \frac{2(1+\alpha^2)}{1-\alpha^2} + \frac{ (1- \alpha^2)^2 \left(C_{1}(\alpha)\right)^2}{10} + \frac{(1-\alpha^2)^2 C_2(\alpha) }{5} \right]^{\frac{1}{2}} .
\end{eqnarray*}
In particular, $\sqrt{n} \theta_n$ converges in distribution to $\mathcal{N}\left( 0, (1+ \alpha^2)/(1- \alpha^2) \right)$ as $n \rightarrow \infty$.
\end{theorem}
\begin{proof}
For simplicity, we use $\sqrt{n} \, \widetilde{\theta}_{n}$ to denote $ \sqrt{ \frac{(1-\alpha^2)^3}{1 + \alpha^2}} \sqrt{n} Z_{12}^{n} $. It follows by Lemma \ref{lm:boundedness4moment} that, for $n\geq 10$,
\begin{equation}
E\left[ \left( \sqrt{n} \, \widetilde{\theta}_{n} \right)^{4} \right] \leq \frac{(1- \alpha^2)^{6}}{(1+ \alpha^2)^2} \, C_{5}(\alpha). \label{eq:4boundtildetheta}
\end{equation}
By Theorem \ref{thm:nominator} and the triangle inequality, it suffices to prove that
\begin{equation}
d_{W} \left( \sqrt{\frac{1-\alpha^2}{1+ \alpha^2}} \sqrt{n} \theta_n , \, \sqrt{n} \, \widetilde{\theta}_{n} \right) \leq \frac{C_{6}(\alpha)}{\sqrt{n}}. 
\label{eq:dwnominator}
\end{equation}
Note that
\begin{equation*}
d_{W} \left( \sqrt{\frac{1-\alpha^2}{1+ \alpha^2}} \sqrt{n} \theta_n , \, \sqrt{n} \, \widetilde{\theta}_{n} \right)
= \sup_{f \in \mathrm{Lip}(1)} \left| Ef\left( \sqrt{\frac{1-\alpha^2}{1+ \alpha^2}} \sqrt{n} \theta_n \right) - Ef\left( \sqrt{n} \, \widetilde{\theta}_{n} \right) \right|.
\end{equation*}
For every $f \in \mathrm{Lip}(1)$, and every pair of integrable random variables $(X,Y)$ on the same probability space, we have that
\begin{equation*}
| Ef(X) - Ef(Y) | \leq E | f(X) - f(Y) | \leq E|X-Y|.
\end{equation*}
We now take the supremum over $f \in \mathrm{Lip}(1)$ on both sides of the above display with $X$, $Y$ replaced with $ \sqrt{(1-\alpha^2)/(1+ \alpha^2)} \sqrt{n} \theta_n $ and $ \sqrt{n} \, \widetilde{\theta}_{n} $, respectively. This yields
\begin{equation*}
d_{W} \left( \sqrt{\frac{1-\alpha^2}{1+ \alpha^2}} \sqrt{n} \theta_n , \, \sqrt{n} \, \widetilde{\theta}_{n} \right)
\leq E \left|  \sqrt{\frac{1-\alpha^2}{1+ \alpha^2}} \sqrt{n} \theta_n - \sqrt{n} \, \widetilde{\theta}_{n} \right|.
\end{equation*}
We now only need to bound the expectation of $ | \sqrt{(1-\alpha^2)/(1+ \alpha^2)} \sqrt{n} \theta_n - \sqrt{n} \, \widetilde{\theta}_{n}  | $. Note that
\begin{eqnarray*}
&& \sqrt{\frac{1-\alpha^2}{1+ \alpha^2}} \sqrt{n} \theta_n -   \sqrt{n} \, \widetilde{\theta}_{n}
= \frac{\sqrt{n} \, \widetilde{\theta}_{n}}{\sqrt{ (1- \alpha^2) Z_{11}^{n} \times (1- \alpha^2) Z_{22}^{n} }} - \sqrt{n} \, \widetilde{\theta}_{n}  \\
&=& \frac{\sqrt{n} \, \widetilde{\theta}_{n}}{\sqrt{ (1- \alpha^2) Z_{11}^{n} \times (1- \alpha^2) Z_{22}^{n} }} \, \left( 1- \sqrt{ (1- \alpha^2) Z_{11}^{n} \times (1- \alpha^2) Z_{22}^{n} }\right)  \\
&=& \frac{\sqrt{n} \, \widetilde{\theta}_{n}}{\sqrt{ (1- \alpha^2) Z_{11}^{n} \times (1- \alpha^2) Z_{22}^{n} }} \,  \left( 1- \sqrt{(1-\alpha^2) Z_{11}^{n}} \right)   \\
&& + \frac{\sqrt{n} \, \widetilde{\theta}_{n}}{\sqrt{ (1- \alpha^2) Z_{11}^{n} \times (1- \alpha^2) Z_{22}^{n} }} \,  \sqrt{(1-\alpha^2) Z_{11}^{n} } \left( 1- \sqrt{(1-\alpha^2) Z_{22}^{n} } \right)  \\
&=& \frac{\sqrt{n} \, \widetilde{\theta}_{n}}{\sqrt{ (1- \alpha^2) Z_{11}^{n} \times (1- \alpha^2) Z_{22}^{n} }} \,  \left( 1- \sqrt{(1-\alpha^2) Z_{11}^{n}} \right)  \\
&& + \frac{\sqrt{n} \, \widetilde{\theta}_{n}}{\sqrt{ (1- \alpha^2) Z_{22}^{n} }} \, \left( 1- \sqrt{(1-\alpha^2) Z_{22}^{n} } \right).
\end{eqnarray*}
Then,
\begin{eqnarray}
&& \left| \sqrt{\frac{1-\alpha^2}{1+ \alpha^2}} \sqrt{n} \theta_n -   \sqrt{n} \, \widetilde{\theta}_{n} \right|  \notag \\
&\leq& \frac{ \left| \sqrt{n} \, \widetilde{\theta}_{n} \right| }{\sqrt{ (1- \alpha^2) Z_{11}^{n} \times (1- \alpha^2) Z_{22}^{n} }} \, \left| 1- \sqrt{(1-\alpha^2) Z_{11}^{n}}  \right|  \notag \\ [1mm]
&& + \frac{ \left|\sqrt{n} \, \widetilde{\theta}_{n} \right| }{\sqrt{ (1- \alpha^2) Z_{22}^{n} }} \, \left| 1- \sqrt{(1-\alpha^2) Z_{22}^{n}}  \right| \notag \\
&\leq& \frac{ \left| \sqrt{n} \, \widetilde{\theta}_{n} \right| }{\sqrt{ (1- \alpha^2) Z_{11}^{n} \times (1- \alpha^2) Z_{22}^{n} }} \, \left| 1- (1-\alpha^2) Z_{11}^{n}  \right|  \notag \\ [1mm]
&& + \frac{ \left|\sqrt{n} \, \widetilde{\theta}_{n} \right| }{\sqrt{ (1- \alpha^2) Z_{22}^{n} }} \, \left| 1- (1-\alpha^2) Z_{22}^{n}  \right|, \label{eq:boundthedifference}
\end{eqnarray}
where in the last inequality we have applied the inequality $|1- x| \leq |1- x^2|$ for $x \geq 0$. Note that
\begin{equation*}
E\left[ Z_{11}^{n} \right] = E\left[ \sum_{k=1}^{n} \lambda_k W_k^2  \right] = \sum_{k=1}^{n} \lambda_k,
\end{equation*}
and
\begin{eqnarray*}
&& E\left[ \left( Z_{11}^{n}\right)^2  \right] = E\left[ \left( \sum_{k=1}^{n} \lambda_k W_k^2 \right)^2 \right] = E\left[ \sum_{j,k=1}^{n} \lambda_k \lambda_j W_k^2 \, W_j^2 \right]   \\
&=& \sum_{k=1}^{n} \lambda_k^2 E\left[W_{k}^{4} \right] + \sum_{k \neq j} \lambda_k \lambda_j E\left[ W_k^2 \, W_j^2 \right]
= 3 \sum_{k=1}^{n} \lambda_k^2 + \sum_{k \neq j} \lambda_k \lambda_j \\
&=& \left(\sum_{k=1}^{n}  \lambda_k \right)^2 + 2 \sum_{k=1}^{n} \lambda_k^2.
\end{eqnarray*}
Then
\begin{eqnarray}
&& E\left[ \left( 1- (1-\alpha^2) Z_{11}^{n} \right)^2 \right]  \notag \\  
&=& 1 - 2 (1- \alpha^2) E\left[ Z_{11}^{n} \right] + (1-\alpha^2)^{2} E\left[ \left( Z_{11}^{n} \right)^2 \right]  \notag \\
&=& 1 - 2 (1- \alpha^2) \sum_{k=1}^{n} \lambda_k + (1- \alpha^2)^2 \left(\sum_{k=1}^{n}  \lambda_k \right)^2 + 2 (1- \alpha^2)^2 \sum_{k=1}^{n} \lambda_k^2 \notag \\
&=& \left( 1 - (1-\alpha^2)  \sum_{k=1}^{n} \lambda_k \right)^2 +  2 (1- \alpha^2)^2 \sum_{k=1}^{n} \lambda_k^2  \notag \\
&=& \frac{(1- \alpha^2)^2 (\kappa_{1}(n))^2}{n^2} + \frac{2(1+\alpha^2)}{1-\alpha^2} \, \frac{1}{n} + \frac{2 (1- \alpha^2)^2 \kappa_2(n)}{n^2}  \notag \\ [1mm]
&\leq& \left[ \frac{2(1+\alpha^2)}{1-\alpha^2} + \frac{ (1- \alpha^2)^2 \left(C_{1}(\alpha)\right)^2}{10} + \frac{(1-\alpha^2)^2 C_2(\alpha) }{5} \right] \times n^{-1},  \label{eq:bounddenominatorn}
\end{eqnarray}
where in the fourth equality we have invoked Lemma \ref{lm:lambda1} and Lemma \ref{lm:lambda2}. Similarly,
\begin{equation*}
E\left[ \left( 1- (1-\alpha^2) Z_{22}^{n} \right)^2 \right] \leq 
\left[ \frac{2(1+\alpha^2)}{1-\alpha^2} + \frac{ (1- \alpha^2)^2 \left(C_{1}(\alpha)\right)^2}{10} + \frac{(1-\alpha^2)^2 C_2(\alpha) }{5} \right] \times n^{-1}.
\end{equation*}
Applying H\"{o}lder's inequality twice yields
\begin{eqnarray}
&& E\left[ \frac{ \left| \sqrt{n} \, \widetilde{\theta}_{n} \right| }{\sqrt{ (1- \alpha^2) Z_{11}^{n} \times (1- \alpha^2) Z_{22}^{n} }} \, \left| 1- (1-\alpha^2) Z_{11}^{n}  \right|  \right]  \notag \\
&\leq& \left\{ E \left[\left(   \frac{1}{\sqrt{ (1- \alpha^2) Z_{11}^{n} \times (1- \alpha^2) Z_{22}^{n} }} \right)^4\right]  \right\}^{\frac{1}{4}}
\, \left\{ E\left[ \left( \sqrt{n} \, \widetilde{\theta}_{n} \right) \right]^{4} \right\}^{\frac{1}{4}} \notag  \\
&& \times \left\{ E\left[ \left( 1- (1-\alpha^2) Z_{11}^{n} \right)^2 \right] \right\}^{\frac{1}{2}} \notag \\
&=& \frac{1}{1- \alpha^2} \left\{ E\left[ \left( Z_{11}^{n} \right)^{-2} \right] \right\}^{\frac{1}{4}} \left\{ E\left[ \left( Z_{22}^{n} \right)^{-2} \right] \right\}^{\frac{1}{4}}
\, \left\{ E\left[ \left( \sqrt{n} \, \widetilde{\theta}_{n} \right) \right]^{4} \right\}^{\frac{1}{4}} \notag \\
&& \times \left\{ E\left[ \left( 1- (1-\alpha^2) Z_{11}^{n} \right)^2 \right] \right\}^{\frac{1}{2}} \notag  \\
&\leq& 4 \left(2 \sqrt{2/\pi} + 1 \right)^{4} \, \sqrt{ \frac{1- \alpha^2}{1+ \alpha^2} } \, \left( C_{5}(\alpha) \right)^{\frac{1}{4}} \notag \\
&& \times \left[ \frac{2(1+\alpha^2)}{1-\alpha^2} + \frac{ (1- \alpha^2)^2 \left(C_{1}(\alpha)\right)^2}{10} + \frac{(1-\alpha^2)^2 C_2(\alpha) }{5} \right]^{\frac{1}{2}} \times n^{-\frac{1}{2}}, \label{eq:expectationdiffterm1}
\end{eqnarray}
where the first equality follows by noting that $Z_{11}^{n}$ and $Z_{22}^{n}$ are independent and the last inequality follows by combining \eqref{eq:4boundtildetheta}, \eqref{eq:bounddenominatorn} and Lemma \ref{lm:inversesecondmoment}. Similarly,
\begin{eqnarray}
&& E\left[ \frac{ \left|\sqrt{n} \, \widetilde{\theta}_{n} \right| }{\sqrt{ (1- \alpha^2) Z_{22}^{n} }} \, \left| 1- (1-\alpha^2) Z_{22}^{n}  \right| \right] \notag \\
&\leq&  2 \left(2 \sqrt{2/\pi} + 1 \right)^{2} \, \frac{1- \alpha^2}{\sqrt{1+ \alpha^2}} \, \left( C_{5}(\alpha) \right)^{\frac{1}{4}}  \notag  \\
&& \times \left[ \frac{2(1+\alpha^2)}{1-\alpha^2} + \frac{ (1-\alpha^2)^2 \left(C_{1}(\alpha)\right)^2}{10} + \frac{(1-\alpha^2)^2 C_2(\alpha) }{5} \right]^{\frac{1}{2}} \times n^{-\frac{1}{2}}.  \label{eq:expectationdiffterm2}
\end{eqnarray}
Then \eqref{eq:dwnominator} follows by combining \eqref{eq:boundthedifference}, \eqref{eq:expectationdiffterm1} and \eqref{eq:expectationdiffterm2}. This completes the proof.
\end{proof}

\subsection{Convergence in Kolmogorov distance.} \label{subsec:convergenceKolmogorov}
In this subsection, we will derive an upper bound for the Kolmogorov distance between $\sqrt{n} \theta_n$ and the standard normal distribution. This result relies on two preparatory lemmas. The first lemma, displayed in Lemma \ref{lm:kolmogorovratio} below, is from the work of Michel and Pfanzagl \cite{michel1971accuracy}. The lemma provides upper bounds for (i) the Kolmogorov distance between the ratio of two random variables and a standard normal random variable and (ii) the Kolmogorov distance between the sum of two random variables and the standard normal random variable. The second lemma, Lemma \ref{lm:asymptoticdn}, provides the asymptotics of $ d_{n}\left( t \sqrt{n \ln n} \right) $ and plays an important role in deriving the upper bound.

\begin{lemma} \label{lm:kolmogorovratio}
Let $X$, $Y$ and $Z$ be three random variables defined on a probability space $ (\Omega, \mathcal{F}, P )$ such that $P( Y >0 ) =1$. Then, for all $\epsilon >0$, we have
\begin{enumerate}
\item $d_{Kol} \left( \frac{X}{Y}, \, N \right) \leq d_{Kol} (X, N) + P\left( |Y-1| > \epsilon \right) + \epsilon$,
\item $ d_{Kol}(X+Z, N) \leq d_{Kol}(X,N) + P \left( |Z|> \epsilon \right) + \epsilon $,
\end{enumerate}
where $N \sim \mathcal{N}(0,1)$.
\end{lemma}
\begin{proof}
See Michel and Pfanzagl \cite{michel1971accuracy}.
\end{proof}

\begin{lemma} \label{lm:asymptoticdn}
For every fixed $t \in \mathbb{R}$, as $n$ tends to $\infty$,
\begin{eqnarray*}
d_{n}\left(t\sqrt{n \ln n}\right) = (1 + o(1)) \, e^{ - \frac{t}{1- \alpha^2} \sqrt{n \ln n} \,  -  \left[ \frac{(1+ \alpha^2)^2 \, t^2}{4(1- \alpha^2)^3} + \frac{t^2}{2(1-\alpha^2)^2} - \frac{t^2}{4(1- \alpha^2)} + o(1) \right] \ln n }.
\end{eqnarray*}
\end{lemma}
\begin{proof}
For simplicity of notation, we use $\gamma_{1,n,\mathrm{ln}}$, $\gamma_{2,n,\mathrm{ln}}$ and $\Delta_{n, \mathrm{ln}}$ to denote $\gamma_{1}\left(t \sqrt{\ln n/n}\right)$, $\gamma_{2}\left(t \sqrt{\ln n/n}\right)$ and $ \Delta\left( t \sqrt{\ln n /n} \right) $ respectively. Then, together with Lemma \ref{lem1}, we have
\begin{eqnarray}
&& d_{n}(t \sqrt{n \ln n})  \notag  \\
&=& \frac{\left( \gamma_{1,n,\mathrm{ln}} \right)^{n+1}  - \left( \gamma_{2,n,\mathrm{ln}} \right)^{n+1}}{\sqrt{\Delta_{n, \mathrm{ln}}}}
-  \frac{\alpha^2\left[\left( \gamma_{1,n,\mathrm{ln}} \right)^{n}  - \left( \gamma_{2,n,\mathrm{ln}} \right)^{n}\right]}{\sqrt{\Delta_{n, \mathrm{ln}}}} \notag \\ [2mm]
&& + \frac{t \sqrt{n \ln n}\left[\left(\gamma_{1,n,\mathrm{ln}}\right)^{n+1}  + \left(\gamma_{2,n,\mathrm{ln}}\right)^{n+1}\right]}{n \,\Delta_{n, \mathrm{ln}}}  
  - \frac{(n-1) \alpha^2 \, t \sqrt{n \ln n}\left[\left(\gamma_{1,n,\mathrm{ln}}\right)^{n}  + \left(\gamma_{2,n,\mathrm{ln}}\right)^{n}\right]}{n^2 \,\Delta_{n, \mathrm{ln}}}  \notag \\ [2mm]
&& + \frac{2 (n-1) \alpha \, t \sqrt{n \ln n}\left[\left(\gamma_{1,n,\mathrm{ln}}\right)^{n+1}  + \left(\gamma_{2,n,\mathrm{ln}}\right)^{n+1}\right]}{n \,\Delta_{n, \mathrm{ln}} \left( n (1-\alpha)^2 - t \sqrt{n \ln n} \right)}  \notag  \\ [2mm]
&& - \frac{2 (n-2) \alpha^3 \, t \sqrt{n \ln n}\left[\left(\gamma_{1,n,\mathrm{ln}}\right)^{n}  + \left(\gamma_{2,n,\mathrm{ln}}\right)^{n}\right]}{n \,\Delta_{n, \mathrm{ln}} \left( n (1-\alpha)^2 - t \sqrt{n \ln n} \right)} \notag  \\  [2mm]
&& - \frac{2 (n+1) \alpha^2 \, t \sqrt{n \ln n}\left[\left(\gamma_{1,n,\mathrm{ln}}\right)^{n}  + \left(\gamma_{2,n,\mathrm{ln}}\right)^{n}\right]}{n \,\Delta_{n, \mathrm{ln}} \left( n (1-\alpha)^2 - t \sqrt{n \ln n} \right)}
+ \frac{2  \alpha^4 \, t \sqrt{n \ln n}\left[\left(\gamma_{1,n,\mathrm{ln}}\right)^{n-1}  + \left(\gamma_{2,n,\mathrm{ln}}\right)^{n-1}\right]}{\Delta_{n, \mathrm{ln}} \left( n (1-\alpha)^2 - t \sqrt{n \ln n} \right)} \notag \\[2mm]
&& + \frac{2 \, \alpha^{n+1} (1- \alpha) \, t \sqrt{n \ln n}\left[\gamma_{1,n,\mathrm{ln}}  + \gamma_{2,n,\mathrm{ln}}\right]}{n \,\Delta_{n, \mathrm{ln}} \left( n (1-\alpha)^2 - t \sqrt{n \ln n} \right)}  
+ \frac{4 \, \alpha^{n+2} (1- \alpha) \, t \sqrt{n \ln n}}{n \,\Delta_{n, \mathrm{ln}} \left( n (1-\alpha)^2 - t \sqrt{n \ln n} \right)}. \label{eq:dnlogscale}
\end{eqnarray}
Note that
\begin{eqnarray*}
\gamma_{1, n, \mathrm{ln}} &=& \frac{1-t \sqrt{\frac{\ln n}{n}} + \alpha^2 + \sqrt{\left(1 -t \sqrt{\frac{\ln n}{n}} + \alpha^2 \right)^2 -4\alpha^2}}{2}  \\
&=& \frac{1+ \alpha^2 - t \sqrt{\frac{\ln n}{n}} + (1- \alpha^2) \sqrt{1 - \frac{2(1+\alpha^2) t}{(1-\alpha^2)^2}  \sqrt{\frac{\ln n}{n}}+ \frac{t^2}{(1- \alpha^2)^2}\frac{\ln n}{n}}  }{2}.
\end{eqnarray*}
Applying Taylor's expansion, it follows after rearrangement that
\begin{eqnarray*}
&& \sqrt{1 - \frac{2(1+\alpha^2) t}{(1-\alpha^2)^2}  \sqrt{\frac{\ln n}{n}}+ \frac{t^2}{(1- \alpha^2)^2}\frac{\ln n}{n}}  \\
&=& 1 - \frac{(1 + \alpha^2)t}{(1- \alpha^2)^2} \sqrt{\frac{\ln n }{n}}
 -\left[ \frac{(1+ \alpha^2)^2 t^2}{2 (1-\alpha^2)^4} - \frac{t^2}{2(1- \alpha^2)^2} \right] \frac{\ln n}{n} + o\left( \frac{\ln n }{n} \right).
\end{eqnarray*}
Hence,
\begin{eqnarray*}
\gamma_{1, n, \mathrm{ln}} = 1- \frac{t}{1- \alpha^2} \sqrt{\frac{\ln n}{n}} - \left[ \frac{(1+ \alpha^2)^2 t^2}{4(1- \alpha^2)^3} - \frac{t^2}{4(1- \alpha^2)} \right] \frac{\ln n }{n} + o\left( \frac{\ln n}{n} \right).
\end{eqnarray*}
Again, applying Taylor's expansion, it follows after a rearrangement of terms that
\begin{eqnarray*}
\ln\left( \gamma_{1, n, \mathrm{ln}} \right)
= - \frac{t}{1- \alpha^2} \sqrt{\frac{\ln n}{n}} - \left[ \frac{(1+ \alpha^2)^2 t^2}{4(1- \alpha^2)^3} + \frac{t^2}{2(1-\alpha^2)^2} - \frac{t^2}{4(1- \alpha^2)} \right] \frac{\ln n }{n} + o\left( \frac{\ln n}{n} \right).
\end{eqnarray*}
Then,
\begin{eqnarray*}
 \left( \gamma_{1, n, \mathrm{ln}} \right)^n = \exp   \left(n \ln\left( \gamma_{1, n, \mathrm{ln}} \right)\right)  
 =  e^{ - \frac{t}{1- \alpha^2} \sqrt{n \ln n} \, -  \left[ \frac{(1+ \alpha^2)^2 \, t^2}{4(1- \alpha^2)^3} + \frac{t^2}{2(1-\alpha^2)^2} - \frac{t^2}{4(1- \alpha^2)} + o(1) \right] \ln n }.
\end{eqnarray*}
Note that $ \gamma_{2, n, \mathrm{ln}} = \alpha^2 / \gamma_{1, n, \mathrm{ln}}$. Combining this with the fact that $0 \leq \alpha^2 <1$, we have
\begin{eqnarray*}
&&\left(\gamma_{2, n, \mathrm{ln}}\right)^{n} = \alpha^{2n} \left( \gamma_{1, n, \mathrm{ln}} \right)^{-n} \\  [2mm]
&=&  e^{ n \ln\left(\alpha^2\right) + \frac{t}{1- \alpha^2} \sqrt{n \ln n} \, +  \left[ \frac{(1+ \alpha^2)^2 \, t^2}{4(1- \alpha^2)^3} + \frac{t^2}{2(1-\alpha^2)^2} - \frac{t^2}{4(1- \alpha^2)} + o(1) \right] \ln n }  \\ [2mm]
&=& o(1) \,  e^{ - \frac{t}{1- \alpha^2} \sqrt{n \ln n} \,  -  \left[ \frac{(1+ \alpha^2)^2 \, t^2}{4(1- \alpha^2)^3} + \frac{t^2}{2(1-\alpha^2)^2} - \frac{t^2}{4(1- \alpha^2)} + o(1) \right] \ln n }.
\end{eqnarray*}
A routine calculation gives the result that as $n \rightarrow \infty$,
\begin{eqnarray*}
\gamma_{1, n, \mathrm{ln}}  &=& 1 + o(1), \\
\gamma_{2, n, \mathrm{ln}} &=& \alpha^2 + o(1), \\
\Delta_{n, \mathrm{ln}} &=& (1- \alpha^2)^{2} + o(1).
\end{eqnarray*}
Combining the last three displays yields
\begin{eqnarray*}
&&\frac{\left( \gamma_{1,n,\mathrm{ln}} \right)^{n+1}  - \left( \gamma_{2,n,\mathrm{ln}} \right)^{n+1}}{\sqrt{\Delta_{n, \mathrm{ln}}}}
-  \frac{\alpha^2\left[\left( \gamma_{1,n,\mathrm{ln}} \right)^{n}  - \left( \gamma_{2,n,\mathrm{ln}} \right)^{n}\right]}{\sqrt{\Delta_{n, \mathrm{ln}}}} \\
&=&(1 + o(1)) \, e^{ - \frac{t}{1- \alpha^2} \sqrt{n \ln n} \,  -  \left[ \frac{(1+ \alpha^2)^2 \, t^2}{4(1- \alpha^2)^3} + \frac{t^2}{2(1-\alpha^2)^2} - \frac{t^2}{4(1- \alpha^2)} + o(1) \right] \ln n }.
\end{eqnarray*}
Then, the desired results follow by noting that the remaining items on the right-hand side of \eqref{eq:dnlogscale} are at most
\begin{equation*}
 o(1) \, e^{ - \frac{t}{1- \alpha^2} \sqrt{n \ln n} \,  -  \left[ \frac{(1+ \alpha^2)^2 \, t^2}{4(1- \alpha^2)^3} + \frac{t^2}{2(1-\alpha^2)^2} - \frac{t^2}{4(1- \alpha^2)} + o(1) \right] \ln n }.
\end{equation*}
\end{proof}

We now turn to the upper bound for the Kolmogorov distance between $\sqrt{n} \theta_n$ (after scaling) and standard normal distribution.
\begin{theorem} \label{thm:kolmogorov}
There exists a constant $C_{7}(\alpha)$ depending on $\alpha$ such that for $n$ sufficiently large,
\begin{equation*}
d_{Kol}\left( \sqrt{\frac{1-\alpha^2}{1+ \alpha^2}} \sqrt{n} \theta_n , \, \mathcal{N}(0,1) \right) \leq C_{7}(\alpha) \, \frac{ \sqrt{  \ln n}}{\sqrt{n}}.
\end{equation*}
In particular, $\sqrt{n} \theta_n$ converges in distribution to $\mathcal{N}\left( 0, (1+ \alpha^2)/(1- \alpha^2) \right)$ as $n \rightarrow \infty$.
\end{theorem}
\begin{proof}
Let
\begin{equation*}
M_{1} := \frac{1+ \alpha^2}{1- \alpha^2} +1.
\end{equation*}
Note that
\begin{equation*}
\sqrt{\frac{1-\alpha^2}{1+ \alpha^2}} \sqrt{n} \theta_n
  = \frac{\sqrt{\frac{(1- \alpha^2)^3}{1+\alpha^2}} \sqrt{n} \, Z_{12}^{n}}{ \sqrt{(1-\alpha^2)Z_{11}^{n} \times (1-\alpha^2) Z_{22}^{n}  }  }.
\end{equation*}
Invoking Lemma \ref{lm:kolmogorovratio} with $X$, $Y$ and $\epsilon$ replaced by $ \sqrt{\frac{(1- \alpha^2)^3}{1+\alpha^2}} \sqrt{n} \, Z_{12}^{n},$ $  \sqrt{(1-\alpha^2)Z_{11}^{n} \times (1-\alpha^2) Z_{22}^{n}  },$ and $ 3 \, M_{1} \sqrt{\ln n / n} $ respectively, together with Theorem \ref{thm:nominator}, it suffices to prove that
\begin{equation*}
P\left( \left|  \sqrt{(1-\alpha^2)Z_{11}^{n} \times (1-\alpha^2) Z_{22}^{n}  } -1 \right| > 3 \, M_{1} \sqrt{\ln n / n} \right)
\leq \frac{C_{8}(\alpha)}{\sqrt{n}},
\end{equation*}
for $n$ sufficiently large and for some constant $C_{8}(\alpha)$ depending on $\alpha$.

Applying the inequality $|x-1| \leq |x^2 -1|$ for $x\geq 0$, we have that
\begin{eqnarray*}
&& \left|  \sqrt{(1-\alpha^2)Z_{11}^{n} \times (1-\alpha^2) Z_{22}^{n}  } -1 \right|
\leq \left| (1-\alpha^2)Z_{11}^{n} \times (1-\alpha^2) Z_{22}^{n} -1  \right|   \\
&=& \left|  \left( (1-\alpha^2)Z_{11}^{n} -1 \right) \left( (1-\alpha^2)Z_{22}^{n} -1 \right) + \left( (1-\alpha^2)Z_{11}^{n} -1 \right)  + \left( (1-\alpha^2)Z_{22}^{n} -1 \right)  \right| \\
&\leq& \left|  \left( (1-\alpha^2)Z_{11}^{n} -1 \right) \left( (1-\alpha^2)Z_{22}^{n} -1 \right) \right| + \left|  (1-\alpha^2)Z_{11}^{n} -1 \right| 
+ \left|  (1-\alpha^2)Z_{22}^{n} -1 \right|.
\end{eqnarray*}
Then, for $n$ sufficiently large,
\begin{eqnarray}
&& P\left( \left|  \sqrt{(1-\alpha^2)Z_{11}^{n} \times (1-\alpha^2) Z_{22}^{n}  } -1 \right| > 3 \, M_{1} \sqrt{\ln n / n} \right)  \notag \\
&\leq& P\left( \left|  \left( (1-\alpha^2)Z_{11}^{n} -1 \right) \left( (1-\alpha^2)Z_{22}^{n} -1 \right) \right| > M_{1} \sqrt{\ln n / n} \right) \notag \\
&& + P\left( \left|  (1-\alpha^2)Z_{11}^{n} -1 \right|  > M_{1} \sqrt{\ln n / n}  \right) 
 +  P\left( \left|  (1-\alpha^2)Z_{22}^{n} -1 \right|  > M_{1} \sqrt{\ln n / n}  \right)  \notag \\
&\leq&  P\left( \left|  (1-\alpha^2)Z_{11}^{n} -1 \right|  > \sqrt{ M_{1} \sqrt{\ln n / n} }  \right) 
 + P\left( \left|  (1-\alpha^2)Z_{22}^{n} -1 \right|  > \sqrt{ M_{1} \sqrt{\ln n / n} }  \right)  \notag \\
&& + P\left( \left|  (1-\alpha^2)Z_{11}^{n} -1 \right|  > M_{1} \sqrt{\ln n / n}  \right) 
 +  P\left( \left|  (1-\alpha^2)Z_{22}^{n} -1 \right|  > M_{1} \sqrt{\ln n / n}  \right) \notag \\
&\leq&  2P\left( \left|  (1-\alpha^2)Z_{11}^{n} -1 \right|  > M_{1} \sqrt{\ln n / n}  \right) 
 +  2P\left( \left|  (1-\alpha^2)Z_{22}^{n} -1 \right|  > M_{1} \sqrt{\ln n / n}  \right)  \notag \\
&=& 4 P\left( \left|  (1-\alpha^2)Z_{11}^{n} -1 \right|  > M_{1} \sqrt{\ln n / n}  \right),  \label{eq1:tailDbound}
\end{eqnarray}
where the last inequality follows by the fact that $M_{1} \sqrt{\ln n /n} <1$ for $n$ sufficiently large and the last equality follows by the identical distribution of $Z_{11}^{n}$ and $Z_{22}^{n}$. Thus, we need only to bound
\begin{equation*}
P\left( \left|  (1-\alpha^2)Z_{11}^{n} -1 \right|  > M_{1} \sqrt{\ln n / n}  \right).
\end{equation*}
We first note that
\begin{eqnarray}
&& P\left( \left|  (1-\alpha^2)Z_{11}^{n} -1 \right|  > M_{1} \sqrt{\ln n / n}  \right) \notag  \\
&=& P\left( (1- \alpha^2) Z_{11}^{n} > 1 +  M_{1} \sqrt{\ln n / n} \right)
 + P\left( - (1- \alpha^2) Z_{11}^{n} >   -1 +  M_{1} \sqrt{\ln n / n} \right).  \label{eq:twotermforprobability}
\end{eqnarray}
Applying Markov's inequality yields
\begin{eqnarray}
&& P\left( (1- \alpha^2) Z_{11}^{n} > 1 +  M_{1} \sqrt{\ln n / n} \right) \notag \\
&=& P\left( \sqrt{n \ln n} (1- \alpha^2) Z_{11}^{n} > \sqrt{n \ln n} +  M_{1} \ln n \right) \notag  \\
&\leq& \frac{E \left[ e^{ \sqrt{n \ln n} (1- \alpha^2) Z_{11}^{n} } \right]}{e^{\sqrt{n \ln n} +  M_{1} \ln n }}. \label{eq:concentrationinequality}
\end{eqnarray}
It follows by Remark \ref{remark:boundlambda} that $$\sqrt{n \ln n} (1- \alpha^2) \,\lambda_k \leq C_{3}(\alpha)^{1/4} \, (1-\alpha^2) \, n^{-1/4} \, \sqrt{\ln n} < 1/2$$ for $n$ sufficiently large. Then
\begin{eqnarray*}
&& E \left[ e^{ \sqrt{n \ln n} (1- \alpha^2) Z_{11}^{n} } \right] 
=  E \left[ e^{ \sum_{k} \sqrt{n \ln n} (1- \alpha^2) \lambda_k W_{k}^{2} } \right] 
= \prod_{k=1}^{n}  E \left[ e^{  \sqrt{n \ln n} (1- \alpha^2) \lambda_k W_{k}^{2} } \right] \\
&=& \prod_{k=1}^{n} \left( 1- 2 \sqrt{n \ln n} \, (1- \alpha^2) \lambda_k \right)^{-\frac{1}{2}}
 = \left( d_{n}\left( 2 \sqrt{n \ln n} \, (1- \alpha^2) \right) \right)^{-\frac{1}{2}}  \\
&=& (1+ o(1)) \, e^{\sqrt{n \ln n}  + \left( \frac{1+\alpha^2}{1- \alpha^2}  + o(1)\right) \ln n },
\end{eqnarray*}
where the third equality follows by a standard expression for the mgf of a linear-quadratic functional of a Gaussian random variable, the fourth equality follows by the representation \eqref{eq:anotherformofacp} of $d_{n}(\lambda)$, and the last equality follows by invoking Lemma \ref{lm:asymptoticdn} with $t$ replaced by $2(1- \alpha^2)$. Note that $M_{1} = (1 + \alpha^2)/(1 - \alpha^2) + 1$. Together with \eqref{eq:concentrationinequality},  we have
\begin{eqnarray*}
&& P\left( (1- \alpha^2) Z_{11}^{n} > 1 +  M_{1} \sqrt{\ln n / n} \right)  \\
&\leq& ( 1 + o(1)) \, e^{ \left( \frac{1+\alpha^2}{1- \alpha^2}  + o(1)\right) \ln n - M_{1} \ln n }  \\ [2mm]
&=& ( 1 + o(1)) \, e^{  - ( M_{1} - \frac{1+\alpha^2}{1- \alpha^2} - o(1)) \,  \ln n }  \\  [2mm]
&=& ( 1 + o(1)) \, e^{  - ( 1 - o(1)) \,  \ln n }  \\ [2mm]
 &\leq& \frac{C_{9}(\alpha)}{\sqrt{n}},
\end{eqnarray*}
for $n$ sufficiently large and for some constant $C_{9}(\alpha)$ depending on $\alpha$. Similarly,
\begin{equation*}
P\left( - (1- \alpha^2) Z_{11}^{n} >   -1 +  M_{1} \sqrt{\ln n / n} \right)
 \leq  \frac{C_{10}(\alpha)}{\sqrt{n}}, 
\end{equation*}
for $n$ large enough for some constant $C_{10}(\alpha)$ depending on $\alpha$. Combining the last two displays with \eqref{eq:twotermforprobability} yields that, for $n$ sufficiently large,
\begin{equation*}
P\left( \left|  (1-\alpha^2)Z_{11}^{n} -1 \right|  > M_{1} \sqrt{\ln n / n}  \right)
\leq \frac{C_{9}(\alpha) +  C_{10}(\alpha) }{\sqrt{n}}. 
\end{equation*}
The desired result then follows by combining the last display and \eqref{eq1:tailDbound}.
\end{proof}

\begin{remark} \rm On a practical level, the study of the aforementioned convergence rates should be of direct use to practitioners wishing to conduct tests of independence between two AR(1) processes with smaller sample sizes. For example, suppose that a practitioner studying time series data that can be appropriately modeled as AR(1) processes is considering employing the $n^{-1/2}$ rate of convergence in Wasserstein distance as mentioned above. This rate would lead the practitioner to conclude that in order to achieve an error of 1\% between the empirical correlation and the standard Gaussian, one would need 10,000 data points. Furthermore, this informs the practitioner that, should the sample size of these two processes be considerably less than 10,000, that the exact distribution of the scaled empirical correlation (rather than its scaled asymptotic distribution) should be utilized.
\end{remark}

\section{Power for a test of independence.} \label{sec:Power}

The purpose of this section is to study a test of independence between two AR(1) processes with Gaussian increments using their path data. For the two AR(1) processes defined in \eqref{eq:defofAR(1)processes}, we propose the null hypothesis as
\begin{equation*}
H_{0}: \{X_n\}_{n=0}^{\infty} \text{ and } \{Y_n\}_{n=0}^{\infty} \text{ are independent},
\end{equation*}
and the alternative hypothesis as
\begin{equation*}
H_{a}: \{X_n\}_{n=0}^{\infty} \text{ and } \{Y_n\}_{n=0}^{\infty} \text{ are correlated with coefficient $r$}.
\end{equation*}
We pause to provide some further explanation for the two hypotheses mentioned above. For the null hypothesis, we define independence between $\{X_{n}\}_{n=0}^{\infty}$ and $\{Y_{n}\}_{n=0}^{\infty}$ as the increments $\xi_{1}, \xi_{2}, \xi_{3}, \dots, \eta_{1}, \eta_{2}, \eta_{3}, \dots$ being independent. In contrast, for the alternative hypothesis, we say that $\{X_{n}\}_{n=0}^{\infty}$ and $\{Y_{n}\}_{n=0}^{\infty}$ are correlated with coefficient $r$ if the pairs of increments $(\xi_n, \eta_{n})$ are i.i.d. Gaussian random vectors with mean zero and covariance matrix $$\begin{pmatrix}
1 & r \\
r & 1
\end{pmatrix},$$ where $|r| \leq 1$ and $r \neq 0$. 

The asymptotic normality of $\sqrt{n}\, \theta_n$, as presented in Section \ref{sec:mainconvergence}, indicates that, under the null hypothesis $H_{0}$, the rejection region shall be
\begin{equation*}
|\sqrt{n}\, \theta_n| > c_a,
\end{equation*}
where $c_a >0 $ is selected according to the given level of the test. In this section, we will study the power of the above statistical test. We will provide a lower bound for the power and prove that it converges to $1$ as $n$ tends to $\infty$. The key ingredient for doing so shall be the convergence rate of the distribution of $\sqrt{n}\, (\theta_n - r)$ to the normal distribution under the alternative hypothesis, the latter to be exposed in Section \ref{subsec:convergencerateunderHa}. Relying on the convergence rate, we study the power of the test in Section \ref{subsec:testpower}.

\subsection{Convergence rate under $H_{a}$.} \label{subsec:convergencerateunderHa}

In this subsection, we will study the Kolmogorov distance between $\sqrt{n}\, (\theta_n -r )$ (after scaling) and the standard normal distribution under the alternative hypothesis. Therefore, in the remainder of this section, we shall assume that the pairs of random variables $(\xi_n, \eta_{n})$ are i.i.d. Gaussian vectors with mean zero and covariance matrix $$\begin{pmatrix}
1 & r \\
r & 1
\end{pmatrix}.$$

For purposes of simplicity, we will continue to use the notation given in Section \ref{sec:preliminaries} and Section \ref{sec:jointmgf}. That is,
\begin{equation*}
\theta_n = \frac{Z_{12}^{n}}{ \sqrt{ Z_{11}^{n} Z_{22}^{n} } },
\end{equation*}
and $Z_{11}^{n}$, $Z_{12}^{n}$, and $Z_{22}^{n}$ can be simplified as
\begin{equation*}
Z_{11}^{n} = \sum_{k=1}^{n} \lambda_{k} \, W_{k}^{2}, \  Z_{12}^{n} = \sum_{k=1}^{n} \lambda_{k} \, W_{k} V_{k}, \ \text{and } Z_{22}^{n} = \sum_{k=1}^{n} \lambda_{k} \, V_{k}^{2},
\end{equation*}
where
\begin{eqnarray*}
&& \left( W_1, W_2, \dots, W_n \right)^{\intercal} := P_{n}\, \mathbf{\Xi}_n = P_{n} \,  \left(\xi_1, \xi_2, \dots, \xi_n \right)^{\intercal}\\
&&  \left( V_1, V_2, \dots, V_n \right)^{\intercal} := P_{n}\, \mathbf{H}_n= P_{n} \left( \eta_1, \eta_2, \dots, \eta_n \right)^{\intercal}.
\end{eqnarray*}
As defined in Section \ref{sec:jointmgf}, $P_n$ is an orthogonal matrix. By the invariance of distribution of Gaussian random vector under orthogonal transformation, it may be easily shown that the pairs of random variables $(W_n, V_{n})$ are still i.i.d. Gaussian vectors with mean zero and covariance matrix $$\begin{pmatrix}
1 & r \\
r & 1
\end{pmatrix}.$$ Then
\begin{equation}
\theta_n -r = \frac{Z_{12}^{n}}{ \sqrt{ Z_{11}^{n} Z_{22}^{n} } } - r = \frac{ Z_{12}^{n} - r \sqrt{ Z_{11}^{n} Z_{22}^{n} } }{\sqrt{ Z_{11}^{n} Z_{22}^{n} }} = \frac{\left(Z_{12}^{n}\right)^2 - r^2 \, Z_{11}^{n} Z_{22}^{n} }{\sqrt{ Z_{11}^{n} Z_{22}^{n} } \left( Z_{12}^{n} + r \sqrt{ Z_{11}^{n} Z_{22}^{n} } \right) }.  \label{eqc:reformationoftheta-r}
\end{equation}

In what follows, we will first investigate the asymptotics of the scaled numerator of $\theta_n -r$, i.e., $ \sqrt{n} \,  \left(\left(Z_{11}^{n}\right)^2 - r^2 \, Z_{11}^{n} Z_{22}^{n}\right) $. We then study the tail of the denominator of $\theta_n -r$. Finally, we will derive the asymptotics of $\theta_n -r$ by invoking Lemma \ref{lm:kolmogorovratio}. These results are encapsulated in Theorem \ref{thm:asymptoticsnumeratoralternative}, Theorem \ref{thm:denominatortailboundalternative}, and Theorem \ref{thm:mainresultthetaconvergealternative}.

\begin{theorem} \label{thm:asymptoticsnumeratoralternative}
Let
\begin{eqnarray*}
C_{11}(\alpha, r )&:=& C_{19}(\alpha, r) + \frac{1+8 r^2 + 4 r^4}{r^2} \, \frac{(1-\alpha^2)^5}{1+ \alpha^2} \, C_{3}(\alpha) +1 \\
&& +\, \frac{1+ 2 r^2 }{2 |r|} \sqrt{\frac{(1-\alpha^2)^5}{1+\alpha^2}} \left( \frac{1+\alpha^2}{(1-\alpha^2)^3}  + C_{2}(\alpha)\right),
\end{eqnarray*}
with $C_{19}(\alpha, r)$ defined in Theorem \ref{thm:2chaos+4chaosasymptotics}.
Under the alternative hypothesis $H_a$, we have
\begin{equation*}
d_{Kol} \left( \frac{1}{2 r (1 -r^2)} \sqrt{\frac{(1- \alpha^2)^5}{1 + \alpha^2}} \, \sqrt{n} \left( \left(Z_{11}^{n}\right)^2 - r^2 \, Z_{11}^{n} Z_{22}^{n} \right) , \, \mathcal{N}(0,1) \right) \leq \frac{C_{11}( \alpha, r )}{\sqrt{n}}.
\end{equation*}
\end{theorem}

\begin{proof}
See Appendix C.
\end{proof}

Note that $M_{1} = \frac{1+\alpha^2}{1- \alpha^2} + 1$, as was defined in the proof of Theorem \ref{thm:kolmogorov}. We also define
\begin{equation*}
M_{2} := \frac{1+r^2}{2 r^2} \, \frac{1+\alpha^2}{1 - \alpha^2} + 1.
\end{equation*}

We now continue with Theorem \ref{thm:denominatortailboundalternative} below.

\begin{theorem} \label{thm:denominatortailboundalternative}
Under the alternative hypothesis $H_a$, there exists a constant $C_{12}(\alpha, r)$ such that for $n$ sufficiently large,
\begin{equation*}
P\left( \left|  \frac{ Z_{12}^{n} \sqrt{ Z_{11}^{n} Z_{22}^{n} } + r Z_{11}^{n} Z_{22}^{n} }{2 r / (1-\alpha^2)^2 }  -1 \right| 
 > \frac{1}{2} (6 M_1 + M_2 + 3M_1 M_2) \sqrt{\frac{\ln n }{n}} \right) \leq \frac{C_{12}(\alpha, r)}{\sqrt{n}}.
\end{equation*}
\end{theorem}

\begin{proof}
See Appendix C.
\end{proof}

With the above two theorems, we arrive at our main theorem in this section.

\begin{theorem} \label{thm:mainresultthetaconvergealternative}
Under the alternative hypothesis $H_a$, there exists a constant $C_{13}(\alpha, r)$ such that for $n$ large enough,
\begin{equation*}
d_{Kol} \left( \frac{1}{1-r^2} \sqrt{\frac{1-\alpha^2}{1+\alpha^2}} \sqrt{n} (\theta_n -r), \, \mathcal{N}(0,1) \right)
\leq C_{13}(\alpha, r) \times \sqrt{\frac{\ln n}{n}}.
\end{equation*}
\begin{proof}
Note that
\begin{eqnarray*}
\frac{1}{1-r^2} \sqrt{\frac{1-\alpha^2}{1+\alpha^2}} \sqrt{n} (\theta_n -r)
= \frac{ \frac{1}{2 r (1 -r^2)} \sqrt{\frac{(1- \alpha^2)^5}{1 + \alpha^2}} \, \sqrt{n} \left( \left(Z_{11}^{n}\right)^2 - r^2 \, Z_{11}^{n} Z_{22}^{n} \right) }{\frac{ Z_{12}^{n} \sqrt{ Z_{11}^{n} Z_{22}^{n} } + r Z_{11}^{n} Z_{22}^{n} }{2 r / (1-\alpha^2)^2 } }.
\end{eqnarray*}
Together with Theorem \ref{thm:asymptoticsnumeratoralternative} and Theorem \ref{thm:denominatortailboundalternative}, the desired result follows immediately by invoking Lemma \ref{lm:kolmogorovratio} with
\begin{eqnarray*}
&& X = \frac{1}{2 r (1 -r^2)} \sqrt{\frac{(1- \alpha^2)^5}{1 + \alpha^2}} \, \sqrt{n} \left( \left(Z_{11}^{n}\right)^2 - r^2 \, Z_{11}^{n} Z_{22}^{n} \right), \\
&& Y = \frac{ Z_{12}^{n} \sqrt{ Z_{11}^{n} Z_{22}^{n} } + r Z_{11}^{n} Z_{22}^{n} }{2 r / (1-\alpha^2)^2 }, \\
&& \epsilon =  \frac{1}{2} (6 M_1 + M_2 + 3M_1 M_2) \sqrt{\frac{\ln n }{n}}.
\end{eqnarray*}
\end{proof}
\end{theorem}

\subsection{Power.} \label{subsec:testpower}
In Section \ref{subsec:convergencerateunderHa}, we derived the rate of convergence of $\theta_n - r$ (after scaling) to the normal distribution under the alternative hypothesis $H_a$. We now proceed to study the power of the test with the rejection region 
\begin{equation*}
|\sqrt{n}\, \theta_n| > c_a.
\end{equation*}
We first recall a trivial property of the Kolmogorov distance. For a pair of random variables $(X, Y)$, we have that for any constant $c$,
\begin{eqnarray*}
&& \left| P\left( X >c \right) - P\left( Y > c \right) \right| \leq d_{Kol}(X,Y),  \\
&& \left| P\left( X  < c \right) - P\left( Y  < c \right) \right| \leq d_{Kol}(X,Y).
\end{eqnarray*}
Indeed, by the definition of Kolmogorov distance,
\begin{eqnarray*}
&& \left| P\left( X >c \right) - P\left( Y > c \right) \right|
= \left| \left( 1- P(X \leq c) \right) - \left( 1- P(Y \leq c) \right) \right| \\
&=& \left| P(X \leq c) - P(Y \leq c)\right| \leq d_{Kol}(X,Y).
\end{eqnarray*}
Further,
\begin{eqnarray*}
\left| P\left( X  < c \right) - P\left( Y  < c \right) \right|
= \left| \lim_{k \rightarrow \infty} \left( P\left( X \leq c- 1/k \right) -P\left( Y \leq c- 1/k \right) \right) \right|
\leq d_{Kol}(X,Y).
\end{eqnarray*}

We now turn to the power of the test. Let $\Phi(x)$ be the tail of the standard normal distribution, i.e.,
\begin{equation*}
\Phi(x) = \int_{x}^{\infty} \, \frac{1}{\sqrt{2 \pi}} \, e^{-\frac{y^2}{2}} \, dy.
\end{equation*}
Under the alternative hypothesis $H_a$, the power is
\begin{eqnarray*}
P\left( | \sqrt{n}\, \theta_n  | > c_a \Big| H_a \right)
= P\left(  \sqrt{n}\, \theta_n  > c_a \Big| H_a \right) + P\left(  \sqrt{n}\, \theta_n   < - c_a \Big| H_a \right).
\end{eqnarray*}
Applying Theorem \ref{thm:mainresultthetaconvergealternative} and the aforementioned property of the Kolmogorov distance, we have for $n$ sufficiently large,
\begin{eqnarray*}
&& P\left(  \sqrt{n}\, \theta_n  > c_a \Big| H_a \right)
= P \left( \frac{1}{1-r^2} \sqrt{\frac{1-\alpha^2}{1 + \alpha^2}} \sqrt{n} \left( \theta_n -r \right) >  \frac{1}{1-r^2} \sqrt{\frac{1-\alpha^2}{1 + \alpha^2}} \left( c_a - \sqrt{n} r \right) \right)  \\
&\geq&  \Phi\left( \frac{1}{1-r^2} \sqrt{\frac{1-\alpha^2}{1 + \alpha^2}} \left(   c_a - \sqrt{n} r \right) \right) - C_{13}(\alpha,r) \sqrt{\frac{\ln n}{n}}.
\end{eqnarray*}
Similarly,
\begin{eqnarray*}
P\left(  \sqrt{n}\, \theta_n   < - c_a \Big| H_a \right)
\geq  \Phi\left( \frac{1}{1-r^2} \sqrt{\frac{1-\alpha^2}{1 + \alpha^2}} \left(   c_a + \sqrt{n} r \right) \right) - C_{13}(\alpha,r) \sqrt{\frac{\ln n}{n}}.
\end{eqnarray*}
Combining the last three displays yields
\begin{eqnarray}
 P\left( | \sqrt{n}\, \theta_n  | > c_a \Big| H_a \right)   
&\geq&  \Phi\left( \frac{1}{1-r^2} \sqrt{\frac{1-\alpha^2}{1 + \alpha^2}} \left(   c_a - \sqrt{n} r \right) \right) \notag \\
&&  + \, \Phi\left( \frac{1}{1-r^2} \sqrt{\frac{1-\alpha^2}{1 + \alpha^2}} \left(   c_a + \sqrt{n} r \right) \right)
  - 2 C_{13}(\alpha, r) \sqrt{\frac{\ln n}{n}}. \label{eq:poweralternative}
\end{eqnarray}
In the case $r \neq 0$, we can see easily that the right-hand side of \eqref{eq:poweralternative} tends to $1$ as $n$ tends to $\infty$. Therefore, we conclude that our test with the rejection region $| \sqrt{n} \, \theta_n| > c_a$ has asymptotic power tending towards 1.

\newpage

\bibliographystyle{plain}
\bibliography{reference_AR(1)}

\newpage

\section{Appendix A}
The alternative characteristic polynomial $d_{n}(\lambda)$ plays a key role in calculating the joint moment generating function $\phi_{n}(s_{11},s_{12}, s_{22})$ (Theorem \ref{thm1}). An explicit formula for $d_{n}(\lambda)$ is given in Lemma \ref{lem1}. The purpose of this appendix is to prove Lemma \ref{lem1}.

\subsection{Proof of Lemma \ref{lem1}}
For simplicity, we will start with $d_{n}(n^2 \lambda)$. By the definition of $d_{n}(\lambda)$, we have
\begin{equation*}
d_{n}(n^2 \lambda) = \det(I_{n} - n^2 \lambda K_n)
= \det\left( I_{n} - \left\{ n \lambda \, \frac{\alpha^{|k-j|} - \alpha^{k+j}}{1- \alpha^2} - \lambda\, \frac{(1-\alpha^k)(1- \alpha^j)}{(1-\alpha)^2} \right\}_{j,k=1}^{n} \right).
\end{equation*}
Let $A_n$ denote the $n\times n$ matrix
\begin{equation*}
I_n - \left\{ n \lambda \, \frac{\alpha^{|k-j|} - \alpha^{k+j}}{1- \alpha^2} \right\}_{j,k=1}^{n},
\end{equation*}
and let $B_n$ denote the $n\times n$ matrix
\begin{equation*}
\left\{ \lambda\, \frac{(1-\alpha^k)(1- \alpha^j)}{(1-\alpha)^2} \right\}_{j,k=1}^{n}.
\end{equation*}
Then $d_{n}(n^2\lambda) = \det(A_n +B_n)$.
 
We now perform elementary column operations and then perform elementary row operations to $A_n$ and $B_n$ simultaneously. We do so by first adding $(-\alpha)\times$ the $(j-1)$-th column of $A_n$ to its $j$-th column in the order $j=n , n-1, \dots, 2$. We then obtain
\begin{equation*}
\bar{A}_{n} := \left(
\begin{matrix}
1 - n\lambda &  -\alpha  & 0 & \cdots  &  0  & 0 \\
-n\lambda \alpha  & 1- n\lambda &  -\alpha   &\cdots  & 0 & 0  \\
 -n\lambda \alpha^2 & -n\lambda \alpha  & 1- n\lambda & \cdots  & 0 & 0  \\
 \vdots & \vdots &  \vdots &  \ddots  &  \vdots & \vdots \\
 -n\lambda \alpha^{n-2} & -n\lambda \alpha^{n-3}  & -n\lambda \alpha^{n-4} & \cdots  & 1- n \lambda & -\alpha  \\
 -n\lambda \alpha^{n-1} & -n\lambda \alpha^{n-2}  & -n\lambda \alpha^{n-3} & \cdots  & - n \lambda \alpha & 1- n \lambda
\end{matrix}
\right).
\end{equation*}
Performing the same column operations on $B_n$, we obtain
\begin{equation*}
\bar{B}_{n} := \left(
\begin{matrix}
\lambda \frac{1- \alpha}{1- \alpha} & \lambda \frac{1- \alpha}{1- \alpha} & \lambda \frac{1- \alpha}{1- \alpha} & \cdots & \lambda \frac{1- \alpha}{1- \alpha} \vspace{1ex} \\ 
\lambda \frac{1- \alpha^2}{1- \alpha} & \lambda \frac{1- \alpha^2}{1- \alpha} & \lambda \frac{1- \alpha^2}{1- \alpha} & \cdots & \lambda \frac{1- \alpha^2}{1- \alpha} \vspace{1ex} \\
\lambda \frac{1- \alpha^3}{1- \alpha} & \lambda \frac{1- \alpha^3}{1- \alpha} & \lambda \frac{1- \alpha^3}{1- \alpha} & \cdots & \lambda \frac{1- \alpha^3}{1- \alpha} \\
\vdots & \vdots & \vdots & \ddots & \vdots \\
\lambda \frac{1- \alpha^n}{1- \alpha} & \lambda \frac{1- \alpha^n}{1- \alpha} & \lambda \frac{1- \alpha^n}{1- \alpha} & \cdots & \lambda \frac{1- \alpha^n}{1- \alpha} \\
\end{matrix}
\right).
\end{equation*}
Since the determinant is invariant by adding one column multiplied by a scalar to another column, we have that $\det(A_n + B_n) = \det(\bar{A}_n + \bar{B}_n)$. We then add $(-\alpha)\times$ the $(j-1)$-th row of $\bar{A}_n$ to its $j$-th row in the order $j=n, n-1, \dots, 2$ and obtain a tri-diagonal matrix
\begin{equation*}
\widetilde{A}_n := \left(
\begin{matrix}
1- n\lambda   &  -\alpha   &   0   &  \cdots   &  0   &  0  \\
-\alpha  &  1- n\lambda + \alpha^2  &   -\alpha   &  \cdots   &  0   &  0  \\
0   &   -\alpha  &  1- n\lambda  + \alpha^2   &  \cdots   &  0   &  0  \\
\vdots & \vdots &  \vdots &  \ddots & \vdots & \vdots \\
0  &   0  &  0  &  \cdots  &  1- n\lambda +\alpha^2  &  -\alpha  \\
0  &   0  &  0  &  \cdots  &  -\alpha  &  1- n\lambda +\alpha^2
\end{matrix}
\right).
\end{equation*}
We make the same row operations to $\bar{B}_n$ and obtain
\begin{equation*}
\widetilde{B}_{n} := \lambda \, b_n b_n^{\intercal},
\end{equation*}
where $b_n := (1,1, \dots ,1)^{\intercal}$ is a $n \times 1$ column vector. Again, since the determinant is invariant to the addition of one row multiplied by a scalar to another row, we have $\det(\bar{A}_n + \bar{B}_n) = \det(\widetilde{A}_n + \widetilde{B}_n) = \det(\widetilde{A}_n + \lambda \, b_n b_n^{\intercal}) $. Hence, 
\begin{equation}
d_{n}(n^2\lambda) = \det(\widetilde{A}_n + \lambda \, b_n b_n^{\intercal}).  \label{eq:d_n1}
\end{equation}

Before proceeding to calculate $d_{n}(n^2 \lambda)$, we pause here to introduce two new determinants, closely related to $d_{n}(n^2 \lambda)$. For $n \in \mathds{N}_{+}$, let us denote by $p_{n}(\lambda)$ the following $n \times n$ determinant
\begin{equation}
\left|
\begin{matrix}
1- \lambda + \alpha^2  &  -\alpha   &   0   &  \cdots   &  0   &  0  \\
-\alpha  &  1- \lambda + \alpha^2  &   -\alpha   &  \cdots   &  0   &  0  \\
0   &   -\alpha  &  1- \lambda  + \alpha^2   &  \cdots   &  0   &  0  \\
\vdots & \vdots &  \vdots &  \ddots & \vdots & \vdots \\
0  &   0  &  0  &  \cdots  &  1- \lambda +\alpha^2  &  -\alpha  \\
0  &   0  &  0  &  \cdots  &  -\alpha  &  1- \lambda +\alpha^2
\end{matrix}
\right|. \label{eq:determinantoftridiagnol}
\end{equation}
Let us denote by $q_{n}(\lambda)$ the following $n \times n$ determinant
\begin{equation*}
\left|
\begin{matrix}
1- \lambda   &  -\alpha   &   0   &  \cdots   &  0   &  0  \\
-\alpha  &  1- \lambda + \alpha^2  &   -\alpha   &  \cdots   &  0   &  0  \\
0   &   -\alpha  &  1- \lambda  + \alpha^2   &  \cdots   &  0   &  0  \\
\vdots & \vdots &  \vdots &  \ddots & \vdots & \vdots \\
0  &   0  &  0  &  \cdots  &  1- \lambda +\alpha^2  &  -\alpha  \\
0  &   0  &  0  &  \cdots  &  -\alpha  &  1- \lambda +\alpha^2
\end{matrix}
\right|.
\end{equation*}
By convention, $p_{0}(\lambda) = q_{0}(\lambda) =1$. It follows immediately that, for $n \in \mathds{N}_{+}$, $q_{n}(\lambda)= p_{n}(\lambda) - \alpha^2 \, p_{n-1}(\lambda)$. To make sure this equation also holds for $n = 0$, we let $p_{-1}(\lambda)=0$. Further, expanding determinant \eqref{eq:determinantoftridiagnol} by its first row, we have that for $n \geq 3$,
\begin{equation*}
p_{n}(\lambda) = (1 - \lambda + \alpha^2)\, p_{n-1}(\lambda) - \alpha^2 \, p_{n-2} (\lambda).
\end{equation*}
We can easily check that the above iterative formula for $p_{n}(\lambda)$ also holds for $n=1,2$. Further, this iterative formula tells us that for $\lambda \neq (1+\alpha)^2$ or $(1- \alpha)^2$, $p_{n}(\lambda)$ must have the form
\begin{equation*}
C_1 \, \gamma_{1}^{n} + C_2 \, \gamma_{2}^{n},
\end{equation*}
where
\begin{eqnarray*}
&& \gamma_1 = \gamma_{1} (\lambda) = \frac{(1-\lambda + \alpha^2) + \sqrt{(1-\lambda+\alpha^2)^2 - 4 \alpha^2}}{2},  \\ [1mm]
&& \gamma_2 = \gamma_{2} (\lambda) = \frac{(1-\lambda + \alpha^2) - \sqrt{(1-\lambda+\alpha^2)^2 - 4 \alpha^2}}{2},
\end{eqnarray*}
and $C_1$, $C_2$ are two constants. Note that as defined in \eqref{eq:defDelta},
\begin{equation*}
\Delta = \Delta(\lambda) = (1-\lambda+\alpha^2)^2 - 4 \alpha^2.
\end{equation*}
Since $p_{1}(\lambda) = 1 - \lambda + \alpha^2$ and $p_{2}(\lambda) = (1-\lambda+\alpha^2)^2 -  \alpha^2$, the constants $C_1$ and $C_2$ can be determined as
\begin{equation*}
C_1 = \frac{\gamma_1}{\sqrt{\Delta}} \quad \text{and} \quad  C_2 = -\frac{\gamma_2}{\sqrt{\Delta}}.
\end{equation*}
Hence, for $n \geq -1$ and $\lambda \neq (1+\alpha)^2$ or $(1-\alpha)^2$,
\begin{eqnarray}
p_{n}(\lambda) &=& \frac{\gamma_{1}^{n+1} - \gamma_{2}^{n+1}}{\sqrt{\Delta}}  \notag  \\
&=& \frac{1}{\sqrt{(1-\lambda+\alpha^2)^2 - 4 \alpha^2}} \left( \frac{(1-\lambda + \alpha^2) + \sqrt{(1-\lambda+\alpha^2)^2 - 4 \alpha^2}}{2} \right)^{n+1}  \notag \\
&& - \frac{1}{\sqrt{(1-\lambda+\alpha^2)^2 - 4 \alpha^2}} \left( \frac{(1-\lambda + \alpha^2) - \sqrt{(1-\lambda+\alpha^2)^2 - 4 \alpha^2}}{2} \right)^{n+1} \notag \\
&=&  \frac{1}{2^n} \, \sum_{k=1}^{[n/2]+1} \binom{n+1}{2k-1}  \,(1-\lambda+ \alpha^2)^{n+2 -2k} \, \left((1-\lambda+\alpha^2)^2 - 4 \alpha^2\right)^{k-1}, \label{eq:p_n}
\end{eqnarray}
where $[x]$ is the largest integer less than or equal to $x$. We note that since both sides of \eqref{eq:p_n} are continuous functions of $\lambda$, the expression holds for every $\lambda \in \mathbb{R}$ and $n \geq -1$.

We now return to the expression $d_{n}(n^2 \lambda)$. Note that $\det(\widetilde{A}_{n}) = q_{n}(n\lambda)$ is a polynomial of $\lambda$ and that $\det{\widetilde{A}_{n}}$ is nonzero for all but finitely many $\lambda$. Then $\widetilde{A}_n$ is invertible for all but finitely many $\lambda$. If $\widetilde{A}_n$ is invertible, then
\begin{equation*}
\begin{pmatrix}
1 &   0  \\
b_n  &  \widetilde{A}_{n} + \lambda\, b_n b_n^{\intercal}
\end{pmatrix}
\begin{pmatrix}
1  &  - \lambda\, b_{n}^{\intercal}  \\
0  &  I_{n}
\end{pmatrix}
\begin{pmatrix}
1  &  0  \\
-\widetilde{A}_{n}^{-1} b_n  & I_n
\end{pmatrix}
=
\begin{pmatrix}
1 + \lambda \, b_n^{\intercal} \widetilde{A}_{n}^{-1} b_n  &  - \lambda \, b_n^{\intercal}  \\
0  &  \widetilde{A}_{n}
\end{pmatrix}.
\end{equation*}
Taking determinants on both sides of the last display yields
\begin{equation*}
\det(\widetilde{A}_{n} + \lambda\, b_n b_n^{\intercal})
= (1 + \lambda \, b_n^{\intercal} \widetilde{A}_{n}^{-1} b_n ) \, \det(\widetilde{A}_{n})
= \det(\widetilde{A}_{n}) + \lambda \, b_n^{\intercal} \left(\det(\widetilde{A}_{n}) \widetilde{A}_{n}^{-1}\right) b_n.
\end{equation*}
Together with \eqref{eq:d_n1}, we have
\begin{equation}
d_{n}(n^2 \lambda) = \det(\widetilde{A}_{n}) + \lambda \, b_n^{\intercal} \left(\det(\widetilde{A}_{n}) \widetilde{A}_{n}^{-1}\right) b_n, \label{eq:d_nexpressionbymatrix}
\end{equation}
for all but finitely many $\lambda$. Note that $\det(\widetilde{A}_{n}) \widetilde{A}_{n}^{-1}$ is the adjoint matrix of $\widetilde{A}_{n}$. A direct calculation gives the entry in the $j$-th row and $k$-th column of $\det(\widetilde{A}_{n}) \widetilde{A}_{n}^{-1}$ as
\begin{equation*}
(-1)^{j+k} \, q_{j\wedge k -1} (n \lambda) \, (-\alpha)^{|k-j|} \, p_{n- j \vee k}(n \lambda)
= \alpha^{|k-j|} \, q_{j\wedge k -1} (n \lambda) \, p_{n- j \vee k}(n \lambda),
\end{equation*}
where $j\wedge k := \min(j,k)$ and $j\vee k := \max(j,k)$. Hence,
\begin{eqnarray}
&& b_n^{\intercal} \left(\det(\widetilde{A}_{n}) \widetilde{A}_{n}^{-1}\right) b_n
= \sum_{j,k=1}^{n} \alpha^{|k-j|} \, q_{j\wedge k -1} (n \lambda) \, p_{n- j \vee k}(n \lambda) \notag \\
&=& \sum_{j,k=1}^{n} \alpha^{|k-j|} \, \left( p_{j\wedge k -1} (n \lambda) - \alpha^{2} \,  p_{j\wedge k -2} (n \lambda)\right) \, p_{n- j \vee k}(n \lambda) \notag \\
&=& \sum_{j,k=1}^{n} \alpha^{|k-j|} \, p_{j\wedge k -1} (n \lambda) \, p_{n- j \vee k}(n \lambda)
- \alpha^2 \sum_{j,k=1}^{n} \alpha^{|k-j|} \, p_{j\wedge k -2} (n \lambda) \, p_{n- j \vee k}(n \lambda).  \label{eq:bAb}
\end{eqnarray}
It follows immediately that $b_n^{\intercal} \left(\det(\widetilde{A}_{n}) \widetilde{A}_{n}^{-1}\right) b_n$ is a continuous function of $\lambda$. Hence, the right-hand side of \eqref{eq:d_nexpressionbymatrix} is a continuous function of $\lambda$. Further, the left-hand side of \eqref{eq:d_nexpressionbymatrix} is also a continuous function of $\lambda$. We may thus conclude that \eqref{eq:d_nexpressionbymatrix} holds for all $\lambda$.

In what follows, we let $\gamma_{1}(n \lambda)$, $\gamma_{2}(n \lambda)$ and $\Delta(n\lambda)$ be abbreviated as $\gamma_{1,n}$, $\gamma_{2,n}$ and $\Delta_n$ respectively. Then
\begin{equation*}
p_{k}(n \lambda) = \frac{\gamma_{1,n}^{k+1} - \gamma_{2,n}^{k+1}}{\sqrt{\Delta_n}}.
\end{equation*}
Hence,
\begin{eqnarray}
&& \sum_{j,k=1}^{n} \alpha^{|k-j|} \, p_{j\wedge k -1} (n \lambda) \, p_{n- j \vee k}(n \lambda)  \notag  \\
&=& \sum_{j,k=1}^{n} \alpha^{|k-j|} \, \left(\frac{\gamma_{1,n}^{j\wedge k} - \gamma_{2,n}^{j\wedge k}}{\sqrt{\Delta_n}}\right) \, 
\left(\frac{\gamma_{1,n}^{n +1 - j \vee k} - \gamma_{2,n}^{n + 1- j \vee k}}{\sqrt{\Delta_n}}\right) \notag  \\
&=& \frac{1}{\Delta_n}\, \sum_{j,k=1}^{n} \alpha^{|k-j|} \,\left( \gamma_{1,n}^{n+1 - |k-j|} + \gamma_{2,n}^{n+1 - |k-j|}  -  \gamma_{1,n}^{ j\wedge k } \,\gamma_{2,n}^{n - j\vee k +1}- \gamma_{2,n}^{ j\wedge k } \, \gamma_{1,n}^{n - j\vee k +1} \right)  \notag  \\
&=& \frac{1}{\Delta_n}\, \sum_{j,k=1}^{n} \alpha^{|k-j|} \, \gamma_{1,n}^{n+1 - |k-j|}
 + \frac{1}{\Delta_n}\, \sum_{j,k=1}^{n} \alpha^{|k-j|} \, \gamma_{2,n}^{n+1 - |k-j|}  \notag  \\
 && - \frac{1}{\Delta_n}\, \sum_{j,k=1}^{n} \alpha^{|k-j|} \, \left( \gamma_{1,n}^{ j\wedge k } \,\gamma_{2,n}^{n - j\vee k +1} + \gamma_{2,n}^{ j\wedge k } \, \gamma_{1,n}^{n - j\vee k +1} \right). \label{eq:app0}
\end{eqnarray}
Further,
\begin{eqnarray}
&& \sum_{j,k=1}^{n} \alpha^{|k-j|} \, \gamma_{1,n}^{n+1 - |k-j|}
= n \, \gamma_{1,n}^{n+1} + 2 \sum_{i=1}^{n-1} (n-i) \alpha^{i} \gamma_{1,n}^{n+1- i}  \notag \\
&=& n \, \gamma_{1,n}^{n+1} + 2 \sum_{i=1}^{n-1} i \, \alpha^{n-i} \gamma_{1,n}^{i+1}
= n \, \gamma_{1,n}^{n+1} + 2  \alpha^{n} \gamma_{1,n} \sum_{i=1}^{n-1} i \left( \gamma_{1,n}/ \alpha \right)^{i} \notag \\
&=& n \, \gamma_{1,n}^{n+1} + 2 \alpha^{n} \gamma_{1,n} \frac{\gamma_{1,n}/ \alpha  + (n-1) \left(\gamma_{1,n}/ \alpha \right)^{n+1} - n \left(\gamma_{1,n}/ \alpha \right)^n}{(1- \gamma_{1,n}/ \alpha )^2}  \notag \\ [1mm]
&=& n \, \gamma_{1,n}^{n+1} + 2 \alpha \gamma_{1,n}^{2}\, \frac{\alpha^n  + (n-1)  \gamma_{1,n}^{n} - n \alpha \gamma_{1,n}^{n-1} }{(\gamma_{1,n} -\alpha)^2},\label{eq:app1}
\end{eqnarray}
where the second equality holds by making the change of variables $i = n-i$ and the fourth equality follows by the fact that $\sum_{i=1}^{n-1} i x^{i} = \left(x + (n-1) x^{n+1} - n x^n\right)/(1-x)^2$. Similarly,
\begin{equation}
\sum_{j,k=1}^{n} \alpha^{|k-j|} \, \gamma_{2,n}^{n+1 - |k-j|}
= n \, \gamma_{2,n}^{n+1} + 2 \alpha \gamma_{2,n}^{2}\, \frac{\alpha^n  + (n-1)  \gamma_{2,n}^{n} - n \alpha \gamma_{2,n}^{n-1} }{(\gamma_{2,n} -\alpha)^2}. \label{eq:app2}
\end{equation}
Noting that $\gamma_{1,n} \gamma_{2,n} = \alpha^2$, 
\begin{eqnarray}
&& \sum_{j,k=1}^{n} \alpha^{|k-j|} \, \left( \gamma_{1,n}^{ j\wedge k } \,\gamma_{2,n}^{n - j\vee k +1} + \gamma_{2,n}^{ j\wedge k } \, \gamma_{1,n}^{n - j\vee k +1} \right)  \notag \\
&=& \sum_{j,k=1}^{n} \alpha^{|k-j|} \, \gamma_{1,n}^{ j\wedge k } \gamma_{2,n}^{ j\wedge k } \left( \gamma_{1,n}^{n - j\vee k +1 - j\wedge k} +  \gamma_{2,n}^{n - j\vee k +1 - j\wedge k} \right)  \notag  \\
&=& \sum_{j,k=1}^{n} \alpha^{|k-j|} \, \alpha^{2( j\wedge k )}  \left( \gamma_{1,n}^{n + 1 -j -k} +  \gamma_{2,n}^{n+1 -j -k} \right) \notag  \\
&=& \sum_{j,k=1}^{n} \alpha^{j+k}  \left( \gamma_{1,n}^{n + 1 -j -k} +  \gamma_{2,n}^{n+1 -j -k} \right) \notag \\
&=& \gamma_{1,n}^{n+1} \sum_{j,k=1}^{n} \left( \alpha/\gamma_{1,n} \right)^{j+k}  +   \gamma_{2,n}^{n+1} \sum_{j,k=1}^{n} \left( \alpha/\gamma_{2,n} \right)^{j+k}  \notag \\
&=& \gamma_{1,n}^{n+1} \left( \sum_{k=1}^{n} \left( \alpha/\gamma_{1,n} \right)^{k} \right)^{2}
  + \gamma_{2,n}^{n+1} \left( \sum_{k=1}^{n} \left( \alpha/\gamma_{2,n} \right)^{k} \right)^{2}  \notag  \\
&=& \gamma_{1,n}^{n+1} \left( \frac{(\alpha/ \gamma_{1,n})(1- (\alpha/ \gamma_{1,n})^{n})}{1- \alpha/ \gamma_{1,n}} \right)^2
  + \gamma_{2,n}^{n+1} \left( \frac{(\alpha/ \gamma_{2,n})(1- (\alpha/ \gamma_{2,n})^{n})}{1- \alpha/ \gamma_{2,n}} \right)^2  \notag  \\
&=& \frac{\alpha^2 \gamma_{1,n}^{n+1} - 2\alpha^{n+2} \gamma_{1,n} + \alpha^{2n+2} \gamma_{1,n}^{-(n-1)} }{(\gamma_{1,n} -\alpha)^2}  
  + \frac{\alpha^2 \gamma_{2,n}^{n+1} - 2\alpha^{n+2} \gamma_{2,n} + \alpha^{2n+2} \gamma_{2,n}^{-(n-1)} }{(\gamma_{2,n} -\alpha)^2}. \label{eq:app3}
\end{eqnarray}
Combining \eqref{eq:app0}, \eqref{eq:app1}, \eqref{eq:app2}, \eqref{eq:app3} with the fact that $\left( \gamma_{1,n} -\alpha \right)^2 /\gamma_{1,n} = \left( \gamma_{2,n} -\alpha \right)^2 /\gamma_{2,n} = (1-\alpha)^2 - n\lambda$, it follows (after rearrangement of terms) that
\begin{eqnarray}
&& \sum_{j,k=1}^{n} \alpha^{|k-j|} \, p_{j\wedge k -1} (n \lambda) \, p_{n- j \vee k}(n \lambda)  \notag  \\
&=&  n  \frac{\gamma_{1,n}^{n+1} + \gamma_{2,n}^{n+1}}{\Delta_n}  
+ \frac{2 \alpha^{n+1}  (\gamma_{1,n} + \gamma_{2,n}) }{ \Delta_n \left((1-\alpha)^2 - n\lambda\right)} 
+ \frac{2 (n-1) \alpha  (\gamma_{1,n}^{n+1} + \gamma_{2,n}^{n+1}) }{\Delta_n \left((1-\alpha)^2 - n\lambda\right)}  \notag  \\
&& - \frac{2 (n+1) \alpha^2  (\gamma_{1,n}^{n} + \gamma_{2,n}^{n}) }{\Delta_n \left((1-\alpha)^2 - n\lambda\right)}
 + \frac{4 \alpha^{n+2}}{\Delta_n \left((1-\alpha)^2 - n\lambda\right)}.  \label{eq:napp}
\end{eqnarray}
Similarly,
\begin{eqnarray}
&& \sum_{j,k=1}^{n} \alpha^{|k-j|} \, p_{j\wedge k -2} (n \lambda) \, p_{n- j \vee k}(n \lambda)  \notag   \\
&=&  (n-1)  \frac{\gamma_{1,n}^{n} + \gamma_{2,n}^{n}}{\Delta_n}  
+ \frac{2 \alpha^{n}  (\gamma_{1,n} + \gamma_{2,n}) }{ \Delta_n \left((1-\alpha)^2 - n\lambda\right)} 
+ \frac{2 (n-2) \alpha  (\gamma_{1,n}^{n} + \gamma_{2,n}^{n}) }{\Delta_n \left((1-\alpha)^2 - n\lambda\right)}  \notag  \\
&& - \frac{2 n \alpha^2  (\gamma_{1,n}^{n-1} + \gamma_{2,n}^{n-1}) }{\Delta_n \left((1-\alpha)^2 - n\lambda\right)}
 + \frac{4 \alpha^{n+1}}{\Delta_n \left((1-\alpha)^2 - n\lambda\right)}. \label{eq:n-1app}
\end{eqnarray}
Combining \eqref{eq:bAb}, \eqref{eq:napp} and \eqref{eq:n-1app} yields
\begin{eqnarray*}
&& b_n^{\intercal} \left(\det(\widetilde{A}_{n}) \widetilde{A}_{n}^{-1}\right) b_n  \\
&=& n  \frac{\gamma_{1,n}^{n+1} + \gamma_{2,n}^{n+1}}{\Delta_n} 
 - (n-1)  \frac{ \alpha^2 \left(\gamma_{1,n}^{n} + \gamma_{2,n}^{n}\right)}{\Delta_n}  
 + \frac{2 \alpha^{n+1} (1-\alpha) (\gamma_{1,n} + \gamma_{2,n}) }{ \Delta_n \left((1-\alpha)^2 - n\lambda\right)}  \\ [1mm]
&&  + \frac{2 (n-1) \alpha  (\gamma_{1,n}^{n+1} + \gamma_{2,n}^{n+1}) }{\Delta_n \left((1-\alpha)^2 - n\lambda\right)}
  - \frac{2 (n-2) \alpha^3  (\gamma_{1,n}^{n} + \gamma_{2,n}^{n}) }{\Delta_n \left((1-\alpha)^2 - n\lambda\right)}  \\ [1mm]
&& - \frac{2 (n+1) \alpha^2  (\gamma_{1,n}^{n} + \gamma_{2,n}^{n}) }{\Delta_n \left((1-\alpha)^2 - n\lambda\right)}
+ \frac{2 n \alpha^4  (\gamma_{1,n}^{n-1} + \gamma_{2,n}^{n-1}) }{\Delta_n \left((1-\alpha)^2 - n\lambda\right)}
+ \frac{4 \alpha^{n+2} (1-\alpha)}{\Delta_n \left((1-\alpha)^2 - n\lambda\right)}.
\end{eqnarray*}
Together with \eqref{eq:d_nexpressionbymatrix} and the fact that
\begin{equation*}
\det( \widetilde{A}_n ) = q_{n}(n\lambda) = p_{n}(n\lambda) - \alpha^2\, p_{n-1}(n\lambda)
= \frac{\gamma_{1,n}^{n+1} - \gamma_{2,n}^{n+1}}{\sqrt{\Delta_n}} - \frac{\alpha^2 \left(\gamma_{1,n}^{n} - \gamma_{2,n}^{n}\right)}{\sqrt{\Delta_n}},
\end{equation*}
we have that
\begin{eqnarray*}
d_{n}(n^2 \lambda) &=& \frac{\gamma_{1,n}^{n+1} - \gamma_{2,n}^{n+1}}{\sqrt{\Delta_n}} - \frac{\alpha^2 \left(\gamma_{1,n}^{n} - \gamma_{2,n}^{n}\right)}{\sqrt{\Delta_n}}  
+ \frac{2 \alpha^{n+1} (1-\alpha) \, \lambda \, (\gamma_{1,n} + \gamma_{2,n}) }{ \Delta_n \left((1-\alpha)^2 - n\lambda\right)}\\ [2mm]
&& + n  \, \frac{ \lambda\left(\gamma_{1,n}^{n+1} + \gamma_{2,n}^{n+1}\right)}{\Delta_n} 
 - (n-1) \, \frac{ \alpha^2 \, \lambda \left(\gamma_{1,n}^{n} + \gamma_{2,n}^{n}\right)}{\Delta_n}    \\ [2mm]
&&  + \frac{2 (n-1) \alpha \, \lambda \, (\gamma_{1,n}^{n+1} + \gamma_{2,n}^{n+1}) }{\Delta_n \left((1-\alpha)^2 - n\lambda\right)}
  - \frac{2 (n-2) \alpha^3 \, \lambda \, (\gamma_{1,n}^{n} + \gamma_{2,n}^{n}) }{\Delta_n \left((1-\alpha)^2 - n\lambda\right)}  \\ [2mm]
&& - \frac{2 (n+1) \alpha^2 \, \lambda \,  (\gamma_{1,n}^{n} + \gamma_{2,n}^{n}) }{\Delta_n \left((1-\alpha)^2 - n\lambda\right)}
+ \frac{2 n \alpha^4 \, \lambda \, (\gamma_{1,n}^{n-1} + \gamma_{2,n}^{n-1}) }{\Delta_n \left((1-\alpha)^2 - n\lambda\right)}
+ \frac{4 \alpha^{n+2} (1-\alpha) \, \lambda}{\Delta_n \left((1-\alpha)^2 - n\lambda\right)}.
\end{eqnarray*}
By making change of variables $\lambda = n^2 \lambda$, it follows immediately that
\begin{eqnarray*}
d_{n}(\lambda) &=& \frac{\left(\gamma_{1}(\lambda/n)\right)^{n+1}  - \left(\gamma_{2}(\lambda/n)\right)^{n+1}}{\sqrt{\Delta(\lambda/n)}}
-  \frac{\alpha^2\left[\left(\gamma_{1}(\lambda/n)\right)^{n}  - \left(\gamma_{2}(\lambda/n)\right)^{n}\right]}{\sqrt{\Delta(\lambda/n)}} \\ [2mm]
&& + \frac{\lambda\left[\left(\gamma_{1}(\lambda/n)\right)^{n+1}  + \left(\gamma_{2}(\lambda/n)\right)^{n+1}\right]}{n \,\Delta(\lambda/n)}  
  - \frac{(n-1) \alpha^2 \, \lambda\left[\left(\gamma_{1}(\lambda/n)\right)^{n}  + \left(\gamma_{2}(\lambda/n)\right)^{n}\right]}{n^2 \,\Delta(\lambda/n)}  \\ [2mm]
&& + \frac{2 (n-1) \alpha \, \lambda\left[\left(\gamma_{1}(\lambda/n)\right)^{n+1}  + \left(\gamma_{2}(\lambda/n)\right)^{n+1}\right]}{n \,\Delta(\lambda/n) \left( n (1-\alpha)^2 - \lambda \right)}
- \frac{2 (n-2) \alpha^3 \, \lambda\left[\left(\gamma_{1}(\lambda/n)\right)^{n}  + \left(\gamma_{2}(\lambda/n)\right)^{n}\right]}{n \,\Delta(\lambda/n) \left( n (1-\alpha)^2 - \lambda \right)}  \\  [2mm]
&& - \frac{2 (n+1) \alpha^2 \, \lambda\left[\left(\gamma_{1}(\lambda/n)\right)^{n}  + \left(\gamma_{2}(\lambda/n)\right)^{n}\right]}{n \,\Delta(\lambda/n) \left( n (1-\alpha)^2 - \lambda \right)}
+ \frac{2  \alpha^4 \, \lambda\left[\left(\gamma_{1}(\lambda/n)\right)^{n-1}  + \left(\gamma_{2}(\lambda/n)\right)^{n-1}\right]}{\Delta(\lambda/n) \left( n (1-\alpha)^2 - \lambda \right)} \\[2mm]
&& + \frac{2 \, \alpha^{n+1} (1- \alpha) \, \lambda\left[\gamma_{1}(\lambda/n)  + \gamma_{2}(\lambda/n)\right]}{n \,\Delta(\lambda/n) \left( n (1-\alpha)^2 - \lambda \right)}  
+ \frac{4 \, \alpha^{n+2} (1- \alpha) \, \lambda}{n \,\Delta(\lambda/n) \left( n (1-\alpha)^2 - \lambda \right)}.
\end{eqnarray*}

\newpage

\section{Appendix B}
The purpose of this section is to provide the proofs of Lemma \ref{lm:lambda2}, Lemma \ref{lm:lambda4}, and Lemma \ref{lm:productlambda}.

\subsection{Proof of Lemma \ref{lm:lambda2}}
We first recall the definition of $K_n$ as presented in \eqref{eq:defKn} and note that $\lambda_1, \dots, \lambda_n$ are the eigenvalues of $K_n$.
It is immediate that
\begin{eqnarray}
&&\sum_{k=1}^{n} \lambda_{k}^{2} = \mathrm{tr} (K_{n}^{2}) 
= \sum_{j,k=1}^{n} \left[ \frac{1}{n}\, \frac{\alpha^{|k-j|} - \alpha^{k+j}}{1- \alpha^2}  - \frac{1}{n^2} \, \frac{(1-\alpha^k)(1- \alpha^{j})}{(1- \alpha)^2}\right]^2 \notag  \\ [1mm]
&=& \frac{1}{n^2} \sum_{j,k=1}^{n} \frac{\left(\alpha^{|k-j|} - \alpha^{k+j}\right)^2}{(1- \alpha^2)^2}
- \frac{2}{n^3} \, \sum_{j,k=1}^{n} \frac{(\alpha^{|k-j|} - \alpha^{k+j})(1- \alpha^k)(1- \alpha^j)}{(1- \alpha)^2 ( 1- \alpha^2)} \notag \\ [1mm]
&& + \frac{1}{n^4} \sum_{j,k=1}^{n} \frac{(1- \alpha^k)^2 (1-\alpha^j)^2}{(1-\alpha)^4}. \label{eq:lambda2main}
\end{eqnarray}
Further
\begin{eqnarray}
&& \sum_{j,k=1}^{n} \frac{\left(\alpha^{|k-j|} - \alpha^{k+j}\right)^2}{(1- \alpha^2)^2}
= \sum_{j,k=1}^{n}\frac{\alpha^{2|k-j|} - 2 \alpha^{|k-j|+k+j} + \alpha^{2 k + 2j} }{(1-\alpha^2)^2}  \notag \\ [1mm]
&=& \sum_{j,k=1}^{n} \frac{\alpha^{2|k-j|}}{(1- \alpha^2)^2} - 2 \sum_{j,k=1}^{n} \frac{\alpha^{|k-j|+k+j}}{(1-\alpha^2)^2}
   + \sum_{j,k=1}^{n} \frac{\alpha^{2k+2j}}{(1-\alpha^2)^2}.  \label{eq:lambda2term1}
\end{eqnarray}
Lengthy but direct calculations yield
\begin{eqnarray*}
&&\sum_{j,k=1}^{n} \frac{\alpha^{2|k-j|}}{(1- \alpha^2)^2}
= \sum_{k=1}^{n}\frac{1}{(1-\alpha^2)^2} + 2 \sum_{k=1}^{n} \sum_{j=1}^{k-1} \frac{\alpha^{2(k-j)}}{(1-\alpha^2)^2}  \\
&=& \frac{n}{(1- \alpha^2)^2} + \frac{2 n \, \alpha^2}{(1- \alpha^2)^3} - \frac{2\alpha^2 (1- \alpha^{2n})}{(1- \alpha^2)^4}  
= \frac{n (1 + \alpha^2)}{(1- \alpha^2)^3} - \frac{2\alpha^2 (1- \alpha^{2n})}{(1- \alpha^2)^4},  
\end{eqnarray*}
\begin{eqnarray*}
&& \sum_{j,k=1}^{n} \frac{\alpha^{|k-j|+k+j}}{(1-\alpha^2)^2}
= \sum_{k=1}^{n} \frac{\alpha^{2k}}{(1- \alpha^2)^2} + 2 \sum_{k=1}^{n} \sum_{j=1}^{k-1} \frac{\alpha^{2k}}{(1-\alpha^2)^2}  \\
&=& \frac{\alpha^2 ( 1- \alpha^{2n})}{(1- \alpha^2)^3} + 2 \sum_{k=1}^{n} \frac{(k-1) \alpha^{2k}}{(1- \alpha^2)^2}
= \frac{\alpha^2 ( 1- \alpha^{2n})}{(1- \alpha^2)^3} + \frac{2 \alpha^4 ( 1- \alpha^{2n})}{(1- \alpha^2)^4} - \frac{2 n \alpha^{2n+2}}{(1- \alpha^2)^3},
\end{eqnarray*}
and
\begin{eqnarray*}
\sum_{j,k=1}^{n} \frac{\alpha^{2k+2j}}{(1-\alpha^2)^2}
= \left( \sum_{j=1}^{n} \frac{\alpha^{2j}}{1- \alpha^2} \right)  \left( \sum_{k=1}^{n} \frac{\alpha^{2k}}{1- \alpha^2} \right)
 = \frac{\alpha^4 (1- \alpha^{2n})^2}{(1- \alpha^2)^4}.
\end{eqnarray*}
Together with \eqref{eq:lambda2term1}, it follows (after rearrangement of terms) that
\begin{eqnarray}
&& \sum_{j,k=1}^{n} \frac{\left(\alpha^{|k-j|} - \alpha^{k+j}\right)^2}{(1- \alpha^2)^2}  \notag \\
&=& \frac{n ( 1+ \alpha^2)}{(1- \alpha^2)^3} + \frac{4n \alpha^{2n+2}}{(1-\alpha^2)^3} - \frac{4\alpha^2 + \alpha^{4} - 4 \alpha^{2n+2}  - \alpha^{4n+4}}{(1- \alpha^2)^4}.  \label{eq:lambda2term1final}
\end{eqnarray}
Combining \eqref{eq:lambda2main} and \eqref{eq:lambda2term1final} yields
\begin{eqnarray*}
\sum_{k=1}^{n} \lambda_{k}^{2}
&=& \frac{( 1+ \alpha^2)}{n(1- \alpha^2)^3} + \frac{1}{n^2} \Bigg[ \frac{4n \alpha^{2n+2}}{(1-\alpha^2)^3} - \frac{4\alpha^2 + \alpha^{4} - 4 \alpha^{2n+2}  - \alpha^{4n+4}}{(1- \alpha^2)^4} \\
&&  \quad- \frac{2}{n} \, \sum_{j,k=1}^{n} \frac{(\alpha^{|k-j|} - \alpha^{k+j})(1- \alpha^k)(1- \alpha^j)}{(1- \alpha)^2 ( 1- \alpha^2)} 
+ \frac{1}{n^2} \sum_{j,k=1}^{n} \frac{(1- \alpha^k)^2 (1-\alpha^j)^2}{(1-\alpha)^4}\Bigg].
\end{eqnarray*}
Let
\begin{eqnarray}
\kappa_2(n) &:=& \frac{4n \alpha^{2n+2}}{(1-\alpha^2)^3} - \frac{4\alpha^2 + \alpha^{4} - 4 \alpha^{2n+2}  - \alpha^{4n+4}}{(1- \alpha^2)^4} \notag \\
&& - \frac{2}{n} \, \sum_{j,k=1}^{n} \frac{(\alpha^{|k-j|} - \alpha^{k+j})(1- \alpha^k)(1- \alpha^j)}{(1- \alpha)^2 ( 1- \alpha^2)} 
+ \frac{1}{n^2} \sum_{j,k=1}^{n} \frac{(1- \alpha^k)^2 (1-\alpha^j)^2}{(1-\alpha)^4}. \label{eq:defkappa}
\end{eqnarray}
Now, we are only left with the task of showing the boundedness of $\kappa_2(n)$.

Since $|\alpha|<1$, $n \alpha^{2n+2}$ tends to $0$ as $n \rightarrow \infty$. It follows that $\sup_{n} n \alpha^{2n+2} < \infty$. Hence the absolute value of the first term on the right-hand side of \eqref{eq:defkappa} is less than or equal to
\begin{equation}
\frac{4}{(1- \alpha^2)^3} \, \sup_{n \in \mathds{N}_{+}} n \alpha^{2n+2}. \label{eq:kammaboundedness1}
\end{equation}
It is easy to check that
\begin{equation*}
0 \leq 4\alpha^2 + \alpha^{4} - 4 \alpha^{2n+2}  - \alpha^{4n+4} \leq 4 \alpha^2 + \alpha^4.
\end{equation*}
Then
\begin{equation}
\left| \frac{4\alpha^2 + \alpha^{4} - 4 \alpha^{2n+2}  - \alpha^{4n+4}}{(1- \alpha^2)^4}  \right|
 \leq \frac{4 \alpha^2 + \alpha^4}{(1- \alpha^2)^4}.   \label{eq:kammaboundedness2}
\end{equation}
Noting that
\begin{equation*}
\left|  \alpha^{|k-j|} - \alpha^{j+k}  \right| = |\alpha|^{|k-j|} - |\alpha|^{k+j} \leq |\alpha|^{|k-j|}
\end{equation*}
and $$|(1- \alpha^{k})(1- \alpha^j)| \leq 4,$$ it follows that
\begin{eqnarray}
&& \left| \frac{2}{n} \, \sum_{j,k=1}^{n} \frac{(\alpha^{|k-j|} - \alpha^{k+j})(1- \alpha^k)(1- \alpha^j)}{(1- \alpha)^2 ( 1- \alpha^2)}  \right| 
\leq \frac{8}{n} \, \sum_{j,k=1}^{n} \frac{|\alpha|^{|k-j|}}{(1- \alpha)^2 ( 1- \alpha^2)}  \notag \\
&=& \frac{8}{n} \, \sum_{k=1}^{n} \frac{1}{(1- \alpha)^2 ( 1- \alpha^2)} + \frac{16}{n} \, \sum_{k=1}^{n} \sum_{j=1}^{k-1} \frac{|\alpha|^{k-j}}{(1- \alpha)^2 ( 1- \alpha^2)} \notag \\ [1mm]
&=& \frac{8}{(1- \alpha)^2 ( 1- \alpha^2)} + \frac{16|\alpha|}{(1- |\alpha|)(1- \alpha)^2 ( 1- \alpha^2)} - \frac{16}{n}\, \frac{|\alpha| (1- |\alpha|^{n})}{(1-|\alpha|)^2 (1- \alpha)^2 ( 1- \alpha^2)} \notag  \\ [1mm]
&\leq& \frac{8}{(1- \alpha)^2 ( 1- \alpha^2)} + \frac{16|\alpha|}{(1- |\alpha|)(1- \alpha)^2 ( 1- \alpha^2)}.  \label{eq:kammaboundedness3}
\end{eqnarray}
Finally,
\begin{eqnarray}
&& \frac{1}{n^2} \sum_{j,k=1}^{n} \frac{(1- \alpha^k)^2 (1-\alpha^j)^2}{(1-\alpha)^4}
= \frac{1}{n^2} \left( \sum_{k=1}^{n} \frac{(1- \alpha^k)^2}{(1- \alpha)^2} \right)^2  \notag \\
&\leq& \frac{1}{n^2} \left( \sum_{k=1}^{n} \frac{2^2}{(1- \alpha)^2} \right)^2 
= \frac{16}{(1- \alpha)^4}. \label{eq:kammaboundedness4}
\end{eqnarray}
The desired result follows by combining \eqref{eq:defkappa}--\eqref{eq:kammaboundedness4}.

\subsection{Proof of Lemma \ref{lm:lambda4}}

For ease of notation, we use $u_{kj}$ to denote the entry of $K_n$ in the $k$-th row and the $j$-th column. That is,
\begin{equation*}
u_{kj} = \frac{1}{n} \, \frac{\alpha^{|k-j|} - \alpha^{k+j}}{1- \alpha^2} - \frac{1}{n^2} \, \frac{(1-\alpha^k)(1- \alpha^j)}{(1-\alpha)^2}.
\end{equation*}
Let
\begin{equation*}
g_{n}(k) := \left(\frac{|\alpha|^{|k|}}{n (1- \alpha^2)} + \frac{4}{n^2 (1- \alpha)^2}\right) \mathds{1}_{\{ |k| <n \}}
\end{equation*}
be a symmetric positive function on $\mathds{Z}$. Noting that
\begin{equation*}
\left| \alpha^{|k-j|}  - \alpha^{k+j}\right| = |\alpha|^{|k-j|} - |\alpha|^{j+k} \leq |\alpha|^{|k-j|},
\end{equation*}
and
\begin{equation*}
\left| (1-\alpha^k) (1- \alpha^j) \right| \leq 4,
\end{equation*}
we have
\begin{eqnarray*}
&& |u_{kj}| = \left| \frac{1}{n}\, \frac{\alpha^{|k-j|} - \alpha^{k+j}}{1- \alpha^2}  - \frac{1}{n^2} \, \frac{(1-\alpha^k)(1- \alpha^{j})}{(1- \alpha)^2} \right|  \\
&\leq& \frac{|\alpha|^{|k-j|}}{n (1- \alpha^2)} + \frac{4}{n^2 (1- \alpha)^2} = g_{n}(k-j).
\end{eqnarray*}
Then
\begin{eqnarray}
&& \sum_{k=1}^{n} \lambda_{k}^{4} = \mathrm{tr}(K_{n}^{4}) = \sum_{k,j,i,l=1}^{n} \, u_{kj} \, u_{ji} \, u_{il} \, u_{lk} 
\leq \sum_{k,j,i,l=1}^{n} \, |u_{kj}| \, |u_{ji}| \, |u_{il}| \, |u_{lk}|  \notag  \\
&\leq& \sum_{k,j,i,l=1}^{n} \, g_{n}(k-j)   g_{n}(j-i)  g_{n}(i-l)  g_{n}(l-k)  \notag \\
&=& \sum_{k,j,i,l=1}^{n} \, g_{n}(k-j)   g_{n}(j-i)  g_{n}(l-i)  g_{n}(k-l) \notag \\
&=& \sum_{i=1}^{n} \sum_{k=1}^{n} \, \left(\sum_{j=1}^{n} \, g_{n}(k-j)   g_{n}(j-i)\right)  \left( \sum_{l=1}^{n} \, g_{n}(l-i)  g_{n}(k-l) \right) \notag \\
&\leq& \sum_{i=1}^{n} \sum_{k=1}^{n} \, \left(g_{n} * g_{n} (k-i)\right)^2  
\leq \sum_{i=1}^{n} \sum_{k=-\infty}^{\infty} \,  \left(g_{n} * g_{n} (k)\right)^2 \notag  \\
&=& n \sum_{k=-\infty}^{\infty} \,  \left(g_{n} * g_{n} (k)\right)^2,  \label{eq:lambda4}
\end{eqnarray}
where the star $*$ denotes the convolution operator, i.e., $g_{n}*g_{n} (j) = \sum_{k= - \infty}^{\infty} g_{n}(j-k) g_{n}(k)$. Applying Young's convolution inequality yields
\begin{equation*}
\left( \sum_{k=-\infty}^{\infty} \,  \left(g_{n} * g_{n} (k)\right)^2\right)^{\frac{1}{2}}
\leq \left( \sum_{k=-\infty}^{\infty} \left( g_{n}(k) \right)^{\frac{4}{3}} \right)^{\frac{3}{2}}.
\end{equation*}
Then
\begin{eqnarray*}
&& \sum_{k=-\infty}^{\infty} \,  \left(g_{n} * g_{n} (k)\right)^2
\leq \left( \sum_{k=-\infty}^{\infty} \left( g_{n}(k) \right)^{\frac{4}{3}} \right)^{3}  \\
&=& \left(\sum_{k= -(n-1)}^{n-1} \left( \frac{|\alpha|^{|k|}}{n (1- \alpha^2)} + \frac{4}{n^2 (1- \alpha)^2} \right)^{\frac{4}{3}}\right)^{3}  \\
&\leq& \left( \sum_{k= -(n-1)}^{n-1} \left(  2^{\frac{1}{3}} \frac{|\alpha|^{4|k|/3}}{n^{4/3} (1- \alpha^2)^{4/3}} + 2^{\frac{1}{3}} \frac{4^{4/3}}{n^{8/3} (1- \alpha)^{8/3}} \right) \right)^{3}  \\
& \leq & \left( 2 \sum_{k= 0}^{n-1} \left(  2^{\frac{1}{3}} \frac{|\alpha|^{4|k|/3}}{n^{4/3} (1- \alpha^2)^{4/3}} + 2^{\frac{1}{3}} \frac{4^{4/3}}{n^{8/3} (1- \alpha)^{8/3}} \right) \right)^{3} \\
&=& \left( 2^{\frac{4}{3}} \, \frac{1- |\alpha|^{4n/3}}{n^{4/3}  (1- \alpha^2)^{4/3} ( 1- |\alpha|^{4/3}) }  + \frac{16}{n^{5/3} (1- \alpha)^{8/3}} \right)^{3}  \\
&\leq& \left( n^{-\frac{4}{3}} \, \left( \frac{2^{4/3}}{ (1- \alpha^2)^{4/3} ( 1- |\alpha|^{4/3}) }  + \frac{16}{(1-\alpha)^{8/3}}\right)\right)^{3} \\
&=& n^{-4} \, \left( \frac{2^{4/3}}{ (1- \alpha^2)^{4/3} ( 1- |\alpha|^{4/3}) }  + \frac{16}{(1-\alpha)^{8/3}}\right)^3,
\end{eqnarray*}
where the first inequality follows by Jensen's inequality.
Together with \eqref{eq:lambda4}, the desired result follows.

\subsection{Proof of Lemma \ref{lm:productlambda}}
From the representation \eqref{eq:anotherformofacp} of $d_{n}(\lambda)$ and the fact that $\lambda_n =0$, we have
\begin{equation*}
d_{n}(\lambda) = (1 - \lambda_1 \, \lambda)(1- \lambda_2 \, \lambda) \dots (1- \lambda_{n-1} \, \lambda).
\end{equation*}
Then
\begin{equation*}
\prod_{k=1}^{n-1} \lambda_k = \lim_{\lambda \rightarrow \infty} \frac{d_{n}(-\lambda)}{\lambda^{n-1}}.
\end{equation*}
It follows from Lemma \ref{lem1} that
\begin{eqnarray}
&& d_{n}(-\lambda)  \notag \\
&=& \frac{\left(\gamma_{1}(-\lambda/n)\right)^{n+1}  - \left(\gamma_{2}(-\lambda/n)\right)^{n+1}}{\sqrt{\Delta(-\lambda/n)}}
-  \frac{\alpha^2\left[\left(\gamma_{1}(-\lambda/n)\right)^{n}  - \left(\gamma_{2}(-\lambda/n)\right)^{n}\right]}{\sqrt{\Delta(-\lambda/n)}} \notag \\ [2mm]
&& - \frac{\lambda\left[\left(\gamma_{1}(-\lambda/n)\right)^{n+1}  + \left(\gamma_{2}(-\lambda/n)\right)^{n+1}\right]}{n \,\Delta(-\lambda/n)}  
  + \frac{(n-1) \alpha^2 \, \lambda\left[\left(\gamma_{1}(-\lambda/n)\right)^{n}  + \left(\gamma_{2}(-\lambda/n)\right)^{n}\right]}{n^2 \,\Delta(-\lambda/n)} \notag  \\ [2mm]
&& - \frac{2 (n-1) \alpha \, \lambda\left[\left(\gamma_{1}(-\lambda/n)\right)^{n+1}  + \left(\gamma_{2}(-\lambda/n)\right)^{n+1}\right]}{n \,\Delta(-\lambda/n) \left( n (1-\alpha)^2 + \lambda \right)}
+ \frac{2 (n-2) \alpha^3 \, \lambda\left[\left(\gamma_{1}(-\lambda/n)\right)^{n}  + \left(\gamma_{2}(-\lambda/n)\right)^{n}\right]}{n \,\Delta(-\lambda/n) \left( n (1-\alpha)^2 + \lambda \right)} \notag \\  [2mm]
&& + \frac{2 (n+1) \alpha^2 \, \lambda\left[\left(\gamma_{1}(-\lambda/n)\right)^{n}  + \left(\gamma_{2}(-\lambda/n)\right)^{n}\right]}{n \,\Delta(-\lambda/n) \left( n (1-\alpha)^2 + \lambda \right)}
- \frac{2  \alpha^4 \, \lambda\left[\left(\gamma_{1}(-\lambda/n)\right)^{n-1}  + \left(\gamma_{2}(-\lambda/n)\right)^{n-1}\right]}{\Delta(-\lambda/n) \left( n (1-\alpha)^2 + \lambda \right)} \notag \\[2mm]
&& - \frac{2 \, \alpha^{n+1} (1- \alpha) \, \lambda\left[\gamma_{1}(-\lambda/n)  + \gamma_{2}(-\lambda/n)\right]}{n \,\Delta(-\lambda/n) \left( n (1-\alpha)^2 + \lambda \right)}  
- \frac{4 \, \alpha^{n+2} (1- \alpha) \, \lambda}{n \,\Delta(-\lambda/n) \left( n (1-\alpha)^2 + \lambda \right)}. \label{eq:productlambdamain}
\end{eqnarray}
A routine calculation gives that as $\lambda \rightarrow \infty$
\begin{eqnarray*}
 \gamma_{1}(-\lambda/n) &=& ( 1/n + o(1)) \lambda,  \\
 \gamma_{2}(-\lambda/n) &=&  (n \alpha^2 + o(1)) \lambda^{-1},  \\
 \Delta(- \lambda /n) &=& (1/n^2 + o(1)) \lambda^2, \\
 (n (1-\alpha)^2 + \lambda ) &=& (1 + o(1)) \lambda.
\end{eqnarray*}
Then
\begin{eqnarray}
&& \frac{\left(\gamma_{1}(-\lambda/n)\right)^{n+1}  - \left(\gamma_{2}(-\lambda/n)\right)^{n+1}}{\sqrt{\Delta(-\lambda/n)}}
- \frac{\lambda\left[\left(\gamma_{1}(-\lambda/n)\right)^{n+1}  + \left(\gamma_{2}(-\lambda/n)\right)^{n+1}\right]}{n \,\Delta(-\lambda/n)} \notag \\ [1mm]
&=& \frac{\left( \gamma_{1}(-\lambda/n) \right)^{n+1}}{\Delta(-\lambda/n)} \left( \sqrt{\Delta(-\lambda/n)} - \frac{\lambda}{n} \right)
 - \frac{\left( \gamma_{2}(-\lambda/n) \right)^{n+1}}{\Delta(-\lambda/n)} \left( \sqrt{\Delta(-\lambda/n)} + \frac{\lambda}{n} \right) \notag  \\ [1mm]
 &=& \frac{\left( \gamma_{1}(-\lambda/n) \right)^{n+1}}{\Delta(-\lambda/n)} \, \frac{ \Delta(-\lambda/n) - (\lambda/n)^{2} }{\sqrt{\Delta(-\lambda/n)} + \lambda/n}
 - \frac{\left( \gamma_{2}(-\lambda/n) \right)^{n+1}}{\Delta(-\lambda/n)} \left( \sqrt{\Delta(-\lambda/n)} + \frac{\lambda}{n} \right) \notag  \\ [1mm]
&=& \frac{\left( \gamma_{1}(-\lambda/n) \right)^{n+1}}{\Delta(-\lambda/n)} \, \frac{2(1+ \alpha^2) \lambda/n + (1- \alpha^2)^2}{\sqrt{\Delta(-\lambda/n)} + \lambda/n}
 - \frac{\left( \gamma_{2}(-\lambda/n) \right)^{n+1}}{\Delta(-\lambda/n)} \left( \sqrt{\Delta(-\lambda/n)} + \frac{\lambda}{n} \right) \notag \\ [1mm]
&=& \frac{\left((1/n)^{n+1} +o(1) \right) \lambda^{n+1} }{(1/n^2 + o(1)) \lambda^2} \, \frac{\left( 2(1+\alpha^2)/n + o(1) \right) \lambda}{(2/n + o(1))\lambda}
- \frac{\mathcal{O}( \lambda^{-(n+1)} )}{ (1/n^2 + o(1)) \lambda^2 } \times \mathcal{O}(\lambda)  \notag \\ [1mm]
&=& \left( (1+\alpha^2) n^{-(n-1)} + o(1) \right) \lambda^{n-1}. \label{eq:productlambda1}
\end{eqnarray}
Lengthy but routine calculations yield that 
\begin{eqnarray}
 \frac{\alpha^2\left[\left(\gamma_{1}(-\lambda/n)\right)^{n}  - \left(\gamma_{2}(-\lambda/n)\right)^{n}\right]}{\sqrt{\Delta(-\lambda/n)}}
 &=& \left(\alpha^2 n^{-(n-1)} + o(1) \right) \, \lambda^{n-1} , \label{eq:productlambda2} \\
 \frac{(n-1) \alpha^2 \, \lambda\left[\left(\gamma_{1}(-\lambda/n)\right)^{n}  + \left(\gamma_{2}(-\lambda/n)\right)^{n}\right]}{n^2 \,\Delta(-\lambda/n)} 
&=& \left(\alpha^2 (n-1) n^{-n} + o(1) \right) \lambda^{n-1}, \label{eq:productlambda3} \\
 \frac{2 (n-1) \alpha \, \lambda\left[\left(\gamma_{1}(-\lambda/n)\right)^{n+1}  + \left(\gamma_{2}(-\lambda/n)\right)^{n+1}\right]}{n \,\Delta(-\lambda/n) \left( n (1-\alpha)^2 + \lambda \right)}
 &=&  \left(2 \alpha (n-1) n^{-n} +o(1) \right) \, \lambda^{n-1}. \label{eq:productlambda4}
\end{eqnarray}
Further, the five remaining terms on the right-hand side of \eqref{eq:productlambdamain} are $ \mathcal{O}( \lambda^{n-2})$. Together with \eqref{eq:productlambdamain}--\eqref{eq:productlambda4}, we have
\begin{equation*}
\prod_{k=1}^{n-1} \lambda_k = \lim_{\lambda \rightarrow \infty} \frac{d_{n}(-\lambda)}{\lambda^{n-1}}
= n^{-(n-1)} \left( 1+ \alpha^2 \frac{n-1}{n} - 2\alpha \frac{n-1}{n} \right).
\end{equation*}
This completes the proof of the first assertion. We now turn to the second assertion. Note that
\begin{eqnarray*}
&& (n-1) \sqrt[n-1]{ \lambda_1 \lambda_2 \dots \lambda_{n-1}}
= \frac{n-1}{n} \, \sqrt[n-1]{1 + \alpha^2 \frac{n-1}{n} - 2 \alpha \frac{n-1}{n} }  \\
&=& \frac{n-1}{n} \sqrt[n-1]{ \frac{1}{n}  + (1 - \alpha)^2 \frac{n-1}{n}}
\geq \frac{n-1}{n}  \sqrt[n-1]{\frac{1}{n}} \geq \frac{1}{2}  \sqrt[n-1]{\frac{1}{n}} 
\geq \frac{1}{4},
\end{eqnarray*}
where the second last inequality is due to the fact $n \geq 2$ and the last inequality follows by applying the inequality $ \sqrt[n-1]{1/n} \geq 1/2$. For $n \geq 2$, this inequality is equivalent to $n \leq 2^{n-1}$. This completes the proof of the second assertion.

\newpage

\section{Appendix C}

This section is devoted to the proofs of Theorem \ref{thm:asymptoticsnumeratoralternative} and Theorem \ref{thm:denominatortailboundalternative}.

\subsection{Proof of Theorem \ref{thm:asymptoticsnumeratoralternative}}
As introduced in Section \ref{subsec:convergencerateunderHa}, $Z_{11}^{n}$, $Z_{12}^{n}$ and $Z_{22}^{n}$ can be represented as
\begin{equation*}
Z_{11}^{n} = \sum_{k=1}^{n} \lambda_{k} \, W_{k}^{2}, \  Z_{12}^{n} = \sum_{k=1}^{n} \lambda_{k} \, W_{k} V_{k}, \ \text{and } Z_{22}^{n} = \sum_{k=1}^{n} \lambda_{k} \, V_{k}^{2}
\end{equation*}
where the pairs of random variables $(W_n, V_{n})$ are i.i.d. Gaussian vectors with mean zero and covariance matrix $$\begin{pmatrix}
1 & r \\
r & 1
\end{pmatrix}.$$ There then exist orthonormal functions $e_{1}, e_{2}, \dots, e_{n}, f_{1}, f_{2}, \dots, f_{n}$ on $\mathbb{R}_{+}$ such that
\begin{eqnarray*}
&& \left( W_{1},  \dots, W_{n}, V_{1}, \dots, V_{n} \right) \\
&\overset{d}{=}& \left( \mathcal{I}_{1} (e_{1}),  \dots, \mathcal{I}_{1} (e_{n}), \mathcal{I}_{1} (r e_{1} + \sqrt{1- r^2} f_1),  \dots, \mathcal{I}_{1} (r e_{n} + \sqrt{1- r^2} f_n) \right).
\end{eqnarray*}

Before representing $Z_{11}^{n}$, $Z_{12}^{n}$, and $Z_{22}^n$ in the form of multiple Wiener integrals, we first define six symmetric functions $h_1, \dots, h_6$ on $\mathcal{H}^{\odot 2}$ and study their inner products and contractions. Indeed, $h_1, \dots, h_6$ are the integrands of the multiple Wiener integrals related to $Z_{11}^{n}$, $Z_{12}^{n}$, and $Z_{22}^n$. We define
\begin{eqnarray*}
&& h_1 := \sum_{k=1}^{n} \lambda_{k} \, e_{k} \otimes e_{k}, \quad
 h_2 := \sum_{k=1}^{n} \lambda_{k} \, f_{k} \otimes f_{k}, \quad
 h_3 := \sum_{k=1}^{n} \lambda_{k} \, e_{k} \, \widetilde{\otimes} \, f_{k}, \\
&& h_4 := \sum_{k=1}^{n} \lambda_{k}^{2} \, e_{k} \otimes e_{k}, \quad
 h_5 := \sum_{k=1}^{n} \lambda_{k}^{2} \, f_{k} \otimes f_{k}, \quad
 h_6 := \sum_{k=1}^{n} \lambda_{k}^{2} \, e_{k} \, \widetilde{\otimes} \, f_{k}.
\end{eqnarray*}

We now proceed with Lemma \ref{lm:innerproductandcontraction} below.

\begin{lemma} \label{lm:innerproductandcontraction}
The following three statements hold.
\begin{enumerate}[(a)]
\item For $j,k \in \{1,2,3\}$, 
\begin{equation*}
\langle h_1, h_1  \rangle_{\mathcal{H}^{\otimes 2}} 
= \langle h_2, h_2  \rangle_{\mathcal{H}^{\otimes 2}} 
= 2 \langle h_3, h_3  \rangle_{\mathcal{H}^{\otimes 2}} 
=\sum_{k=1}^{n} \lambda_{k}^{2},
\end{equation*}
and
\begin{equation*}
\langle h_j, h_k  \rangle_{\mathcal{H}^{\otimes 2}} =0 \text{ if $j \neq k$.}
\end{equation*}
\item We have
\begin{equation*}
h_{1} \otimes_{1} h_{1} = h_{4}, \quad h_{2} \otimes_{1} h_{2} = h_{5}, \quad h_{3} \otimes_{1} h_{3} = \frac{1}{4} h_{4} + \frac{1}{4} h_{5},
\end{equation*}
\begin{equation*}
h_{1} \otimes_{1} h_{2} = h_{2} \otimes_{1} h_{1} =0,
\end{equation*}
\begin{equation*}
h_{1} \otimes_{1} h_{3} = h_{3} \otimes_{1} h_{2} = \frac{1}{2} \sum_{k=1}^{n} \lambda_{k}^{2} \, e_{k} \otimes f_{k},
\end{equation*}
\begin{equation*}
h_{3} \otimes_{1} h_{1} = h_{2} \otimes_{1} h_{3} = \frac{1}{2} \sum_{k=1}^{n} \lambda_{k}^{2} \, f_{k} \otimes e_{k}.
\end{equation*}
 \item For $j,k \in \{4,5,6\}$, 
\begin{equation*}
\langle h_4, h_4  \rangle_{\mathcal{H}^{\otimes 2}} 
= \langle h_5, h_5  \rangle_{\mathcal{H}^{\otimes 2}} 
= 2 \langle h_6, h_6  \rangle_{\mathcal{H}^{\otimes 2}} 
=\sum_{k=1}^{n} \lambda_{k}^{4},
\end{equation*}
and
\begin{equation*}
\langle h_j, h_k  \rangle_{\mathcal{H}^{\otimes 2}} =0 \text{ if $j \neq k$.}
\end{equation*}
\end{enumerate}
\end{lemma}

\begin{proof}
We use $\sigma_{jk}$ to denote the Kronecker delta, which is $1$ if $j=k$ and $0$ if $j \neq k$.\\

We first prove statement (a). By the bilinearity of the inner product, we have
\begin{eqnarray*}
&&\langle h_1, h_1  \rangle_{\mathcal{H}^{\otimes 2}} 
= \left\langle \sum_{j=1}^{n} \lambda_{j} \, e_{j} \otimes e_{j}, \sum_{k=1}^{n} \lambda_{k} \, e_{k} \otimes e_{k}  \right\rangle_{\mathcal{H}^{\otimes 2}}
= \sum_{j,k=1}^{n} \lambda_j \lambda_k \langle  e_{j} \otimes e_{j},  e_{k} \otimes e_{k} \rangle_{\mathcal{H}^{\otimes 2}} \\
&=& \sum_{j,k=1}^{n} \lambda_j \lambda_k \int_{\mathbb{R}_{+}^{2}} e_{j}(s) e_{j}(t) e_{k}(s) e_{k}(t) \, dsdt
= \sum_{j,k=1}^{n} \lambda_j \lambda_k \, \sigma_{jk} = \sum_{k=1}^{n} \lambda_{k}^{2}.
\end{eqnarray*}
Similarly, we have $ \langle h_2, h_2 \rangle_{\mathcal{H}^{\otimes 2}} = \sum_{k=1}^{n} \lambda_{k}^{2} $. Note that $e_{k} \, \widetilde{\otimes} \, f_{k} = \frac{1}{2} e_{k} \otimes f_{k} + \frac{1}{2} f_{k} \otimes e_{k}$. This gives
\begin{eqnarray*}
&& \left\langle e_{j} \, \widetilde{\otimes} \, f_{j}, e_{k} \, \widetilde{\otimes}\, f_{k}  \right\rangle_{\mathcal{H}^{\otimes 2}}
= \left\langle \frac{1}{2} e_{j} \otimes f_{j} + \frac{1}{2} f_{j} \otimes e_{j}, \frac{1}{2} e_{k} \otimes f_{k} + \frac{1}{2} f_{k} \otimes e_{k}  \right\rangle_{\mathcal{H}^{\otimes 2}}  \\
&=& \frac{1}{4} \langle e_{j} \otimes f_{j},  e_{k} \otimes f_{k} \rangle_{\mathcal{H}^{\otimes 2}}
 + \frac{1}{4}  \langle e_{j} \otimes f_{j},  f_{k} \otimes e_{k} \rangle_{\mathcal{H}^{\otimes 2}} \\
&& +  \,\frac{1}{4}  \langle f_{j} \otimes e_{j},  e_{k} \otimes f_{k} \rangle_{\mathcal{H}^{\otimes 2}}
 + \frac{1}{4}  \langle f_{j} \otimes e_{j},  f_{k} \otimes e_{k} \rangle_{\mathcal{H}^{\otimes 2}} \\
&=& \frac{1}{4} \sigma_{jk} + 0 + 0 + \frac{1}{4} \sigma_{jk} = \frac{1}{2} \sigma_{jk}.
\end{eqnarray*}
Then
\begin{eqnarray*}
&&\langle h_3, h_3  \rangle_{\mathcal{H}^{\otimes 2}} 
= \left\langle \sum_{j=1}^{n} \lambda_{j} \, e_{j} \, \widetilde{\otimes} \, f_{j}, \sum_{k=1}^{n} \lambda_{k} \, e_{k} \, \widetilde{\otimes} \, f_{k}  \right\rangle_{\mathcal{H}^{\otimes 2}} \\
&=& \sum_{j,k=1}^{n} \lambda_j \lambda_k \langle  e_{j} \, \widetilde{\otimes} \, f_{j},  e_{k} \, \widetilde{\otimes} \, f_{k} \rangle_{\mathcal{H}^{\otimes 2}}
= \sum_{j,k=1}^{n} \left(\lambda_j \lambda_k \times \frac{1}{2} \sigma_{jk}\right) = \frac{1}{2} \sum_{k=1}^{n} \lambda_{k}^{2}.
\end{eqnarray*}
Direct calculation yields that
\begin{equation*}
\left \langle e_{j} \otimes e_{j}, f_{k} \otimes f_{k} \right \rangle_{\mathcal{H}^{\otimes 2}}
= \left \langle e_{j} \otimes e_{j}, e_{k} \, \widetilde{\otimes} \, f_{k} \right \rangle_{\mathcal{H}^{\otimes 2}}
= \left \langle f_{j} \otimes f_{j}, e_{k} \, \widetilde{\otimes}  \, f_{k} \right \rangle_{\mathcal{H}^{\otimes 2}}
=0,
\end{equation*}
regardless of the values of $j$ and $k$. Then, by the bilinearity of the inner product, we deduce that
\begin{equation*}
\langle h_j, h_k  \rangle_{\mathcal{H}^{\otimes 2}} =0,
\end{equation*}
if $j \neq k$ and $j,k \in \{1,2,3\}$. This completes the proof of statement (a).

We now proceed to the proof of statement (b). For positive integers $j$ and $k$,
\begin{equation*}
\left( e_{j} \otimes e_{j} \right) \otimes_{1} \left( e_{k} \otimes e_{k} \right)
= e_{j} \otimes e_{k} \langle e_j, e_k\rangle_{\mathcal{H}}
= \sigma_{jk} \, e_{j} \otimes e_{k}.
\end{equation*}
Similarly, we have 
\begin{equation*}
\left( f_{j} \otimes f_{j} \right) \otimes_{1} \left( f_{k} \otimes f_{k} \right) = \sigma_{jk} \, f_{j} \otimes f_{k},
\end{equation*}
and
\begin{equation*}
\left( e_{j} \otimes e_{j} \right) \otimes_{1} \left( f_{k} \otimes f_{k} \right)
= \left( f_{j} \otimes f_{j} \right) \otimes_{1} \left( e_{k} \otimes e_{k} \right)
=0.
\end{equation*}
Noting that $e_{k} \, \widetilde{\otimes} \, f_{k} = \frac{1}{2} e_{k} \otimes f_{k} + \frac{1}{2} f_{k} \otimes e_{k}$, we have that
\begin{eqnarray*}
&& \left( e_{j} \otimes e_{j} \right) \otimes_{1} \left( e_{k} \, \widetilde{\otimes} \, f_{k} \right)
= \left( e_{j} \otimes e_{j} \right) \otimes_{1} \left( \frac{1}{2} e_{k} \otimes f_{k} +   \frac{1}{2} f_{k} \otimes e_{k} \right)  \\
&=& \frac{1}{2} \left( e_{j} \otimes e_{j} \right) \otimes_{1} \left( e_{k} \otimes f_{k} \right)
+ \frac{1}{2} \left( e_{j} \otimes e_{j} \right) \otimes_{1} \left( f_{k} \otimes e_{k} \right) \\
&=& \frac{1}{2} e_{j} \otimes e_{k} \langle e_{j}, f_{k} \rangle_{\mathcal{H}}
 + \frac{1}{2} e_{j} \otimes f_{k} \langle e_{j}, e_{k} \rangle_{\mathcal{H}}
 = 0 + \frac{1}{2} \sigma_{jk} \, e_{j} \otimes f_{k}
 = \frac{1}{2} \sigma_{jk}\, e_{j} \otimes f_{k}.
\end{eqnarray*} 
By a similar argument, we also have that
\begin{eqnarray*}
&& \left( e_{j} \, \widetilde{\otimes} \, f_{j} \right) \otimes_{1} \left( e_{k} \otimes e_{k} \right) = \frac{1}{2} \sigma_{jk} \, f_j \otimes e_k, \\
&& \left( f_{j} \otimes f_{j} \right) \otimes_{1} \left( e_{k} \, \widetilde{\otimes} \, f_{k} \right) = \frac{1}{2} \sigma_{jk} \, f_j \otimes e_k, \\
&& \left( e_{j} \, \widetilde{\otimes} \, f_{j} \right) \otimes_{1} \left( f_{k} \otimes f_{k} \right) = \frac{1}{2} \sigma_{jk} \, e_j \otimes f_k, \\
&& \left( e_{j} \, \widetilde{\otimes} \, f_{j} \right) \otimes_{1} \left( e_{k} \, \widetilde{\otimes} \, f_{k} \right) = \frac{1}{4} \sigma_{jk} \, e_j \otimes e_k + \frac{1}{4} \sigma_{jk} f_j \otimes f_k.
\end{eqnarray*}
Statement (b) now follows directly by the bilinearity of contraction. Since the proof of statement (c) is nearly identical to the proof of statement (a), we omit the details. This completes the proof.
\end{proof}

Given the definitions of $h_i$'s, we can now express $Z_{11}^{n}$, $Z_{12}^{n}$ and $Z_{22}^{n}$ in terms of multiple Wiener integrals. Since
\begin{eqnarray*}
&& \left( W_{1},  \dots, W_{n}, V_{1}, \dots, V_{n} \right) \\
&\overset{d}{=}& \left( \mathcal{I}_{1} (e_{1}),  \dots, \mathcal{I}_{1} (e_{n}), \mathcal{I}_{1} (r e_{1} + \sqrt{1- r^2} f_1),  \dots, \mathcal{I}_{1} (r e_{n} + \sqrt{1- r^2} f_n) \right),
\end{eqnarray*}
then
\begin{eqnarray*}
Z_{12}^{n} = \sum_{k=1}^{n} \lambda_{k} W_{k} V_{k}
\overset{d}{=}  \sum_{k=1}^{n} \lambda_{k} \, \mathcal{I}_{1} (e_{k}) \mathcal{I}_{1} (r e_{k} + \sqrt{1- r^2} f_k)
\end{eqnarray*}
Applying the product formula in \eqref{eq:productformulaforWienerintegral} yields
\begin{eqnarray*}
&& \mathcal{I}_{1} (e_{k}) \mathcal{I}_{1} (r e_{k} + \sqrt{1- r^2} f_k) \\
&=& \mathcal{I}_{2} \left( e_{k} \, \widetilde{\otimes} \left( r e_{k} + \sqrt{1- r^2} f_k \right) \right) + \langle e_{k} , r e_{k} + \sqrt{1- r^2} f_k \rangle_{\mathcal{H}} \\
&=& r \, \mathcal{I}_{2} ( e_{k} \otimes e_{k} ) + \sqrt{1- r^2} \,  \mathcal{I}_{2} (e_{k} \, \widetilde{\otimes} \, f_{k}) + r.
\end{eqnarray*}
Then,
\begin{eqnarray*}
Z_{12}^{n} &\overset{d}{=}& \sum_{k=1}^{n} \lambda_{k} \, \left( r \, \mathcal{I}_{2} ( e_{k} \otimes e_{k} ) + \sqrt{1- r^2} \,  \mathcal{I}_{2} (e_{k} \, \widetilde{\otimes} \, f_{k}) + r \right) \\
&=& r \, \mathcal{I}_{2}( h_1 ) + \sqrt{1- r^2} \,  \mathcal{I}_{2} (h_3) + r \sum_{k=1}^{n} \lambda_k.
\end{eqnarray*}
By a similar argument, we have
\begin{eqnarray*}
&& Z_{11}^{n} \overset{d}{=} \mathcal{I}_{2}(h_1) + \sum_{k=1}^{n} \lambda_{k}, \\
&& Z_{22}^{n} \overset{d}{=} r^2 \, \mathcal{I}_{2}(h_1) + (1-r^2) \, \mathcal{I}_{2}(h_2) + 2 r \sqrt{1-r^2} \, \mathcal{I}_{2}(h_3) + \sum_{k=1}^{n} \lambda_k.
\end{eqnarray*}
A routine calculation gives
\begin{eqnarray*}
&& \left( Z_{12}^{n} \right)^{2} - r^2 \, Z_{11}^{n} Z_{22}^{n} \\
&\overset{d}{=}&  (1-r^2) \left( r^2\, \mathcal{I}_{2}^{2}(h_1) +  \mathcal{I}_{2}^{2}(h_3) + 2r\sqrt{1- r^2} \, \mathcal{I}_{2}(h_1)  \mathcal{I}_{2}(h_3) - r^2  \, \mathcal{I}_{2}(h_1)  \mathcal{I}_{2}(h_2) \right)
 \\
&& + r(1- r^2) \left( \sum_{k=1}^{n} \lambda_k \right) \left( r \, \mathcal{I}_{2}(h_1) + 2 \sqrt{1- r^2} \, \mathcal{I}_{2}(h_3) - r \, \mathcal{I}_{2}(h_2) \right).
\end{eqnarray*}
Applying the product formula in \eqref{eq:productformulaforWienerintegral} with $p=q=2$, together with Lemma \ref{lm:innerproductandcontraction}, we have
\begin{eqnarray*}
&& \mathcal{I}_{2}^{2}(h_1) = \mathcal{I}_{4} ( h_1 \, \widetilde{\otimes} \, h_1 ) + 4 \, \mathcal{I}_{2} ( h_4) + 2 \sum_{k=1}^{n} \lambda_k^2, \\
&& \mathcal{I}_{2}^{2}(h_3) = \mathcal{I}_{4} ( h_3 \, \widetilde{\otimes} \, h_3 ) +  \mathcal{I}_{2} ( h_4 + h_5) +  \sum_{k=1}^{n} \lambda_k^2,  \\
&& \mathcal{I}_{2}(h_1) \mathcal{I}_{2}(h_3) = \mathcal{I}_{4}( h_1 \, \widetilde{\otimes} \, h_3 ) + 2 \mathcal{I}_{2} (h_6), \\ [2mm]
&& \mathcal{I}_{2}(h_1) \mathcal{I}_{2}(h_2) = \mathcal{I}_{4}( h_1 \, \widetilde{\otimes} \, h_2 ).
\end{eqnarray*}
Combining the last two displays yields
\begin{eqnarray}
&& \frac{1}{1-r^2} \left(\left( Z_{12}^{n} \right)^{2} - r^2 \, Z_{11}^{n} Z_{22}^{n}\right) \notag \\
&\overset{d}{=}&   \left( r^2\, \mathcal{I}_{4}( h_1 \, \widetilde{\otimes} \, h_1 ) +   \mathcal{I}_{4}( h_3 \, \widetilde{\otimes} \, h_3 ) + 2r\sqrt{1- r^2} \,  \mathcal{I}_{4}( h_1 \, \widetilde{\otimes} \, h_3 ) - r^2  \,  \mathcal{I}_{4}( h_1 \, \widetilde{\otimes} \, h_2 ) \right) \notag \\
&& + \, r \left( \sum_{k=1}^{n} \lambda_k \right) \left( r\, \mathcal{I}_{2}(h_1) + 2 \sqrt{1-r^2} \, \mathcal{I}_{2}(h_3) - r\, \mathcal{I}_{2}(h_2) \right) \notag \\
&& + \,  \left( (1+4r^2)\, \mathcal{I}_{2}(h_4) + \mathcal{I}_{2}(h_5) + 4 r \sqrt{1-r^2} \mathcal{I}_{2}(h_6) \right) \notag \\
&& + \, (1+2 r^2) \sum_{k=1}^{n} \lambda_k^2  \notag \\
&=& \mathcal{I}_{4} (w) + \mathcal{I}_{2}(v) + A_{1} + A_{2},  \label{eq:expressionfornumeratoralternative}
\end{eqnarray}
where
\begin{eqnarray*}
&& w := r^2 \, h_1 \, \widetilde{\otimes} \, h_1 + h_3 \, \widetilde{\otimes} \, h_3 + 2r\sqrt{1-r^2} \, h_1 \, \widetilde{\otimes} \, h_3 - r^2 \, h_1 \, \widetilde{\otimes} \, h_2 , \\
&& v := r \left( \sum_{k=1}^{n} \lambda_k \right) \left( r \, h_1 + 2 \sqrt{1-r^2} \,  h_3 - r\, h_2 \right), \\
&& A_{1} := (1+4r^2)\, \mathcal{I}_{2}(h_4) + \mathcal{I}_{2}(h_5) + 4 r \sqrt{1-r^2} \mathcal{I}_{2}(h_6) , \\
&& A_{2} := (1+2 r^2) \sum_{k=1}^{n} \lambda_k^2 .
\end{eqnarray*}

In what follows, we will investigate the asymptotics of $ \mathcal{I}_{4} (w) + \mathcal{I}_{2}(v) $, and then derive the asymptotics of $ \left(\left( Z_{12}^{n} \right)^2 - Z_{11}^{n} Z_{22}^{n}\right)/(1-r^2) $ by applying Lemma \ref{lm:kolmogorovratio}. As before, we use tools from Malliavin calculus and Stein's method to study the convergence rate of $ \mathcal{I}_{4} (w) + \mathcal{I}_{2}(v) $ (after scaling). We shall rely on Lemma \ref{lm:citefromDousissi} below.

\begin{lemma} \label{lm:citefromDousissi}
Let $F = \mathcal{I}_{2}(v) + \mathcal{I}_{4}(w)$, where $v \in \mathcal{H}^{\odot 2}$ and $w \in \mathcal{H}^{\odot 4}$. Then
\begin{eqnarray}
&& d_{Kol}\left( \frac{F}{\sqrt{EF^2}}, \, \mathcal{N}(0,1) \right) \notag  \\
&\leq& \frac{2}{EF^2} \bigg[ \sqrt{2} \,\lVert v \otimes_1 v \rVert_{\mathcal{H}^{\otimes2}} + 2 \sqrt{6!} \, \lVert w \otimes_1 w \rVert_{\mathcal{H}^{\otimes6}} + 18\sqrt{4!} \, \lVert w \otimes_2 w \rVert_{\mathcal{H}^{\otimes4}} \notag  \\
&& \quad \quad \quad +  \, 36\sqrt{2} \, \lVert w \otimes_3 w \rVert_{\mathcal{H}^{\otimes2}}  
 + 9\sqrt{2}\, \sqrt{\left\langle  v\otimes v, w \otimes_2 w \right\rangle_{\mathcal{H}^{\otimes4}}}  \notag   \\
&& \quad \quad \quad  + 3 \sqrt{4!} \, \sqrt{\lVert v \otimes_1 v \rVert_{\mathcal{H}^{\otimes2}} \lVert w \otimes_3 w \rVert_{\mathcal{H}^{\otimes2}}  } \bigg].  \label{eq:citefromDouissi}
\end{eqnarray}
Moreover, letting $R_{F}$ be the bracketed term on the right-hand side of \eqref{eq:citefromDouissi}, for any constant $\sigma>0$, we have
\begin{equation*}
d_{Kol}\left( \frac{F}{\sigma}, \, \mathcal{N}(0,1) \right) \leq \frac{2}{\sigma^2} R_{F} + \left| 1- \frac{EF^2}{\sigma^2} \right|.
\end{equation*}
\end{lemma}
\begin{proof}
The proof is similar to the proof of Theorem 8 in \cite{douissi2020ar}. The only difference is that we substitute the application of \eqref{eq:firstmalliavinboundfortotalvariation} with that of \eqref{eq:firstmalliavinboundforKolmogorov}. We therefore omit the details.
\end{proof}

We proceed by presenting two corollaries which follow from Lemma \ref{lm:citefromDousissi}.

\begin{corollary} \label{corochaos:generalboundforKol}
Let $F = \mathcal{I}_{4}(w)$ where $w \in \mathcal{H}^{\odot 4}$. For any constant $\sigma >0$,
\begin{eqnarray*}
d_{Kol} \left( \frac{F}{\sigma},  \mathcal{N} (0,1) \right)
&\leq& \frac{24}{\sigma^2} \Big[ 2 \sqrt{5} \,  \left\lVert w \otimes_1 w \right\rVert_{\mathcal{H}^{\otimes 6}} + 
3\sqrt{6} \,  \left\lVert w \otimes_2 w \right\rVert_{\mathcal{H}^{\otimes 4}}  \\
&& + 3\sqrt{2} \,  \left\lVert w \otimes_3 w \right\rVert_{\mathcal{H}^{\otimes 2}} \Big]
+ \left| 1- \frac{E F^2}{\sigma^2} \right|.
\end{eqnarray*}
\end{corollary}
\begin{proof}
This follows immediately by letting $v=0$.
\end{proof}

\begin{corollary} \label{coro:Kolmogorovbound}
Let $F = \mathcal{I}_{2}(v) + \mathcal{I}_{4}(w)$, where $v \in \mathcal{H}^{\odot 2}$ and $w \in \mathcal{H}^{\odot 4}$. For any constant $\sigma >0$, we have
\begin{eqnarray*}
&& d_{Kol}\left( \frac{F}{\sigma}, \, \mathcal{N}(0,1) \right) \\
&\leq& \frac{2}{\sigma^2} \bigg[ (\sqrt{2} + 3 \sqrt{6}) \, \lVert v \otimes_1 v \rVert_{\mathcal{H}^{\otimes2}}
+ ( 2 \sqrt{6!} + 39\sqrt{6} + 36 \sqrt{2} ) \, \lVert w \rVert_{\mathcal{H}^{\otimes4}}^{2} \\
&& \quad \quad \quad + 9\sqrt{2} \, \lVert v \rVert_{\mathcal{H}^{\otimes2}} \lVert w \rVert_{\mathcal{H}^{\otimes4}} \bigg] 
 +  \left| 1- \frac{EF^2}{\sigma^2} \right|.
\end{eqnarray*}
\end{corollary}

\begin{proof}
By the definition of contraction, we have
\begin{eqnarray*}
\left( w \otimes_1 w \right) (s_1, s_2, s_3, t_1, t_2, t_3)
= \int_{\mathbb{R}_{+}} w(s_1, s_2, s_3, u) w(t_1, t_2, t_3, u) \, du.
\end{eqnarray*}
Applying Cauchy–Schwarz yields
\begin{eqnarray*}
\left( w \otimes_1 w \right)^{2} (s_1, s_2, s_3, t_1, t_2, t_3)
\leq \int_{\mathbb{R}_{+}} w^{2}(s_1, s_2, s_3, u) \, du \int_{\mathbb{R}_{+}} w^{2}(t_1, t_2, t_3, u) \, du.
\end{eqnarray*}
We proceed to calculate
\begin{eqnarray*}
&& \lVert w \otimes_1 w \rVert_{\mathcal{H}^{\otimes6}}^{2}
= \int_{\mathbb{R}_{+}^{6}} \left( w \otimes_1 w \right)^{2} (s_1, s_2, s_3, t_1, t_2, t_3) \, ds_{1} \, ds_{2} \, ds_{3} \, dt_{1} \, dt_{2} \, dt_{3} \\
&\leq& \int_{\mathbb{R}_{+}^{6}} \left( \int_{\mathbb{R}_{+}} w^{2}(s_1, s_2, s_3, u) \, du \int_{\mathbb{R}_{+}} w^{2}(t_1, t_2, t_3, u) \, du \right) \, ds_{1} \, ds_{2} \, ds_{3} \, dt_{1} \, dt_{2} \, dt_{3} 
= \lVert w \rVert_{\mathcal{H}^{\otimes 4}}^{4},
\end{eqnarray*}
and so
\begin{equation}
\lVert w \otimes_1 w \rVert_{\mathcal{H}^{\otimes6}} \leq \lVert w \rVert_{\mathcal{H}^{\otimes 4}}^{2}. \label{eq:contraction1less}
\end{equation}
Similarly, we have
\begin{equation}
\lVert w \otimes_2 w \rVert_{\mathcal{H}^{\otimes4}} \leq \lVert w \rVert_{\mathcal{H}^{\otimes 4}}^{2}, \quad 
\lVert w \otimes_3 w \rVert_{\mathcal{H}^{\otimes2}} \leq \lVert w \rVert_{\mathcal{H}^{\otimes 4}}^{2}. \label{eq:contraction23less}
\end{equation}
Again by Cauchy–Schwarz, together with \eqref{eq:contraction23less}, we have
\begin{eqnarray}
 \sqrt{\left\langle  v\otimes v, w \otimes_2 w \right\rangle_{\mathcal{H}^{\otimes4}}}
&\leq& \sqrt{  \lVert v \otimes v \rVert_{\mathcal{H}^{\otimes 4}} \, \lVert w \otimes_2 w \rVert_{\mathcal{H}^{\otimes 4}} } \notag \\
&\leq& \sqrt{  \lVert v \rVert_{\mathcal{H}^{\otimes 2}}^{2}  \, \lVert w \rVert_{\mathcal{H}^{\otimes 4}}^{2} }
= \lVert v \rVert_{\mathcal{H}^{\otimes 2}}  \, \lVert w \rVert_{\mathcal{H}^{\otimes 4}}. \label{eq:squarerootinnerproductless}
\end{eqnarray}
By the inequality of arithmetic and geometric means, 
\begin{eqnarray}
\sqrt{\lVert v \otimes_1 v \rVert_{\mathcal{H}^{\otimes2}} \lVert w \otimes_3 w \rVert_{\mathcal{H}^{\otimes2}}  }
&\leq& \frac{1}{2} \lVert v \otimes_1 v \rVert_{\mathcal{H}^{\otimes 2}} + \frac{1}{2} \lVert w \otimes_{3} w \rVert_{\mathcal{H}^{\otimes 2}} \notag  \\
&\leq& \frac{1}{2} \lVert v \otimes_1 v \rVert_{\mathcal{H}^{\otimes 2}} + \frac{1}{2} \lVert w \rVert_{\mathcal{H}^{\otimes 4}}^{2}. \label{eq:squarerootproductless}
\end{eqnarray}
Applying Lemma \ref{lm:citefromDousissi}, and invoking \eqref{eq:contraction1less}-- \eqref{eq:squarerootproductless}, the desired result follows.
\end{proof}

To apply Corollary \ref{coro:Kolmogorovbound}, we need to estimate $ \lVert w \rVert_{\mathcal{H}^{\otimes 4}} $, $ \lVert v \otimes_1 v \rVert_{\mathcal{H}^{\otimes 2}} $ and $ \lVert v \rVert_{\mathcal{H}^{\otimes 2}}$. This is our next task. The following lemma will be helpful in estimating $ \lVert w \rVert_{\mathcal{H}^{\otimes 4}} $.

\begin{lemma} \label{lm:innerproductsymmetrization}
For $h_i, h_j, h_k, h_l \in \mathcal{H}^{\odot 2} $, we have
\begin{eqnarray*}
\left \langle  h_i \, \widetilde{\otimes} \, h_j ,  h_k \, \widetilde{\otimes} \, h_l \right \rangle_{\mathcal{H}^{\otimes 4}}
&=& \frac{1}{6} \left\langle  h_i  ,  h_k  \right \rangle_{\mathcal{H}^{\otimes 2}} \left\langle  h_j  ,  h_l  \right \rangle_{\mathcal{H}^{\otimes 2}}
 + \frac{1}{6} \left\langle  h_i  ,  h_l  \right \rangle_{\mathcal{H}^{\otimes 2}} \left\langle  h_j  ,  h_k  \right \rangle_{\mathcal{H}^{\otimes 2}} \\
&& + \,  \frac{2}{3} \left\langle  h_i \otimes_1 h_k  ,  h_l \otimes_1 h_j \right \rangle_{\mathcal{H}^{\otimes 2}}.
\end{eqnarray*} 
\end{lemma}

\begin{proof}
Note that $h_i \, \widetilde{\otimes} \, h_j$ is the symmetrization of $ h_i \otimes h_j $ and that $h_i$ and $h_j$ are symmetric. Thus,
\begin{eqnarray*}
&& h_i \, \widetilde{\otimes} \, h_j \left( x_1, x_2, x_3, x_4 \right)  \\
&=& \frac{1}{6} \, h_{i}( x_1 , x_2 ) h_{j}(x_3, x_4) + \frac{1}{6} \, h_{i}( x_1 , x_3 ) h_{j}(x_2, x_4)  + \frac{1}{6} \, h_{i}( x_1 , x_4 ) h_{j}(x_2, x_3) \\
&& + \, \frac{1}{6} \, h_{i}( x_2 , x_3 ) h_{j}(x_1, x_4) + \frac{1}{6} \, h_{i}( x_2 , x_4 ) h_{j}(x_1, x_3) + \frac{1}{6} \, h_{i}( x_3 , x_4 ) h_{j}(x_1, x_2).
\end{eqnarray*}
Similarly,
\begin{eqnarray*}
&& h_k \, \widetilde{\otimes} \, h_l \left( x_1, x_2, x_3, x_4 \right)  \\
&=& \frac{1}{6} \, h_{k}( x_1 , x_2 ) h_{l}(x_3, x_4) + \frac{1}{6} \, h_{k}( x_1 , x_3 ) h_{l}(x_2, x_4)  + \frac{1}{6} \, h_{k}( x_1 , x_4 ) h_{l}(x_2, x_3) \\
&& + \, \frac{1}{6} \, h_{k}( x_2 , x_3 ) h_{l}(x_1, x_4) + \frac{1}{6} \, h_{k}( x_2 , x_4 ) h_{l}(x_1, x_3) + \frac{1}{6} \, h_{k}( x_3 , x_4 ) h_{l}(x_1, x_2).
\end{eqnarray*}
The desired result then follows by routine calculation.
\end{proof}

With Lemma \ref{lm:innerproductsymmetrization} in hand, we are now ready to estimate $ \left \lVert  w \right \rVert_{\mathcal{H}^{\otimes 4}} $.

\begin{lemma}  \label{lm:estimatew}
We have
\begin{equation*}
\left \lVert w \right \rVert_{ \mathcal{H}^{\otimes 4} }^{2}
=  \frac{1}{12}\left( 1 + 4 r^2 + 2 r^4 \right) \left( \sum_{k=1}^{n} \lambda_{k}^{2} \right)^{2}
 + \frac{1}{12}\left( 1 + 4 r^2 \right) \sum_{k=1}^{n} \lambda_{k}^{4}.
\end{equation*}
Let
\begin{eqnarray*}
C_{14}(\alpha, r)= \frac{1}{12}\left(1 + 4 r^2 + 2 r^4 \right) \left( \frac{1+\alpha^2}{(1-\alpha^2)^3} + C_{2}(\alpha) \right)^{2} + \frac{1}{12} ( 1+ 4r^2) C_{3}(\alpha),
\end{eqnarray*}
with $C_{2}(\alpha)$ and $C_{3}(\alpha)$ defined in Lemma \ref{lm:lambda2} and Lemma \ref{lm:lambda4} respectively. Then
\begin{equation}
\left \lVert w \right \rVert_{ \mathcal{H}^{\otimes 4} }^{2} \leq C_{14}(\alpha, r) \times n^{-2}.  \label{eq:upperboundforw}
\end{equation}
\end{lemma}

\begin{proof}
Combining Lemma \ref{lm:innerproductandcontraction} and Lemma \ref{lm:innerproductsymmetrization} yields
\begin{eqnarray*}
&& \left \langle  h_1 \, \widetilde{\otimes} \, h_1 ,  h_1 \, \widetilde{\otimes} \, h_1 \right \rangle_{\mathcal{H}^{\otimes 4}}
= \frac{1}{3} \left\langle  h_1  ,  h_1  \right \rangle_{\mathcal{H}^{\otimes 2}}^{2} + \frac{2}{3} \left\langle  h_1 \otimes_1 h_1  ,  h_1 \otimes_1 h_1 \right \rangle_{\mathcal{H}^{\otimes 2}}  \\
&=& \frac{1}{3} \left\langle  h_1  ,  h_1  \right \rangle_{\mathcal{H}^{\otimes 2}}^{2} + \frac{2}{3} \left\langle  h_4  ,  h_4 \right \rangle_{\mathcal{H}^{\otimes 2}}
= \frac{1}{3} \left( \sum_{k=1}^{n} \lambda_{k}^{2} \right)^{2} + \frac{2}{3} \sum_{k=1}^{n} \lambda_{k}^{4}.
\end{eqnarray*}
Similarly, we have
\begin{eqnarray*}
&& \left \langle  h_3 \, \widetilde{\otimes} \, h_3 ,  h_3 \, \widetilde{\otimes} \, h_3 \right \rangle_{\mathcal{H}^{\otimes 4}}
= \frac{1}{12} \left( \sum_{k=1}^{n} \lambda_{k}^{2} \right)^{2} + \frac{1}{12} \sum_{k=1}^{n} \lambda_{k}^{4},  \\
&& \left \langle  h_1 \, \widetilde{\otimes} \, h_3 ,  h_1 \, \widetilde{\otimes} \, h_3 \right \rangle_{\mathcal{H}^{\otimes 4}}
= \frac{1}{12} \left( \sum_{k=1}^{n} \lambda_{k}^{2} \right)^{2} + \frac{1}{6} \sum_{k=1}^{n} \lambda_{k}^{4},  \\
&& \left \langle  h_1 \, \widetilde{\otimes} \, h_2 ,  h_1 \, \widetilde{\otimes} \, h_2 \right \rangle_{\mathcal{H}^{\otimes 4}}
= \frac{1}{6} \left( \sum_{k=1}^{n} \lambda_{k}^{2} \right)^{2}, \\
&& \left \langle  h_3 \, \widetilde{\otimes} \, h_3 ,  h_1 \, \widetilde{\otimes} \, h_2 \right \rangle_{\mathcal{H}^{\otimes 4}}
= \frac{1}{6} \sum_{k=1}^{n} \lambda_{k}^{4},
\end{eqnarray*}
and
\begin{eqnarray*}
&& \left \langle  h_1 \, \widetilde{\otimes} \, h_1 ,  h_3 \, \widetilde{\otimes} \, h_3 \right \rangle_{\mathcal{H}^{\otimes 4}}
= \left \langle  h_1 \, \widetilde{\otimes} \, h_1 ,  h_1 \, \widetilde{\otimes} \, h_3 \right \rangle_{\mathcal{H}^{\otimes 4}}
= \left \langle  h_1 \, \widetilde{\otimes} \, h_1 ,  h_1 \, \widetilde{\otimes} \, h_2 \right \rangle_{\mathcal{H}^{\otimes 4}}  \\ [2mm]
&=& \left \langle  h_3 \, \widetilde{\otimes} \, h_3 ,  h_1 \, \widetilde{\otimes} \, h_3 \right \rangle_{\mathcal{H}^{\otimes 4}}
= \left \langle  h_1 \, \widetilde{\otimes} \, h_3 ,  h_1 \, \widetilde{\otimes} \, h_2 \right \rangle_{\mathcal{H}^{\otimes 4}} = 0.
\end{eqnarray*}
We note that
\begin{equation*}
 w = r^2 \, h_1 \, \widetilde{\otimes} \, h_1 + h_3 \, \widetilde{\otimes} \, h_3 + 2r\sqrt{1-r^2} \, h_1 \, \widetilde{\otimes} \, h_3 - r^2 \, h_1 \, \widetilde{\otimes} \, h_2.
\end{equation*}
It then follows by routine calculation that
\begin{eqnarray*}
\left \lVert w \right \rVert_{ \mathcal{H}^{\otimes 4} }^{2}
=  \frac{1}{12}\left( 1 + 4 r^2 + 2 r^4 \right) \left( \sum_{k=1}^{n} \lambda_{k}^{2} \right)^{2}
 + \frac{1}{12}\left( 1 + 4 r^2 \right) \sum_{k=1}^{n} \lambda_{k}^{4}.
\end{eqnarray*}
By Lemma \ref{lm:lambda2} and Lemma \ref{lm:lambda4}, we have
\begin{eqnarray*}
&& \sum_{k=1}^{n} \lambda_{k}^{2} \leq \left( \frac{1+ \alpha^2}{(1- \alpha^2)^{3}} + C_{2}(\alpha) \right) \times n^{-1},  \\
&& \sum_{k=1}^{n} \lambda_{k}^{4} \leq C_{3}(\alpha) \times n^{-3} \leq C_{3}(\alpha) \times n^{-2}.
\end{eqnarray*}
The inequality in \eqref{eq:upperboundforw} then follows by combining the last two displays.
\end{proof}

We now proceed with Lemma \ref{lm:estimatev}.

\begin{lemma} \label{lm:estimatev}
We have
\begin{eqnarray}
&& \left \lVert v \right \rVert_{ \mathcal{H}^{\otimes 2} }^{2}
= 2 r^2 \left( \sum_{k=1}^{n} \lambda_{k} \right)^{2} \left( \sum_{k=1}^{n} \lambda_{k}^{2} \right), \label{eq:valueofv} \\
&& \left \lVert v \otimes_1 v \right \rVert_{ \mathcal{H}^{\otimes 2} }^{2}
= 2 r^4 \left( \sum_{k=1}^{n} \lambda_{k} \right)^{4} \left( \sum_{k=1}^{n} \lambda_{k}^{4} \right). \label{eq:valuevotimes1v}
\end{eqnarray}
Let
\begin{eqnarray*}
&& C_{15}(\alpha,r):= 2 r^2 \left( \frac{1}{1-\alpha^2} + C_{1}(\alpha) \right)^{2} \left( \frac{1+\alpha^2}{(1-\alpha^2)^3} + C_{2}(\alpha) \right), \\
&& C_{16}(\alpha, r):= 2 r^4 \left( \frac{1}{1-\alpha^2} + C_{1}(\alpha) \right)^{4} C_{3}(\alpha),
\end{eqnarray*}
with $C_{1}(\alpha)$, $C_{2}(\alpha)$ and $C_{3}(\alpha)$ defined in Lemma \ref{lm:lambda1}, Lemma \ref{lm:lambda2} and Lemma \ref{lm:lambda4} respectively. Then
\begin{eqnarray}
&& \left \lVert v \right \rVert_{ \mathcal{H}^{\otimes 2} }^{2} \leq C_{15}(\alpha,r) \times n^{-1}, \label{eq:upperboundforv} \\ [2mm]
&& \left \lVert v \otimes_1 v \right \rVert_{ \mathcal{H}^{\otimes 2} }^{2} \leq C_{16}(\alpha, r) \times n^{-3}. \label{eq:upperboundforvtimes1v}
\end{eqnarray}
\end{lemma}

\begin{proof}
Note that
\begin{equation*}
v = r \left( \sum_{k=1}^{n} \lambda_k \right) \left( r \, h_1 + 2 \sqrt{1-r^2} \,  h_3 - r\, h_2 \right).
\end{equation*}
Then \eqref{eq:valueofv} follows directly by statement (a) of Lemma \ref{lm:innerproductandcontraction}. Invoking statement (b) of Lemma \ref{lm:innerproductandcontraction}, we have
\begin{eqnarray*}
&& \left( r \, h_1 + 2 \sqrt{1-r^2} \,  h_3 - r\, h_2 \right) \otimes_1 \left( r \, h_1 + 2 \sqrt{1-r^2} \,  h_3 - r\, h_2 \right) \\ [2mm]
&=& r^{2} h_{1} \otimes_1 h_{1} + 2r \sqrt{1-r^2} h_{1} \otimes_1 h_{3} - r^2 h_{1} \otimes_1 h_{2}  \\ [2mm]
&& + \, 2r \sqrt{1-r^2} h_{3} \otimes_1 h_{1}  + 4(1-r^2) h_{3} \otimes_1 h_{3} - 2r\sqrt{1-r^2} h_{3} \otimes_1 h_{2} \\ [2mm]
&& - \, r^2 h_{2} \otimes_1 h_{1} - 2r\sqrt{1-r^2} h_{2} \otimes_1 h_{3} + r^2 h_{2} \otimes_1 h_{2}  \\ [2mm]
&=& h_{4} + h_{5}.
\end{eqnarray*}
Hence,
\begin{equation*}
v \otimes_1 v = r^2 \left( \sum_{k=1}^{n} \lambda_{k} \right)^{2} (h_4 + h_5),
\end{equation*}
and \eqref{eq:valuevotimes1v} follows directly by statement (c) of Lemma \ref{lm:innerproductandcontraction}.
By Lemma \ref{lm:lambda1}, Lemma \ref{lm:lambda2} and Lemma \ref{lm:lambda4}, we have
\begin{eqnarray*}
&& \left|  \sum_{k=1}^{n} \lambda_k \right| \leq \frac{1}{1-\alpha^2} + C_{1}(\alpha), \\
&& \sum_{k=1}^{n} \lambda_{k}^{2} \leq \left( \frac{1+ \alpha^2}{(1- \alpha^2)^{3}} + C_{2}(\alpha) \right) \times n^{-1},  \\
&& \sum_{k=1}^{n} \lambda_{k}^{4} \leq C_{3}(\alpha) \times n^{-3} .
\end{eqnarray*}
Together with \eqref{eq:valueofv} and \eqref{eq:valuevotimes1v}, \eqref{eq:upperboundforv} and \eqref{eq:upperboundforvtimes1v} follow.
\end{proof}

With Corollary \ref{coro:Kolmogorovbound}, Lemma \ref{lm:estimatew} and Lemma \ref{lm:estimatev} in hand, we are now ready to study the asymptotics of $\mathcal{I}_{2}(v) + \mathcal{I}_{4}(w)$.

\begin{theorem} \label{thm:2chaos+4chaosasymptotics}
Let
\begin{eqnarray*}
C_{17}(\alpha)&:=& 2 (1-\alpha^2) C_{1}(\alpha) + \frac{(1-\alpha^2)^3}{1+\alpha^2} C_{2}(\alpha) + 2 \frac{(1-\alpha^2)^{4}}{1+\alpha^2}  C_{1}(\alpha) C_{2}(\alpha) \\
&& + \, (1-\alpha^2)^2 C_{1}(\alpha)
+ \frac{(1-\alpha^2)^5}{1+\alpha^2} C_{1}(\alpha)^{2} C_{2}(\alpha), \\ 
C_{18}(\alpha, r) &:=& \left( \sqrt{2} + 3 \sqrt{6} \right) \sqrt{C_{16}(\alpha,r)} + \left( 2 \sqrt{6!} + 39\sqrt{6} + 36\sqrt{2} \right) C_{14}(\alpha,r) \\ [1mm]
&& + \, 9\sqrt{2} \sqrt{C_{14}(\alpha,r) C_{15}(\alpha,r)}, \\  [1mm]
C_{19}(\alpha, r) &:=& \frac{1}{2 r^2} \frac{(1-\alpha^2)^5}{1+\alpha^2}( C_{18}(\alpha, r) + 12 C_{14}(\alpha, r)  ) + C_{17}(\alpha).
\end{eqnarray*}
Under the alternative hypothesis $H_a$, we have
\begin{equation*}
d_{Kol}\left( \frac{1}{2r} \sqrt{\frac{(1-\alpha^2)^{5}}{1+\alpha^2}} \, \sqrt{n} \left( \mathcal{I}_{2}(v) + \mathcal{I}_{4}(w) \right), \, \mathcal{N}(0,1) \right) \leq \frac{C_{19}(\alpha, r)}{\sqrt{n}}.
\end{equation*}
\end{theorem}

\begin{proof}
Let $F = \mathcal{I}_{2}(v) + \mathcal{I}_{4}(w)$ and $ \sigma = 2r \sqrt{(1+\alpha^2)/(1-\alpha^2)^5/n} $. Then
\begin{eqnarray*}
&&\frac{2 \left \lVert v \right \rVert_{ \mathcal{H}^{\otimes 2} }^{2}}{\sigma^2} -1
= \frac{4 r^2 \left( \sum_{k=1}^{n} \lambda_{k} \right)^{2} \left( \sum_{k=1}^{n} \lambda_{k}^{2} \right)}{4 r^2 (1+\alpha^2)/(1-\alpha^2)^5/n } -1  \\
&=& \left( 1 + \frac{(1-\alpha^2) \kappa_{1}(n) }{n} \right)^{2}\left( 1+ \frac{(1-\alpha^2)^3}{1+\alpha^2} \, \frac{\kappa_{2}(n)}{n} \right) -1 \\
&=& \left( 2 (1-\alpha^2) \kappa_{1}(n) + \frac{(1-\alpha^2)^3}{1+\alpha^2} \kappa_{2}(n) \right) \times \frac{1}{n} \\
&& + \, \left( 2 \frac{(1-\alpha^2)^4}{1 + \alpha^2} \kappa_{1}(n) \kappa_{2}(n) + (1-\alpha^2)^2 \kappa_{1}(n) \right) \times \frac{1}{n^2} 
 + \frac{(1-\alpha^2)^5}{1+\alpha^2} \kappa_{1}^{2}(n) \kappa_{2}(n) \times \frac{1}{n^3},
\end{eqnarray*}
where the second equality follows by Lemma \ref{lm:lambda1} and Lemma \ref{lm:lambda2}. Note that $|\kappa_{1}(n) | \leq C_{1}(\alpha)$ and $| \kappa_{2}(n) |\leq C_{2}(\alpha)$. Hence, it follows immediately that
\begin{equation*}
\left| 1 - \frac{2 \left \lVert v \right \rVert_{ \mathcal{H}^{\otimes 2} }^{2}}{\sigma^2} \right| \leq \frac{C_{17}(\alpha)}{n}.
\end{equation*}
Invoking \eqref{eq:innerproductWeinerIntegral} yields
\begin{equation*}
EF^{2} = E\left[ \left( \mathcal{I}_{2}(v) + \mathcal{I}_{4}(w) \right)^2 \right] = 2 \left \lVert v \right \rVert_{ \mathcal{H}^{\otimes 2} }^{2} + 4! \left \lVert w \right \rVert_{ \mathcal{H}^{\otimes 4} }^{2}.
\end{equation*}
Hence,
\begin{eqnarray}
&& \left| 1 - \frac{E F^2}{\sigma^2} \right| = \left| 1 - \frac{ 2 \left \lVert v \right \rVert_{ \mathcal{H}^{\otimes 2} }^{2} + 4! \left \lVert w \right \rVert_{ \mathcal{H}^{\otimes 4} }^{2} }{\sigma^2} \right|
\leq \left|  1-  \frac{2 \left \lVert v \right \rVert_{ \mathcal{H}^{\otimes 2} }^{2}}{\sigma^2} \right| + \frac{4! \left \lVert w \right \rVert_{ \mathcal{H}^{\otimes 4} }^{2}}{\sigma^2} \notag  \\
&\leq& \frac{C_{17}(\alpha)}{n} + \frac{4! \left \lVert w \right \rVert_{ \mathcal{H}^{\otimes 4} }^{2}}{\sigma^2}
\leq  \frac{C_{17}(\alpha)}{n}  + \frac{6}{r^2} \frac{(1-\alpha^2)^5}{1+\alpha^2} C_{14}(\alpha,r) \times \frac{1}{n}, \label{eq:alternativeupperboundtools1}
\end{eqnarray}
where the last inequality follows by Lemma \ref{lm:estimatew} and plugging in the value of $\sigma$. Again, invoking Lemma \ref{lm:estimatew} and Lemma \ref{lm:estimatev}, after arrangement, we have
\begin{eqnarray}
&& (\sqrt{2} + 3 \sqrt{6}) \, \lVert v \otimes_1 v \rVert_{\mathcal{H}^{\otimes2}}
+ ( 2 \sqrt{6!} + 39\sqrt{6} + 36 \sqrt{2} ) \, \lVert w \rVert_{\mathcal{H}^{\otimes4}}^{2}  + 9\sqrt{2} \, \lVert v \rVert_{\mathcal{H}^{\otimes2}} \lVert w \rVert_{\mathcal{H}^{\otimes4}} \notag \\
&\leq& \frac{C_{18}(\alpha, r)}{n\sqrt{n}}. \label{eq:alternativeupperboundtools2}
\end{eqnarray}
Applying Corollary \ref{coro:Kolmogorovbound}, together with \eqref{eq:alternativeupperboundtools1} and \eqref{eq:alternativeupperboundtools2}, the desired result follows easily.
\end{proof}

With the above preparations in hand, we now return to the proof of Theorem \ref{thm:asymptoticsnumeratoralternative}. For purposes of simplicity, let
\begin{eqnarray*}
 C_{20}(\alpha, r) &=& \frac{1+ 2 r^2 }{2 |r|} \sqrt{\frac{(1-\alpha^2)^5}{1+\alpha^2}} \left( \frac{1+\alpha^2}{(1-\alpha^2)^3}  + C_{2}(\alpha)\right), \\ [2mm]
\sigma &=& 2 r \sqrt{\frac{(1+\alpha^2)}{(1-\alpha^2)^5}} \frac{1}{\sqrt{n}}.
\end{eqnarray*}
From \eqref{eq:expressionfornumeratoralternative}, we have
\begin{eqnarray*}
&& \frac{1}{2 r (1 -r^2)} \sqrt{\frac{(1- \alpha^2)^5}{1 + \alpha^2}} \, \sqrt{n} \left( \left(Z_{11}^{n}\right)^2 - r^2 \, Z_{11}^{n} Z_{22}^{n} \right) \\
&\overset{d}{=}& \frac{\mathcal{I}_{2}(v) + \mathcal{I}_{4}(w)}{\sigma} + \frac{A_1}{\sigma} + \frac{A_2}{\sigma}.
\end{eqnarray*}
Since Kolmogorov distance is at the level of distribution, we need only prove that
\begin{equation*}
d_{Kol}\left(  \frac{\mathcal{I}_{2}(v) + \mathcal{I}_{4}(w)}{\sigma} + \frac{A_1}{\sigma} + \frac{A_2}{\sigma}, \, \mathcal{N}(0,1) \right) \leq \frac{C_{11}(\alpha,r)}{\sqrt{n}}.
\end{equation*}
Invoking Lemma \ref{lm:kolmogorovratio} with $X$, $Z$ and $\epsilon$ replaced by $ \left( \mathcal{I}_{2}(v) + \mathcal{I}_{4}(w) \right)/\sigma $, $ A_{1}/\sigma + A_{2}/\sigma $ and $  \left( 1+ C_{20}(\alpha, r) \right)/ \sqrt{n} $, together with Theorem \ref{thm:2chaos+4chaosasymptotics}, it suffices to prove that
\begin{equation}
P\left( \left| \frac{A_1}{\sigma} + \frac{A_2}{\sigma} \right| > \frac{1+ C_{20}(\alpha, r) }{\sqrt{n}} \right) \leq \frac{1+8 r^2 + 4 r^4}{r^2} \, \frac{(1-\alpha^2)^5}{1+ \alpha^2} \, \frac{C_{3}(\alpha)}{\sqrt{n}}. \label{eqa:addingtermboundsum}
\end{equation}
Note that
\begin{equation*}
A_{1} = (1+4r^2)\, \mathcal{I}_{2}(h_4) + \mathcal{I}_{2}(h_5) + 4 r \sqrt{1-r^2} \mathcal{I}_{2}(h_6).
\end{equation*}
We now apply \eqref{eq:innerproductWeinerIntegral} and invoke statement (c) of Lemma \ref{lm:innerproductandcontraction}. Routine calculations yield
\begin{equation*}
E\left[ A_{1}^{2} \right] = 4 ( 1+ 8 r^2 + 4 r^4 ) \sum_{k=1}^{n} \lambda_{k}^{4}.
\end{equation*}
Plugging in the value of $\sigma$, we have
\begin{eqnarray*}
E\left[ \frac{A_{1}^{2}}{\sigma^2} \right]
= \frac{1+8 r^2 + 4 r^4}{r^2} \, \frac{(1-\alpha^2)^5}{1+ \alpha^2} \, n \sum_{k=1}^{n} \lambda_{k}^{4}
\leq \frac{1+8 r^2 + 4 r^4}{r^2} \, \frac{(1-\alpha^2)^5}{1+ \alpha^2}\, \frac{C_{3}(\alpha)}{n^2},
\end{eqnarray*}
where the last inequality follows by Lemma \ref{lm:lambda4}. By Chebyshev's Inequality, we have
\begin{eqnarray}
P\left( \left|  \frac{A_{1}}{\sigma} \right| > \frac{1}{\sqrt{n}} \right)
\leq \frac{E\left[ \frac{A_{1}^{2}}{\sigma^2} \right]}{\frac{1}{n}}
\leq \frac{1+8 r^2 + 4 r^4}{r^2} \, \frac{(1-\alpha^2)^5}{1+ \alpha^2}\, \frac{C_{3}(\alpha)}{n}. \label{eqa:addingtermbound1}
\end{eqnarray}
Note that
\begin{equation*}
A_{2} = (1 + 2r^2) \sum_{k=1}^{n} \lambda_{k}^2.
\end{equation*}
By Lemma \ref{lm:lambda2}, we have
\begin{eqnarray*}
\left|  \frac{A_2}{\sigma} \right|
= \frac{1+ 2 r^2}{2|r|} \sqrt{\frac{(1-\alpha^2)^5}{1+\alpha^2}} \left( \frac{1+\alpha^2}{(1- \alpha^2)^3 \sqrt{n} } + \frac{\kappa_{2}(n)}{n^{3/2}} \right)
\leq \frac{C_{20}(\alpha,r)}{\sqrt{n}}.
\end{eqnarray*}
Hence
\begin{equation}
P\left( \left|  \frac{A_2}{\sigma} \right| > \frac{C_{20}(\alpha,r)}{\sqrt{n}} \right) =0. \label{eqa:addingtermbound2}
\end{equation}
Combining \eqref{eqa:addingtermbound1} and \eqref{eqa:addingtermbound2}, we have
\begin{eqnarray*}
&& P\left( \left| \frac{A_1}{\sigma} + \frac{A_2}{\sigma} \right| > \frac{1+ C_{20}(\alpha, r) }{\sqrt{n}} \right)  \\
&\leq& P\left( \left|  \frac{A_{1}}{\sigma} \right| > \frac{1}{\sqrt{n}} \right) + P\left( \left|  \frac{A_2}{\sigma} \right| > \frac{C_{20}(\alpha,r)}{\sqrt{n}} \right) \\  [2mm]
&\leq& \frac{1+8 r^2 + 4 r^4}{r^2} \, \frac{(1-\alpha^2)^5}{1+ \alpha^2}\, \frac{C_{3}(\alpha)}{n}.
\end{eqnarray*}
which in turn implies \eqref{eqa:addingtermboundsum}. This completes the proof.

\subsection{Proof of Theorem \ref{thm:denominatortailboundalternative}}
We now turn to the proof of Theorem \ref{thm:denominatortailboundalternative}. Note that
\begin{eqnarray*}
M_{1} = \frac{1+ \alpha^2}{1- \alpha^2} + 1, \quad M_{2} = \frac{1+r^2}{2r^2} \, \frac{1+ \alpha^2}{1- \alpha^2} + 1.
\end{eqnarray*}

We first claim that, under the alternative hypothesis $H_a$, we still have (for $n$ sufficiently large)
\begin{eqnarray}
&& P\left( \left|  \sqrt{(1-\alpha^2)Z_{11}^{n} \times (1-\alpha^2) Z_{22}^{n}  } -1 \right| > 3 \, M_1 \sqrt{\ln n / n} \right) \notag \\
&\leq& P\left( \left|  (1-\alpha^2)Z_{11}^{n} \times (1-\alpha^2) Z_{22}^{n}   -1 \right| > 3 \, M_1 \sqrt{\ln n / n} \right) \notag \\
&\leq& \frac{C_{8}(\alpha)}{\sqrt{n}}. \label{eq:denominatorcontrol}
\end{eqnarray}
We note that the proof of \eqref{eq:denominatorcontrol} is exactly the same as the proof under the null hypothesis $H_0$, and that this is contained in the proof of Theorem \ref{thm:kolmogorov}. We therefore omit the details of the proof of \eqref{eq:denominatorcontrol}.

We next claim under that alternative hypothesis $H_a$, for some constant $C_{21}(\alpha,r)$ depending on $\alpha$ and $r$ and for $n$ sufficiently large,
\begin{equation}
P\left( \left| \left(1-\alpha^2\right)Z_{12}^{n}/r -1 \right| > M_{2} \sqrt{\frac{\ln n }{n}} \right) \leq \frac{C_{21}(\alpha,r)}{\sqrt{n}}, \label{eq:tailestimateforZ12n}
\end{equation}
We begin our proof of this claim by noting that
\begin{eqnarray*}
&& P\left( \left| \left(1-\alpha^2\right)Z_{12}^{n}/r -1 \right| > M_{2} \sqrt{\frac{\ln n }{n}} \right) \\ [1mm]
&\leq&  P\left( \left(1-\alpha^2\right) Z_{12}^{n}/r > 1 + M_{2} \sqrt{\ln n/n} \right)  
 +  P\left( -\left(1-\alpha^2\right) Z_{12}^{n}/r > -1 + M_{2} \sqrt{\ln n/n} \right).
\end{eqnarray*}
Applying Markov's inequality yields
\begin{eqnarray}
&& P\left( \left(1-\alpha^2\right) Z_{12}^{n}/r > 1 + M_{2} \sqrt{\ln n/n} \right)  \notag  \\
&=& P\left( \sqrt{n \ln n} \left( 1- \alpha^2 \right) Z_{12}^{n}/r > \sqrt{n \ln n} + M_{2} \ln n  \right) \notag \\ [1mm]
&\leq& \frac{E\left[ e^{ \sqrt{n \ln n} \left( 1- \alpha^2 \right) Z_{12}^{n}/r } \right]}{e^{ \sqrt{n \ln n} + M_{2} \ln n }}. \label{eqa:markovinequalityalternative}
\end{eqnarray}
It follows by Remark \ref{remark:boundlambda} that $$\left| (1 \pm r) \sqrt{n \ln n} (1- \alpha^2) \,\lambda_k /r \right| \leq 2 C_{3}(\alpha)^{1/4} \, (1-\alpha^2) \, n^{-1/4} \, \sqrt{\ln n}/ |r| < 1/2$$ for $n$ sufficiently large. Since under alternative hypothesis $H_a$, $ (W_k, V_k) $ are i.i.d. Gaussian random vectors  with mean zero and covariance matrix $$\begin{pmatrix}
1 & r \\
r & 1
\end{pmatrix},$$ we have
\begin{eqnarray*}
&& E\left[ e^{ \sqrt{n \ln n} \left( 1- \alpha^2 \right) Z_{12}^{n}/r } \right]
= E\left[   e^{ \sqrt{n \ln n} \left( 1- \alpha^2 \right) \sum_{k} \lambda_k W_k V_k /r} \right]
= \prod_{k=1}^{n} E \left[ e^{ \sqrt{n \ln n} \left( 1- \alpha^2 \right) \lambda_k W_k V_k /r} \right]  \\
&=& \prod_{k=1}^{n} \left( 1+ \frac{1-r}{r} \left( 1- \alpha^2 \right) \sqrt{n \ln n} \,  \lambda_k \right)^{-1/2} \left( 1 - \frac{1+r}{r} \left( 1- \alpha^2 \right) \sqrt{n \ln n} \,  \lambda_k \right)^{-1/2}  \\
&=& d_{n}\left( - \frac{1-r}{r}  \left( 1- \alpha^2 \right) \sqrt{n \ln n} \right)^{-1/2}  d_{n}\left(  \frac{1+r}{r}  \left( 1- \alpha^2 \right) \sqrt{n \ln n} \right)^{-1/2},
\end{eqnarray*}
where the third equality follows by a standard expression for the mgf of a linear-quadratic functional of a Gaussian random vector and the last equality follows by the representation \eqref{eq:anotherformofacp} of $d_{n}(\lambda)$. Applying Lemma \ref{lm:asymptoticdn} twice with $t$ replaced by $-(1-r)(1-\alpha^2)/r$ and $(1+r)(1-\alpha^2)/r$, respectively, we obtain
\begin{eqnarray*}
&& d_{n} \left( - \frac{1-r}{r} \left( 1-\alpha^2 \right) \sqrt{n \ln n}  \right)
= \left(1 + o(1) \right) e^{ \frac{1-r}{r}  \sqrt{n \ln n} - \left( \frac{(1-r)^2}{ 2r^2 } \frac{1+\alpha^2}{1- \alpha^2 } + o(1) \right) \ln n  }, \\
&&  d_{n} \left( \frac{1+r}{r} \left( 1-\alpha^2 \right) \sqrt{n \ln n}  \right)
= \left(1 + o(1) \right) e^{ - \frac{1+r}{r}  \sqrt{n \ln n} - \left( \frac{(1+r)^2}{ 2r^2 } \frac{1+\alpha^2}{1- \alpha^2 } + o(1) \right) \ln n  }.
\end{eqnarray*}
Combining the last two displays yields
\begin{eqnarray*}
E\left[ e^{ \sqrt{n \ln n} \left( 1- \alpha^2 \right) Z_{12}^{n}/r } \right]
= \left( 1+ o(1) \right) e^{ \sqrt{n \ln n} + \left( \frac{1+r^2}{2 r^2} \frac{1+\alpha^2}{1- \alpha^2} +o(1) \right) \ln n }.
\end{eqnarray*}
Together with \eqref{eqa:markovinequalityalternative} and the fact that $M_2 = \frac{1+r^2}{2r^2} \, \frac{1+ \alpha^2}{1- \alpha^2} + 1 $, we have, for $n$ sufficiently large and for some constant $C_{22}(\alpha, r)$ depending on $\alpha$ and $r$,
\begin{eqnarray*}
&& P\left( \left(1-\alpha^2\right) Z_{12}^{n}/r > 1 + M_{2} \sqrt{\ln n/n} \right) \\ [1mm]
&\leq& \left(1 + o(1) \right) e^{ \left( \frac{1+r^2}{2 r^2} \frac{1+\alpha^2}{1- \alpha^2} + o(1) \right) \ln n - M_{2} \ln n }  \\ [1mm]
&=& \left( 1 + o(1) \right) e^{ -(1 + o(1)) \ln n  } \\ [1mm]
&\leq& \frac{C_{22}( \alpha, r )}{\sqrt{n}},
\end{eqnarray*}
 Similarly, for $n$ sufficiently large, there exists a constant $C_{23}(\alpha, r)$ such that
\begin{equation*}
P\left( -\left(1-\alpha^2\right) Z_{12}^{n}/r > -1 + M_{2} \sqrt{\ln n/n} \right)
\leq \frac{C_{23}(\alpha, r)}{\sqrt{n}}.
\end{equation*}
Then \eqref{eq:tailestimateforZ12n} follows immediately by considering the last two displays. We now prove that, for $n$ sufficiently large,
\begin{equation*}
P\left( \left|  \frac{Z_{12}^{n} \sqrt{Z_{11}^{n} Z_{22}^{n} } }{r / \left( 1- \alpha^2 \right)^2}  -1  \right| > \left( 3M_1 + M_2 + 3 M_1 M_2 \right) \sqrt{ \frac{\ln n}{n} } \right) \leq \frac{2 \left( C_{8}( \alpha) + C_{21}( \alpha, r ) \right) }{\sqrt{n}}.
\end{equation*}
Note that
\begin{equation*}
\frac{Z_{12}^{n} \sqrt{Z_{11}^{n} Z_{22}^{n} } }{r / \left( 1- \alpha^2 \right)^2}
 = \left( \left( 1- \alpha^2 \right) Z_{12}^{n}/r \right) \times \sqrt{\left(1-\alpha^2 \right) Z_{11}^{n} \times \left( 1- \alpha^2\right) Z_{22}^{n} }.
\end{equation*}
Then
\begin{eqnarray*}
  \frac{Z_{12}^{n} \sqrt{Z_{11}^{n} Z_{22}^{n} } }{r / \left( 1- \alpha^2 \right)^2}  -1  
 &=& \left( \left( 1- \alpha^2 \right) Z_{12}^{n}/r  -1 \right) \times \left(  \sqrt{\left(1-\alpha^2 \right) Z_{11}^{n} \times \left( 1- \alpha^2\right) Z_{22}^{n} } -1 \right) \\
 && + \, \left( \left( 1- \alpha^2 \right) Z_{12}^{n}/r  -1 \right)
 + \left(  \sqrt{\left(1-\alpha^2 \right) Z_{11}^{n} \times \left( 1- \alpha^2\right) Z_{22}^{n} } -1 \right).
\end{eqnarray*}
Hence,
\begin{eqnarray}
&& P\left( \left|  \frac{Z_{12}^{n} \sqrt{Z_{11}^{n} Z_{22}^{n} } }{r / \left( 1- \alpha^2 \right)^2}  -1  \right| > \left( 3M_1 + M_2 + 3 M_1 M_2 \right) \sqrt{ \frac{\ln n}{n} } \right)  \notag \\
&\leq& P\left( \left| \left( \left( 1- \alpha^2 \right) Z_{12}^{n}/r  -1 \right) \times \left(  \sqrt{\left(1-\alpha^2 \right) Z_{11}^{n} \times \left( 1- \alpha^2\right) Z_{22}^{n} } -1 \right) \right|  > 3 M_{1} M_{2} \sqrt{\frac{\ln n }{n}} \right)  \notag \\
&& + \, P\left( \left| \left( 1- \alpha^2 \right) Z_{12}^{n}/r  -1 \right| > M_{2} \sqrt{\frac{\ln n}{n}} \right) \notag \\
&& + \,  P \left( \left| \sqrt{\left(1-\alpha^2 \right) Z_{11}^{n} \times \left( 1- \alpha^2\right) Z_{22}^{n} } -1  \right| > 3 M_{1} \sqrt{\frac{\ln n}{n}} \right) \notag \\
&\leq& P\left(  \left| \left( 1- \alpha^2 \right) Z_{12}^{n}/r  -1 \right| > M_{2} \sqrt[4]{\frac{\ln n }{n}} \right) +  P\left( \left| \left( 1- \alpha^2 \right) Z_{12}^{n}/r  -1 \right| > M_{2} \sqrt{\frac{\ln n}{n}} \right)  \notag \\
&& + \, P \left( \left| \sqrt{\left(1-\alpha^2 \right) Z_{11}^{n} \times \left( 1- \alpha^2\right) Z_{22}^{n} } -1  \right| > 3 M_{1} \sqrt[4]{\frac{\ln n}{n}} \right) \notag \\
&& + \, P \left( \left| \sqrt{\left(1-\alpha^2 \right) Z_{11}^{n} \times \left( 1- \alpha^2\right) Z_{22}^{n} } -1  \right| > 3 M_{1} \sqrt{\frac{\ln n}{n}} \right)  \notag \\
&\leq& 2 P\left( \left| \left( 1- \alpha^2 \right) Z_{12}^{n}/r  -1 \right| > M_{2} \sqrt{\frac{\ln n}{n}} \right) \notag \\
&& + \,2 P \left( \left| \sqrt{\left(1-\alpha^2 \right) Z_{11}^{n} \times \left( 1- \alpha^2\right) Z_{22}^{n} } -1  \right| > 3 M_{1} \sqrt{\frac{\ln n}{n}} \right)  \notag \\
&\leq& \frac{2 \left( C_{8}( \alpha) + C_{21}( \alpha, r ) \right) }{\sqrt{n}}, \label{eqa:tailofproductoftwo}
\end{eqnarray}
where the second to last inequality follows by the fact that $\ln n/n \leq 1$ and the last inequality follows by combining \eqref{eq:denominatorcontrol} and \eqref{eq:tailestimateforZ12n}.

With the above preparation in hand, we now return to the proof of Theorem \ref{thm:denominatortailboundalternative}. Note that
\begin{equation*}
 \left|  \frac{ Z_{12}^{n} \sqrt{ Z_{11}^{n} Z_{22}^{n} } + r Z_{11}^{n} Z_{22}^{n} }{2 r / (1-\alpha^2)^2 }  -1 \right|
 \leq \frac{1}{2} \left|   \frac{Z_{12}^{n} \sqrt{Z_{11}^{n} Z_{22}^{n} } }{r / \left( 1- \alpha^2 \right)^2}  -1  \right|
 + \frac{1}{2} \left|  \left( 1- \alpha^2 \right) Z_{11}^{n} \times \left(1- \alpha^2 \right) Z_{22}^{n} -1 \right|.
\end{equation*}
Together with \eqref{eq:denominatorcontrol} and \eqref{eqa:tailofproductoftwo}, we have
\begin{eqnarray*}
&& P\left( \left|  \frac{ Z_{12}^{n} \sqrt{ Z_{11}^{n} Z_{22}^{n} } + r Z_{11}^{n} Z_{22}^{n} }{2 r / (1-\alpha^2)^2 }  -1 \right| 
 > \frac{1}{2} (6 M_1 + M_2 + 3M_1 M_2) \sqrt{\frac{\ln n }{n}} \right)  \\
&\leq& P\left( \left|  \frac{Z_{12}^{n} \sqrt{Z_{11}^{n} Z_{22}^{n} } }{r / \left( 1- \alpha^2 \right)^2}  -1  \right| > \left( 3M_1 + M_2 + 3 M_1 M_2 \right) \sqrt{ \frac{\ln n}{n} } \right)  \\
&& + \, P\left( \left| \left( 1-\alpha^2 \right) Z_{11}^{n} \times \left( 1- \alpha^2 \right) Z_{22}^{n} -1  \right| > 3 M_{1} \sqrt{\frac{\ln n}{n}} \right) \\
&\leq& \frac{3 C_{8}(\alpha) + 2 C_{21}(\alpha, r)}{\sqrt{n}}.
\end{eqnarray*}
The proof is now completed by letting $C_{12}(\alpha,r) = 3 C_{8}(\alpha) + 2 C_{21}(\alpha, r)$.

\end{document}